\documentclass{amsart}

\usepackage[a4paper]{geometry}
\usepackage{latexsym, amsthm, amsfonts, amsmath,amssymb,graphicx,lmodern,mathrsfs}

\usepackage[latin1]{inputenc} %%% Pour les accents
\usepackage[T1]{fontenc} %%% pour la césure des mots accentués
\usepackage[english]{babel} %%% toujours en dernier

\newtheorem{theo}{Theorem}[section]
\newtheorem{lem}[theo]{Lemma}
\newtheorem{propo}[theo]{Proposition}

\theoremstyle{definition}
\newtheorem{defi}[theo]{Definition}

\theoremstyle{remark}
\newtheorem{rem}[theo]{Remark}
\newtheorem{ex}[theo]{Example}

\makeatletter

\@addtoreset{equation}{section}
\makeatother
\def\R{\mathbb{R}}
\def\Z{\mathbb{Z}}
\def\C{\mathbb{C}}
\def\N{\mathbb{N}}
\def\Q{\mathbb{Q}}

\def\z{\zeta}
\def\n{\eta}
\def\r{\rho}
\def\s{\sigma}
\def\a{\alpha}
\def\e{\varepsilon}
\def\d{\delta}

\def\b{\beta}

\def\n'{\nu}
\def\d{\delta}
\def\l{\lambda}

\def\g{\gamma}

\def\k{\kappa}

\def\D{\Delta}
\def\G{\Gamma}
\def\L{\Lambda}
\def\T{\Theta}
\def\P{\Phi}

\def\dq {\delta_{q}}
\def\sq {\sigma_{q}}

\begin{document}
\sloppy
\title{Confluence of meromorphic solutions of~$q$-difference equations.}
\author{Thomas Dreyfus}
\address{Université Paris Diderot - Institut de Mathématiques de Jussieu,}
\curraddr{4, place Jussieu 75005 Paris.}
\email{tdreyfus@math.jussieu.fr.}
\thanks{Work partially supported by ANR, contract ANR-06-JCJC-0028 and ECOS Nord France-Colombia No~C12M01.}


%\keywords{u}

\date{\today}

%\dedicatory{u}

\begin{abstract}
In this paper, we consider a~$q$-analogue of the Borel-Laplace summation where~$q>1$ is a real parameter. In particular, we show that the Borel-Laplace summation of a divergent power series solution of a linear differential equation can be uniformly approximated on a convenient sector, by a meromorphic solution of a corresponding family of linear~$q$-difference equations. We perform the computations for the basic hypergeometric series. Following Sauloy, we prove how a basis of solutions of a linear differential equation can be uniformly approximated on a convenient domain by a basis of solutions of a corresponding family of linear~$q$-difference equations. This leads us to the approximations of Stokes matrices and monodromy matrices of the linear differential equation by matrices with entries that are invariants by the multiplication by~$q$.
\end{abstract} 

\maketitle

\tableofcontents
\pagebreak[4]
\section*{Introduction}

When~$q$ tends to~$1$, the~$q$-difference operator~$d_{q}:=f\mapsto \frac{f(qz)-f(z)}{(q-1)z}$  ``tends'' to the usual derivation. Hence every differential equation may be discretized by a~$q$-difference equation. Given a linear differential equation~$\widetilde{\D}$ and a family of linear~$q$-difference equations~$\D_{q}$ that discretize~$\widetilde{\D}$, we wonder if there exists a basis of solutions of~$\D_{q}$, that converges as~$q$ goes to~$1$ to a given basis of solutions of~$\widetilde{\D}$. This question has been studied in the Fuchsian case (see \cite{S00}) and the main goal of this paper is to consider the general situation. The problem is that for non-Fuchsian linear differential equations, the fundamental solution, i.e., the invertible solution matrix, given by the Hukuhara-Turrittin theorem involves divergent formal power series. However, we may apply to them a Borel-Laplace summation process in order to obtain a fundamental solution that is analytic on a convenient sector. To extend the work of Sauloy to the non-Fuchsian case, we have to approximate the Borel-Laplace summation of a given formal power series solution of a linear differential equation, by a~$q$-analogue of the Borel-Laplace summation applied to a formal power series solution of a corresponding family of linear~$q$-difference equations. Our main result,  Theorem~\ref{4theo1}, gives a confluence\footnote{Throughout the paper, we will use the word ``confluence'' to describe the $q$-degeneracy when $q\rightarrow 1$.} result of this nature.  Then, we use our main result to prove that under convenient assumptions, a basis of meromorphic solutions of a linear differential equation, not necessarily Fuchsian, can be uniformly approximated on a convenient domain by a basis of solutions of a corresponding family of linear~$q$-difference equations. This leads us to the approximations of Stokes matrices and monodromy matrices of the linear differential equation by matrices with entries that are invariants by the multiplication by~$q$. We also perform the computations for the basic hypergeometric series.\\
\begin{center}
$\ast\ast\ast$
\end{center}
\par
Let~$q>1$ be a real parameter, and let us define the dilatation operator~$\sq$ 
$$\sq \big(f(z)\big):=f(qz).$$
See Remark \ref{4rem1} for the reason why we consider $q$ real, and not $q$ complex number such that $|q|>1$, like others papers present in the literature. We define~$\dq:=\frac{\sq-\mathrm{Id}}{q-1}$, which converges formally to~$\d:=z\frac{d}{dz}$ when~$q\to 1$.
Let us consider \\
$$\left\{\begin{array}{lll}
\dq Y(z,q)&=&B(z)Y(z,q)\\\\
\d \widetilde{Y}(z)&=&B(z)\widetilde{Y}(z),
\end{array}\right.
$$
where~$B(z)\in \mathrm{M}_{m}\Big(\C(z)\Big)$, that is a $m$ by $m$ square matrix with coefficients in~$\C(z)$. We are going to recall the main result of \cite{S00} in the particular case where the above matrix $B(z)$ does not depend upon $q$ and $q>1$ is real. Notice that a part of what follows now is purely local at $z=0$, which means that we could consider systems that have coefficients in the field of germs of meromorphic functions in the neighborhood of~$z=0$, but for the simplicity of exposition, we have assumed that the coefficients are rational. In \cite{S00}, Sauloy assumes that the systems are Fuchsian at~$0$ and the linear differential system has exponents at~$0$ which are non resonant (see \cite{S00},~$\S 1$, for a precise definition). 
The Frobenius algorithm provides a local fundamental solution at~$z=0$,~$\widetilde{\P}_{0}(z)$, of the linear differential system~${\d \widetilde{Y}(z)=B(z)\widetilde{Y}(z)}$. This solution can be analytically continued into an analytic solution on~$\C^{*}$, minus a finite number of lines and half lines of the form~${\R_{>0}\a :=\Big\{ x\a \Big| x\in ]0,\infty[ \Big\}}$ and~${\R_{\geq 1}\b :=\Big\{ x\b \Big| x\in [1,\infty[ \Big\}}$, with~$\a,\b\in \C^{*}$. Notice that in Sauloy's paper, the lines and half lines are in fact respectively $q$-spirals and $q$-half-spirals since the author considers the case where $q$ is a complex number such that $|q|>1$.\par
 In \cite{S00},~$\S 1$, the author uses a~$q$-analogue of the 
Frobenius algorithm to construct a local fundamental matrix solution at~$z=0$,~$\P_{0}(z,q)$, of the family of linear~$q$-difference systems ${\dq Y(z,q)=B(z)Y(z,q)}$, which is for a fixed~$q$, meromorphic on~$\C^{*}$
and has its poles contained in a finite number of~$q$-spirals of the form~${ q^{\Z}\a:=\left\{q^{n}\a , n\in \Z\right\}}$ and~${q^{\N^{*}}\b :=\left\{  q^{n}\b, n\in \N^{*} \right\}}$, with~$\a,\b\in \C^{*}$.  Sauloy proves that~$\P_{0}(z,q)$ converges uniformly to~$\widetilde{\P}_{0}(z)$ when~$q\to 1$, in every compact subset of its domain of definition. \par 
Let us assume that the systems are Fuchsian at~$\infty$ and the linear differential system has exponents at~$\infty$ which are non resonant. Let us consider~$\P_{\infty}(z,q)$ and~$\widetilde{\P}_{\infty}(z)$, the corresponding fundamental solutions at infinity of the linear~$\d$ and~$\dq$-systems. 
 Sauloy shows that the Birkhoff connection matrix~${P(z,q):=\Big(\P_{\infty}(z,q)\Big)^{-1}\P_{0}(z,q)}$, which is invariant under the action of~$\sq$, converges to~$\widetilde{P}(z):=\left(\widetilde{\P}_{\infty}(z)\right)^{-1}\widetilde{\P}_{0}(z)$ when~$q\to 1$. 
The matrix~$\widetilde{P}(z)$ is locally constant and the monodromy matrices at the intermediates singularities (those different from $0$ and $\infty$) of the linear differential system can be expressed with the values of~$\widetilde{P}(z)$. \\ \par 

The goal of this paper is to prove similar results in the non-Fuchsian case. The question implies difficulties of very different nature than in the Fuchsian case, since divergent formal power series may appear as solutions. The prototypical example is the  Euler equation and one possible~$q$-deformation:\\
$$\left \{ \begin{array}{lllll} 
z\dq y(z,q)&+&y(z,q)&=&z\\\\
z\d \widetilde{y}(z)&+&\widetilde{y}(z)&=&z,
\end{array}\right.
$$
which admits respectively the formal divergent solutions: 
$$\displaystyle \sum_{n=0}^{\infty}(-1)^{n}[n]_{q}^{!}z^{n+1},\hbox{ and } \displaystyle \sum_{n=0}^{\infty}(-1)^{n}n!z^{n+1},$$
where~${[n]_{q}^{!}:=\prod_{l=0}^{n} [l]_{q}}$, ${[l]_{q}:=\left(1+...+q^{l-1}\right)}$ if $l\in \N^{*}$, and $[0]_{q}:=1$. In this example, the first formal power series converges coefficientwise to the second when~$q\to 1$. However, there exist also analytic solutions of the linear differential equation. For example, if $d\not \equiv \pi[2\pi]$ the following functions are solutions: 
 $$\int_{0}^{\infty e^{id}} \frac{e^{-\z/z}}{1+\z}d\z.$$ 
 More generally, given a formal power series solution of a linear differential equation in coefficients that are germs of meromorphic functions, it is well known (see~$\S \ref{4sec1}$) that we may apply to it several Borel and Laplace transformations to obtain a germ of analytic solution on a sector of the form $$\overline{S}(a,b):=\left\{z\in \widetilde{\C} \Big|  \arg(z)\in ]a,b[\right\},$$ where~$\widetilde{\C}$ denotes the Riemann surface of the logarithm.\par 
The situation is similar in the $q$-difference case. Consider a linear $q$-difference system with coefficients that are germs of meromorphic functions, and assume that the slopes belongs to $\Z$ (see \cite{RSZ} for the definition). Like in the differential case, formal power series appear as solutions. The authors of \cite{RSZ} show how to transform a formal fundamental solution into fundamental solutions which entries are meromorphic on a punctured neighborhood of $0$ in $\C^{*}$. Then, it is shown how the meromorphic fundamental solutions are linked with the local meromorphic classification of $q$-difference equations.  It is natural to study the behavior, as $q$ goes to $1$ of their meromorphic fundamental solutions. Unfortunately, there are two difficulties for this approach:\\
\begin{itemize}
\item In \cite{RSZ} it is used the Birkhoff-Guenter normal form which has no known analogous in the differential case. Study the behavior of the normal form as $q$ goes to~$1$ seems to be very complicated.\\
\item  Although there are several $q$-analogues of the Borel and Laplace transformations, see \cite{DVZ,MZ,R92,RZ,Z99,Z00,Z01,Z02,Z03}, we do not know how to express the meromorphic fundamental solutions using a $q$-analogue of the Borel-Laplace summation.\\
\end{itemize}
\begin{center}
$\ast\ast\ast$
\end{center} \par 
 Let us state now our main result, Theorem \ref{4theo1}, in a particular case.  Let~$z\mapsto\hat{h}(z,q),\widetilde{h}$ be formal power series solutions of\\ 
$$
\left \{ \begin{array}{llll} 
b_{m}(z)\d_{q}^{m}\hat{h}(z,q)&+\dots+b_{0}(z)\hat{h}(z,q)&=&0\\\\
b_{m}(z)\d^{m}\widetilde{h}(z)&+\dots+b_{0}(z)\widetilde{h}(z)&=&0,
\end{array}\right.$$
where~$b_{0},\dots,b_{m}\in \C[z]$. We assume that~$\hat{h}$ converges coefficientwise to~$\widetilde{h}$ when~$q\to 1$. We prove that for~$q>1$ sufficiently close to~$1$, we may apply to~$\hat{h}$ several~$q$-analogues of the Borel and Laplace transformation and obtain~$S_{q}\left(\hat{h}\right)$, solution of the family of linear {$q$-difference} equations that is for~$q$ fixed meromorphic on~$\C^{*}$. Moreover,~$S_{q}\left(\hat{h}\right)$ converges uniformly on a convenient domain to the Borel-Laplace summation of~$\widetilde{h}$ when~$q\to 1$. Notice that although this theorem deal with a problem which is purely local at $z=0$, we have assumed that the equations have coefficients in $\C[z]$, instead of the ring of germs of analytic functions, since we need this assumption to prove the theorem. Another result of same nature can be found in \cite{DVZ}, Theorem~2.6. See Remark~\ref{4rem8} for the comparison of the setting of this result and our theorem.\par 
In the appendix, we introduce another~$q$-Laplace transformation and prove an analogous result for the associated~$q$-Borel-Laplace summation. See Theorem~\ref{4theo9}. \\ \par
In~$\S \ref{4sec7}$, we consider the basic hypergeometric series~$_{r}\varphi_{s}$. Let us choose~${r,s\in \N}$ with~${r>s+1}$,~$\a_{1},\dots,\a_{r},\b_{1},\dots,\b_{s}\in \C\setminus \left(-\N\right)$  with different images in~$\C/\Z$, let~${\underline{p}:=q^{-1/(r-s-1)}}$,  and consider, see \cite{GR},
$$\begin{array}{l}
 _{r}\varphi_{s}\left(\begin{array}{ll}
\underline{p}^{\a_{1}},\dots,\underline{p}^{\a_{r}}&\\
&;\underline{p},\left(1-\underline{p}\right)^{1+s-r}z\\
\underline{p}^{\b_{1}},\dots,\underline{p}^{\b_{s}}&
\end{array} \right)\\\\
:=\displaystyle \sum_{n=0}^{\infty}\dfrac{(\underline{p}^{\a_{1}};\underline{p})_{n}\dots(\underline{p}^{\a_{r}};\underline{p})_{n}\left(1-\underline{p}\right)^{(1+s-r)n}}
{(\underline{p};\underline{p})_{n}(\underline{p}^{\b_{1}};\underline{p})_{n}\dots(\underline{p}^{\b_{s}};\underline{p})_{n}}
\underline{p}^{-n(n-1)/2}(-1)^{n(1+s-r)}z^{n},
\end{array}
$$
where~$(a;\underline{p})_{n+1}:=(1-a\underline{p}^{n})(a;\underline{p})_{n}$ and~$(a;\underline{p})_{0}:=1$, for~$a\in \C$. The above series converge coefficientwise when $q\to 1$ to $$ 
_{r}F_{s}\left(
\begin{array}{ll}
\a_{1},\dots,\a_{r}&\\
&;(-1)^{1+s-r}z\\
\b_{1},\dots,\b_{s}&
\end{array} 
\right):=\displaystyle \sum_{n=0}^{\infty}\dfrac{(\a_{1})_{n}\dots(\a_{r})_{n}}{n!(\b_{1})_{n}\dots(\b_{s})_{n}}(-1)^{n(1+s-r)}z^{n}$$ 
where,~$(\a)_{n+1}:=(\a+n)(\a)_{n}$ and~$(\a)_{0}:=1$ for~$\a\in \C^{*}$.
We prove that the series $_{r}\varphi_{s}$ and $_{r}F_{s}$ do not satisfy the assumptions of our main result, Theorem \ref{4theo1}. However, we perform explicitly the computation of a~$q$-Borel-Laplace summation of $_{r}\varphi_{s}$, using others $q$-analogues of the Borel and Laplace transformations, and prove the convergence when $q\to 1$ to the classical Borel-Laplace summation of $_{r}F_{s}$.
See Theorem~\ref{4theo8}. See also \cite{Z02},~$\S 2$, for the case~$r=2,s=0$. \\ \par 

In~$\S\ref{4sec8}$, we apply our main result to prove that we can uniformly approximate on a convenient domain a basis of solutions of a linear differential equation by a basis of solutions of a corresponding family of linear~$q$-difference equations. Our theorem holds in the non-Fuchsian case but does not recover Sauloy's result in the Fuchsian case. In other words, the two results are complementary.\par 
In~$\S \ref{4sec82}$, we are interested in the case where the linear~$\dq$ and ~$\d$-equations have formal coefficients  and we want to prove the convergence, in a sense we specify later, of a basis of formal solutions of a family of linear~$\dq$-equations, to the Hukuhara-Turrittin solution of a linear~$\d$-equation. 
A problem is the size of the field of constants. A fundamental solution of a linear differential system is defined modulo an invertible matrix with complex entries, while a fundamental solution of a linear $q$-difference system is defined modulo a matrix with entries in~$\mathcal{M}_{\mathbb{E}}$, the field of functions invariant under the action of~$\sq$, i.e., the field of meromorphic functions over the torus $\C^{*}\setminus q^{\Z}$. This field can be identified with the field of elliptic functions. The consequence of this is that we have to choose very carefully our basis of solutions of the family of linear~$\dq$-equations in order to have the convergence. For example, if we consider\\
$$\left \{ \begin{array}{lll} 
\dq y(z,q)&=&(z^{-1}+1)y(z,q)\\\\
\d \widetilde{y}(z)&=&(z^{-1}+1)\widetilde{y}(z),
\end{array}\right.
$$
the solutions of the linear~$\d$-equation are of the form~$\widetilde{y}(z)=a\left(e^{-z^{-1}}+z\right)$ with~$a\in \C$. 
Let us introduce the Jacobi theta function
$$\T_{q} (z):=\displaystyle \sum_{n \in \Z} q^{\frac{-n(n+1)}{2}}z^{n}=\displaystyle \prod_{n=0}^{\infty}\left(1-q^{-n-1}\right)\left(1+q^{-n-1}z\right)\left(1+q^{-n}z^{-1}\right),$$ 
  which is analytic on~$\C^{*}$, vanishes on the discrete~$q$-spiral~$-q^{\Z}$, with simple zeros, and satisfies: $$\s_{q}\T_{q}(z)=z\T_{q}(z)=\T_{q}\left(z^{-1}\right).$$
The following function is solution of the~$\dq$-equation
${y(z,q)=\dfrac{1}{\T_{q}(z)}\displaystyle\sum_{n=0}^{\infty}\frac{ q^{n}z^{n}}{\prod_{k=0}^{n}(q^{k}-q+1)}}$, but the behavior as~$q$ goes to~$1$ is unclear. If we want to construct a solution of the family of linear~$\dq$-equations that converges to a solution of the linear~$\d$-equation, 
we need to introduce the~$q$-exponential:
$$e_{q}(z):=\displaystyle \sum_{n=0}^{\infty}\dfrac{z^{n}}{[n]_{q}^{!}}=\displaystyle \prod_{n=0}^{\infty}(1+(q-1)q^{-n-1}z).$$
 It is analytic on~$\C$, with simple zeros on the discrete~$q$-spiral~$\frac{q^{\N^{*}}}{1-q}$ and satisfies~${\dq e_{q}(z)=ze_{q}(z)}$.
The function,~$e_{q}\left(qz^{-1}\right)^{-1}+z$ is solution of the family of linear~$\dq$-equations and converges uniformly on the compacts of~$\C^{*}$ to~$e^{-z^{-1}}+z$ when~$q\to 1$.
More generally, we will multiply a fundamental solution of the family of linear~$\dq$-equations by a convenient matrix with entries in~$\mathcal{M}_{\mathbb{E}}$, in order to have a confluence result. See Theorem~\ref{4theo2} for a precise statement. \par
In~$\S \ref{4sec83}$, we are interested in the case where the linear~$\dq$ and ~$\d$-equations have coefficients in~$\C(z)$. We combine our main result, Theorem \ref{4theo1}, and what we have just mentioned above, to prove that under reasonable assumptions, we have the uniform convergence on a convenient domain of a basis of solutions of a family of linear~$\dq$-equations to a basis of solutions of the corresponding linear~$\d$-equation when~$q\to 1$. This leads us to the convergence of the~$q$-Stokes matrices, that do not correspond to the~$q$-Stokes matrices present in \cite{RSZ}, to the Stokes matrices. See Theorem~\ref{4theo4}. \par
In~$\S \ref{4sec84}$, following \cite{S00}, we construct a locally constant matrix, and his values allow us to obtain the monodromy matrices at the intermediate singularities of the linear differential system. This result is an analogue of \cite{S00},~$\S 4$, in the irregular singular case. See Theorem~\ref{4theo3}. The results of~$\S \ref{4sec83}$ and $\S\ref{4sec84}$ could be the first step to a numerical algorithm of approximation of the Stokes and monodromy matrices. See \cite{FRT,FRRT,VdH,LRR,Re} for results  of numerical approximation of the Stokes matrices and \cite{MS10,Mez} for results of numerical approximation of the monodromy matrices.\\
\begin{center}
$\ast\ast\ast$
\end{center} \par 
The paper is organized as follows. In~$\S\ref{4sec1}$, we make a short overview of the Stokes phenomenon of the linear differential equations. In particular, we recall the definition of the Stokes matrices.
In~$\S\ref{4sec2}$, we recall some results that can be found in \cite{RSZ} on the local formal study of linear~$q$-difference equations.  In~$\S \ref{4sec3}$, we introduce the~$q$-Borel and the~$q$-Laplace transformations.  \par
 The~$\S\ref{4sec4}$, is devoted to the statement of our main result, Theorem~\ref{4theo1}, while~$\S \ref{4sec5}$ and~$\S \ref{4sec6}$ are devoted to the proof of Theorem~\ref{4theo1}. In~$\S \ref{4sec5}$, we prove a proposition that deals with the confluence of meromorphic solutions.
In~$\S \ref{4sec61}$, we study the confluence of the~$q$-Laplace transformation. In~$\S \ref{4sec62}$, we show Theorem~\ref{4theo1} in a particular case, and in~$\S \ref{4sec63}$, we prove Theorem~\ref{4theo1} in the general case. \par
As told above, in~$\S \ref{4sec7}$, we study basic hypergeometric series, and in~$\S \ref{4sec8}$, we apply our main result to obtain the uniform convergence on a convenient domain of a basis of solutions of a family of linear~$\dq$-equations to a basis of solutions of the corresponding linear~$\d$-equation when~$q\to 1$.\\ \par

\textbf{Acknowledgments.}
This paper was prepared during my thesis, supported by the region Ile de France. I want to thank my advisor, Lucia Di Vizio, for the interesting discussions we had during the redaction of the paper. I also want to thank Jean-Pierre Ramis, Jacques Sauloy and Changgui Zhang for accepting to answer to my numerous questions about their work. Lastly, I heartily thank the anonymous referee who spent a great lot of time and effort to help me make the present paper more readable.
\pagebreak[3]
\section{Local analytic study of linear differential equations}\label{4sec1}

In this section, we make a short overview of the Stokes phenomenon of linear differential equations. See \cite{B,VdPS} for more details. See also \cite{Ber,LR90,LR95,M95,MR,R93,RM1,S09}.\par 
Let~$\C[[z]]$ be the ring of formal power series and~$\C((z)):=\C[[z]][z^{-1}]$ be its fraction field. Let~$K$ be an intermediate differential field extension:~$\C(z)\subset K\subset\displaystyle\bigcup_{\nu\in \N^{*}}\C\left(\left(z^{1/\nu}\right)\right)$. We recall that~$\d=z\frac{d}{dz}$.  Let us consider the linear differential operator with coefficients in~$K$
$$\widetilde{P}=\widetilde{b}_{m}\d^{m}+\widetilde{b}_{m-1}\d^{m-1}+\dots+\widetilde{b}_{0}.$$
The Newton polygon of~$\widetilde{P}$ is the convex hull of
$$\displaystyle\bigcup_{k=0}^{m}\Big\{ (i,j)\in \N^{*}\times\Q \Big|i\leq k, j\geq v_{0}\left(\widetilde{b}_{k}\right)\Big\},$$ where~$v_{0}$ denotes the~$z$-adic valuation of~$K$.
Let~$\big\{(d_{1},n_{1}),\dots ,(d_{r},n_{r})\big\}$
be a minimal subset of~$\mathbb Z^2$ for the inclusion, with~$d_{1}<\dots<d_{r}$,
such that the Newton polygon is the convex hull of
$$\displaystyle\bigcup_{k=0}^{r}\Big\{ (i,j)\in \N^{*}\times\Q \Big|i\leq d_{k}, j\geq n_{k}\Big\}.$$
 We call slopes of the linear~$\d$-equation the positive rational numbers~${\frac{n_{i+1}-n_{i}}{d_{i+1}-d_{i}}}$, and multiplicity of the slope~$\frac{n_{i+1}-n_{i}}{d_{i+1}-d_{i}}$, the integer~$d_{i+1}-d_{i}$. \par 
Let $\widetilde{b}_{0},\dots,\widetilde{b}_{m-1}\in K$ and~$\widetilde{B}:=\begin{pmatrix}
0 &1&\dots&0\\
\vdots&\ddots&\ddots&\vdots\\
\vdots&\ddots&\ddots&1\\
-\widetilde{b}_{0}&\dots&\dots&-\widetilde{b}_{m-1}
\end{pmatrix}\in \mathrm{M}_{m}(K)$ be a companion matrix. The linear differential system~$\d\widetilde{Y}=\widetilde{B}\widetilde{Y}$ is equivalent to the linear differential equation~$\d^{m}\widetilde{y}+\widetilde{b}_{m-1}\d^{m-1}\widetilde{y}+\dots+\widetilde{b}_{0}\widetilde{y}=0$. Let~${\widetilde{P}:=\d^{m}+\widetilde{b}_{m-1}\d^{m-1}+\dots+\widetilde{b}_{0}}$. We define the Newton polygon of~$\d \widetilde{Y}=\widetilde{B}\widetilde{Y}$, as the Newton polygon of~$\widetilde{P}$. We also define the slopes and the multiplicities of the slopes of~$\d \widetilde{Y}=\widetilde{B}\widetilde{Y}$ as the slopes and the multiplicities of the slopes of~$\widetilde{P}$. Notice that if $\widetilde{B}\in \mathrm{M}_{m}\Big(\C((z))\Big)$ is not a companion matrix, we can still define the Newton polygon of $\d \widetilde{Y}=\widetilde{B}\widetilde{Y}$, but we will not need this in this paper.\\\par
The linear differential equations~$\d \widetilde{Y}=\widetilde{A}\widetilde{Y}$ and~$\d \widetilde{Y}=\widetilde{B}\widetilde{Y}$, with~$\widetilde{A},\widetilde{B}\in \mathrm{M}_{m}(K)$  are said to be equivalent over~$K$ if there exists~$\widetilde{H}\in \mathrm{GL}_{m}(K)$, that is an invertible matrix with coefficients in~$K$, such that $$\widetilde{A}=\widetilde{H}\left[ \widetilde{B}\right]_{\displaystyle\d}:=\widetilde{H}\widetilde{B}\widetilde{H}^{-1}+\d \widetilde{H}\widetilde{H}^{-1}.$$
 Notice that in this case: $$\d \widetilde{Y}=\widetilde{B}\widetilde{Y}\Longleftrightarrow \d \left(\widetilde{H}\widetilde{Y}\right)=\widetilde{A}\widetilde{H}\widetilde{Y}.$$ 
Conversely, if there exist $\widetilde{A},\widetilde{B}\in \mathrm{M}_{m}(K)$  and~$\widetilde{H}\in \mathrm{GL}_{m}(K)$, such that
$\d \widetilde{Y}=\widetilde{B}\widetilde{Y}$, $ \d \widetilde{Z}=\widetilde{A}\widetilde{Z}$ and $\widetilde{Z}=\widetilde{H}\widetilde{Y}$, then 
$$\widetilde{A}=\widetilde{H}\left[ \widetilde{B}\right]_{\displaystyle\d}.$$
\par 
One can prove that if the above matrices~$\widetilde{A},\widetilde{B}\in \mathrm{M}_{m}(K)$ are companion matrices, then they have the same Newton polygon.\par
Let us consider~$\d \widetilde{Y}=\widetilde{B}\widetilde{Y}$, where~$\widetilde{B}\in \mathrm{M}_{m}\Big(\C((z))\Big)$ is a companion matrix, having slopes~$k_{1}<\dots<k_{r-1}$ with multiplicity~$m_{1},\dots,m_{r-1}$, and let~$\nu\in \N^{*}$ be minimal such that all the~$\nu k_{i}$ belongs to~$\N$. The Hukuhara-Turrittin theorem (see Theorem 3.1 in \cite{VdPS} for a statement that is trivially equivalent to the following) says that there exist 
\begin{itemize}
\item~$\widetilde{H}\in \mathrm{GL}_{m}\Big(\C\left(\left(z^{1/\nu}\right)\right)\Big)$,
\item~$\widetilde{L}_{i}\in \mathrm{M}_{m_{i}}(\C)$,
\item~$\widetilde{\l}_{i}\in  z^{-1/\nu}\C\left[z^{-1/\nu}\right]$,
\end{itemize}
such that $\widetilde{B}=\widetilde{H}\left[\mathrm{Diag}_{i}  \left(\widetilde{L}_{i}+\d\widetilde{\l}_{i}\times\mathrm{Id}_{m_{i}}\right)\right]_{\displaystyle\d}$,  where 
$$ \mathrm{Diag}_{i} \left(\widetilde{L}_{i}+\d\widetilde{\l}_{i}\times\mathrm{Id}_{m_{i}}\right):=\begin{pmatrix}
\widetilde{L}_{1}+\d\widetilde{\l}_{1}\times\mathrm{Id}_{m_{1}}&&\\
&\ddots&\\
&&\widetilde{L}_{k}+\d\widetilde{\l}_{k}\times\mathrm{Id}_{m_{k}} 
\end{pmatrix}\footnote{If no confusions is likely to arise we will write~$\mathrm{Diag} \left(\widetilde{L}_{i}+\d\widetilde{\l}_{i}\times\mathrm{Id}_{m_{i}}\right)$ instead of~$\mathrm{Diag}_{i} \left(\widetilde{L}_{i}+\d\widetilde{\l}_{i}\times\mathrm{Id}_{m_{i}}\right)$. Notice that altough the index $i$ seems here to be useless, he will be later helpfull when we will consider diagonal bloc matrices with diagonal bloc having several indexes.}.$$ 
Roughly speaking, this means that if~$\widetilde{B}\in \mathrm{M}_{m}\Big(\C((z))\Big)$ is a companion matrix, there exists a formal fundamental solution of~$\d \widetilde{Y}=\widetilde{B}\widetilde{Y}$, of the form 
$$\widetilde{H}(z)\mathrm{Diag}\left(z^{\widetilde{L}_{i}}e^{\widetilde{\l}_{i}(z)\times\mathrm{Id}_{m_{i}}}\right).$$ 
 Of course, written like this, this statement is not rigorous, since matrices~$\widetilde{H}(z)$ and~$\mathrm{Diag}\left(z^{\widetilde{L}_{i}}e^{\widetilde{\l}_{i}(z)\times\mathrm{Id}_{m_{i}}}\right)$ can not be multiplied.\par 
 Remark that for all~$n\in \Z$, we have also
$$ \widetilde{B}=\left(z^{n}\widetilde{H} \right)\left[\mathrm{Diag} \left(\widetilde{L}_{i}-n\times\mathrm{Id}+\d\widetilde{\l}_{i}\times\mathrm{Id}_{m_{i}}\right)\right]_{\displaystyle\d},$$
 which allows us to reduce to the case where the entries of~$\widetilde{H}$ belongs to $\C\left[\left[z^{1/\nu}\right]\right]$.
\\ \par 
We recall that~$\widetilde{\C}$ is the Riemann surface of the logarithm. If~$a,b\in \R$ with~$a<b$, we define~$\mathcal{A}(a,b)$
as the ring of functions that are analytic in some punctured neighborhood of~$0$ in $$ \overline{S}(a,b):=\left\{z\in \widetilde{\C} \Big|  \arg(z)\in ]a,b[\right\}.$$ Let~$\C\{z\}$ be the ring of germs of analytic functions in the neighborhood of~$z=0$, and~$\C(\{z\})$ be its fraction field, that is the field of germs of meromorphic functions in the neighborhood of~$z=0$. Let~$\widetilde{B}\in \mathrm{M}_{m}\Big(\C(\{z\})\Big)$ be a companion matrix. We are now interested in the existence of a fundamental solution of the system~$\d \widetilde{Y}=\widetilde{B}\widetilde{Y}$, we will see as an equation, that has coefficients in~$\mathcal{A}(a,b)$, for some~$a<b$.\par
Once for all, we fix a determination of the complex logarithm over $\widetilde{\C}$ we call~$\log$.  We define the family of continuous map~$\left(\r_{a}\right)_{a\in \C}$, from the Riemann surface of the logarithm to itself, that sends~$z$ to~$e^{a\log (z)}$. One has~$\r_{b}\circ \r_{c}=\r_{bc}$ for any~$b,c\in \C$. For ${\widetilde{f}:=\sum f_{n}z^{n}\in \displaystyle\bigcup_{\nu\in \N^{*}} \C\left(\left(z^{1/\nu}\right)\right)}$ and~$c\in \Q_{>0}$, we set~${\rho_{c} \left(\widetilde{f}\right):=\sum f_{n}z^{nc}\in \displaystyle\bigcup_{\nu\in \N^{*}} \C\left(\left(z^{1/\nu}\right)\right)}$. 
For~$f\in \mathcal{A}(a,b)$ and~$c\in \Q_{>0}$, we define~$\rho_{c}\left(f\right):=f(z^{c})$.  Of course, the definitions of~$\rho_{c}$ coincide on~$\C(\{z\})$.\par

\pagebreak[3]
\begin{defi}\label{4defi3}
\begin{trivlist}
\item (1) Let~$k\in \Q_{>0}$. We define the formal Borel transform of order $k$,~$\hat{\mathcal{B}}_{k}$ as follows:
$$
\begin{array}{llll}
\hat{\mathcal{B}}_{k}:&\C[[z]]&\longrightarrow&\C[[\z]]\\
&\displaystyle\sum_{n\in \N} a_{n}z^{n}&\longmapsto&\displaystyle\sum_{n\in \N} \frac{a_{n}}{\G\left(1+\frac{n}{k}\right)}\z^{n},
\end{array}
$$
where~$\G$ is the Gamma function.
We remark that we have for all~$k\in \Q_{>0}$:
$$\hat{\mathcal{B}}_{k}=\r_{k}\circ \hat{\mathcal{B}}_{1}\circ \r_{1/k}.$$
\item (2) Let~$d\in \R$ and~$k\in \Q_{>0}$.
Let~$f$ be a function such that there exists~$\e>0$, such that~${f\in \mathcal{A}(d-\e,d+\e)}$. We say that~$f$ belongs to~$\widetilde{\mathbb{H}}_{k}^{d}$, if~$f$ admits an analytic continuation defined on~$\overline{S}(d-\e,d+\e)$ that we will still call~$f$, with exponential growth of order~$k$ at infinity. This means that there exist constants~$J,L>0$, such that for~$\z\in\overline{S}(d-\e,d+\e)$:
$$|f(\z)|<J\exp\left(L|\z|^{k}\right).$$
\item (3)
 Let~$d\in \R$ and~$k\in \Q_{>0}$.  We define the Laplace transformations of order~$1$ and $k$ in the direction~$d$  as follow (see \cite{B}, Page 13 for a justification that the maps are defined)
$$
\begin{array}{llll}
\mathcal{L}_{1}^{d}:&\widetilde{\mathbb{H}}_{1}^{d}&\longrightarrow&\mathcal{A}\left(d-\frac{\pi}{2},d+\frac{\pi}{2}\right)\\
&f&\longmapsto& \displaystyle \int_{0}^{\infty e^{id}}z^{-1}f(\z)e^{-\left(\frac{\z}{z}\right)}d\z ,\\\\
\mathcal{L}_{k}^{d}:&\widetilde{\mathbb{H}}_{k}^{d}&\longrightarrow&\mathcal{A}\left(d-\frac{\pi}{2k},d+\frac{\pi}{2k}\right)\\
&g&\longmapsto& \r_{k}\circ  \mathcal{L}_{1}^{d}\circ \r_{1/k}\left(g\right).
\end{array}
$$
\end{trivlist}
\end{defi}

The following proposition will be needed for the proof of our main result, Theorem~\ref{4theo1}.

\pagebreak[3]
\begin{propo}\label{4propo4}
Let~$\widetilde{f}\in \C[[z]]$, let~$d\in \R$ and let~$\widetilde{g}\in \widetilde{\mathbb{H}}_{1}^{d}$. Then:
\begin{itemize}
\item~$\hat{\mathcal{B}}_{1}\left(\d \widetilde{f}\;\right)=\d\hat{\mathcal{B}}_{1}\left(\widetilde{f}\;\right)$.
\item~ $\d\hat{\mathcal{B}}_{1}\left(z \widetilde{f}\;\right)=\z \hat{\mathcal{B}}_{1}\left(\widetilde{f}\;\right)$, where $\d:=\z\frac{d}{d \z}$.
\item~$\mathcal{L}_{1}^{d}\Big(\d \widetilde{g}\Big)=\d\mathcal{L}_{1}^{d}\Big(\widetilde{g}\Big).$
\item ~$z\mathcal{L}_{1}^{d}\Big(\d \widetilde{g}\Big)=  \mathcal{L}_{1}^{d}\Big(\z \widetilde{g}\Big)-z\mathcal{L}_{1}^{d}\Big(\widetilde{g}\Big)$.
\end{itemize}
\end{propo}

\begin{proof}
The two first points are straightforward computations. Let us prove the third point. Making the variable change~$\z \mapsto q\z$ in the integral, we find that for all~$q>1$,~$\mathcal{L}_{1}^{d}$ commutes with~$\sq$. Then, for all~$q>1$, we find
$$\mathcal{L}_{1}^{d}\Big(\dq \widetilde{g}\Big)=\dq\mathcal{L}_{1}^{d}\Big(\widetilde{g}\Big).$$
Since~$\widetilde{g}\in \widetilde{\mathbb{H}}_{1}^{d}$, the dominated convergence theorem allow us to take the limit as~$q$ goes to~$1$
$$\mathcal{L}_{1}^{d}\Big(\d \widetilde{g}\Big)=\lim\limits_{q \to 1}\mathcal{L}_{1}^{d}\Big(\dq \widetilde{g}\Big)=\lim\limits_{q \to 1} \dq\mathcal{L}_{1}^{d}\Big( \widetilde{g}\Big)=\d\mathcal{L}_{1}^{d}\Big(\widetilde{g}\Big).$$
Let us prove the last equality. Since~$\widetilde{g}\in \widetilde{\mathbb{H}}_{1}^{d}$, we may perform an integration by part (let~$\widetilde{g}'$ be the derivation of~$\widetilde{g}$), and we obtain:
\begin{align*}
z\mathcal{L}_{1}^{d}\Big(\d \widetilde{g}\Big)=&\displaystyle \int_{0}^{\infty e^{id}} \z \widetilde{g}'(\z)e^{-\left(\frac{\z}{z}\right)}d\z\\
=&\displaystyle \int_{0}^{\infty e^{id}}  \widetilde{g}(\z)e^{-\left(\frac{\z}{z}\right)}\left(-1+\frac{\z}{z}\right)d\z \\
=&  \mathcal{L}_{1}^{d}\Big(\z \widetilde{g}\Big)-z\mathcal{L}_{1}^{d}\Big(\widetilde{g}\Big).
\qedhere\end{align*}
\end{proof}

\begin{rem}\label{4rem6}
Let~$k\in \Q_{>0}$, let~$\widetilde{d}_{0},\dots,\widetilde{d}_{r}\in \C\left[z^{k}\right]$ and let us consider~$\widetilde{f}\in \C\left[\left[z^{k}\right]\right]$, that satisfies
\begin{equation}\label{4eq25}
\displaystyle\sum_{i=0}^{r}\widetilde{d}_{i}(z)\d^{i}\widetilde{f}=0.
\end{equation}
 From Proposition~\ref{4propo4}, there exist~$\widetilde{c}_{0},\dots,\widetilde{c}_{s}\in \C\left[z^{k}\right]$ with degree less or equal that the maximum of the degrees of the~$\widetilde{d}_{i}$, such that
$$\displaystyle \sum_{i=0}^{s}\widetilde{c}_{i}(z)\d^{i}\hat{\mathcal{B}}_{k}\left(\widetilde{f}\right)=0.$$
Furthermore, if there exists~$d\in \R$ such that ~$\hat{\mathcal{B}}_{k}\left(\widetilde{f}\right)\in \widetilde{\mathbb{H}}_{k}^{d}$, then we have:
$$\d\mathcal{L}_{k}^{d}\circ\hat{\mathcal{B}}_{k}\left( \widetilde{f}\right)=\mathcal{L}_{k}^{d}\circ\hat{\mathcal{B}}_{k}\left(\d \widetilde{f}\right) 
\hbox{ and }
\d \left(z^{k}\mathcal{L}_{k}^{d}\circ\hat{\mathcal{B}}_{k}\left( \widetilde{f}\right)\right)=\mathcal{L}_{k}^{d}\circ\hat{\mathcal{B}}_{k}\left(\d \left(z^{k}\widetilde{f}\right)\right) .
$$
Hence,~$\mathcal{L}_{k}^{d}\circ\hat{\mathcal{B}}_{k}\left( \widetilde{f}\right)$ is solution of (\ref{4eq25}). 
But in general, if~${\widetilde{f}\in \C\left[\left[z\right]\right]}$ is solution of a linear~$\d$-equation with coefficients in~$\C\left[z\right]$, then, for all $(d,k)\in \R\times \Q_{>0}$, we have~${\hat{\mathcal{B}}_{k}\left(\widetilde{f}\right)\notin \widetilde{\mathbb{H}}_{k}^{d}}$, and we must apply successively several Borel and Laplace transformations to compute an analytic  solution of the same equation. See Proposition~\ref{4propo5}.
\end{rem} 
Let us consider~$\d \widetilde{Y}=\widetilde{B}\widetilde{Y}$, where~$\widetilde{B}\in \mathrm{M}_{m}\Big(\C(\{z\})\Big)$ is a companion matrix and let~$\widetilde{H}$ be a formal matrix obtained with the Hukuhara-Turrittin theorem. We have seen that we may assume that~$\widetilde{H}$ has no poles at~$0$. Let~$\widetilde{h}\in \C\left[\left[z^{1/\nu}\right]\right]$ be an entry of~$\widetilde{H}$ and let us consider a linear~$\d$-equation satisfied by~$\widetilde{h}$:
\begin{equation}\label{4eq2}
\widetilde{b}_{m}\d^{m}\widetilde{h}+\widetilde{b}_{m-1}\d^{m-1}\widetilde{h}+\dots+\widetilde{b}_{0}\widetilde{h}=0,
\end{equation}
 with ~$\widetilde{b}_{m}\neq 0$ and~$\widetilde{b}_{i}\in \C\left(\left\{z^{1/\nu}\right\}\right)$. Assume that (\ref{4eq2}) has at least one slope different from~$0$. Let~$d_{0}:=\max\left(2,\deg \left(\widetilde{b}_{0}\right),\dots, \deg \left(\widetilde{b}_{m}\right)\right)$, where~$\deg$ denotes the degree.
Let~$k_{1}<\dots<k_{r-1}$ be the slopes of (\ref{4eq2}) different from $0$, let~$k_{r}$ be an integer strictly bigger than~$k_{r-1}$ and~$d_{0}$, and set~$k_{r+1}:=+\infty$. 
Let~$(\k_{1},\dots,\k_{r})$ be defined by:
$$\k_{i}^{-1}:=k_{i}^{-1}-k_{i+1}^{-1}.$$  
We define the rational numbers ~$(\widetilde{\k_{1}},\dots,\widetilde{\k_{s}})$ as follows:
We take~$(\k_{1},\dots,\k_{r})$ and for~${i=1,...,i=r}$, replace successively~$\k_{i}$ by~$\a_{i}$ terms~$\a_{i}\k_{i}$, where~$\a_{i}$ is the smallest integer such that~$\a_{i}\k_{i}$ is greater or equal than~$d_{0}$. Therefore, by construction, all the~$\widetilde{\k_{i}}$ are greater than~$d_{0}\geq 2$, $\widetilde{\k_{s}}$ belongs to $\N$, and~$\widetilde{\k_{s}}= \k_{r}=k_{r}> k_{r-1}$. \par 
\pagebreak[3]
\begin{ex}\label{4ex1}
Assume that $\widetilde{h}\in \C[[z]]$ is solution of 
$$\left(z^{4}+z^{3}\right)\d^{3}\widetilde{h}+z\d^{2}\widetilde{h}+\d\widetilde{h}-\widetilde{h}=0.$$
We have~$d_{0}=4$, $r=3$ and $(k_{1},k_{2},k_{3},k_{4})=(1,2,5,\infty)$. Then, we find that ${(\k_{1},\k_{2},\k_{3})=(2,10/3,5)}$,~${s=5}$, and we obtain ${(\widetilde{\k_{1}},\dots,\widetilde{\k_{5}})=(4,4,20/3,20/3,5)}$.
\end{ex}

We recall that $\widetilde{h}\in \C\left[\left[z^{1/\nu}\right]\right]$. Let us write~$\widetilde{h}=:\displaystyle\sum_{\substack{n=0\\n\in \N/\nu}}^{\infty}\widetilde{h}_{n}z^{n}$. Let~$\b\in \N^{*}$ be minimal such that~$\b/\widetilde{\k_{1}},\dots,\b/\widetilde{\k_{s}}$ belong to~$\N^{*}$ and for~~$l\in\{0,\dots,\b\nu-1\}$, let~$\widetilde{h}^{(l)}:=\displaystyle\sum_{n=0}^{\infty}\widetilde{h}_{l/\nu+n\b}z^{n\b}$. 

\pagebreak[3]
\begin{propo}\label{4propo5}
Let us keep the same notations as above. There exists~$\widetilde{\Sigma}_{\widetilde{h}}\subset \R$, finite modulo~$2\pi\Z$, such that for all~$l\in\{0,\dots,\b\nu-1\}$, if~$d\in \R\setminus \widetilde{\Sigma}_{\widetilde{h}}$, 
the series~${\widetilde{f}_{1,l}:=\hat{\mathcal{B}}_{\widetilde{\k_{1}}}\circ\dots\circ\hat{\mathcal{B}}_{\widetilde{\k_{s}}}\left(\widetilde{h}^{(l)}\right)}$ converges and belongs to~$\widetilde{\mathbb{H}}_{\widetilde{\k_{1}}}^{d}$.\par 
Moreover, for~$j=2$ (resp.~$j=3$,~$\dots$, resp.~$j=s$),~${\widetilde{f}_{j,l}:=\mathcal{L}_{\widetilde{\k_{j-1}}}^{d}\left(\widetilde{f}_{j-1,l}\right)}$ belongs to~$\widetilde{\mathbb{H}}_{\widetilde{\k_{j}}}^{d}$.  Let~${\widetilde{S}^{d}\left(\hat{h}^{(l)}\right):=\mathcal{L}_{\widetilde{\k_{s}}}^{d}\left(\widetilde{f}_{s,l}\right)}$. 
The function $$\widetilde{S}^{d}\left(\widetilde{h}\right):=\sum_{l=0}^{\b\nu-1} z^{l/\nu}\widetilde{S}^{d}\left(\widetilde{h}^{(l)}\right)
\in\mathcal{A}\left(d-\frac{\pi}{2\widetilde{\k_{s}}},d+\frac{\pi}{2\widetilde{\k_{s}}}\right)=
\mathcal{A}\left(d-\frac{\pi}{2k_{r}},d+\frac{\pi}{2k_{r}}\right),$$  is solution of the same linear~$\d$-equation than~$\widetilde{h}$.
\end{propo}

\pagebreak[3]
\begin{rem}
We make a priori an abuse of notations, since  $\widetilde{S}^{d}\left(\widetilde{h}\right)$ may depend on the choice of the linear differential equation satisfied by $\widetilde{h}$.
However, we can directly deduce from Lemma 2 in \cite{B},~$\S 6.2$, that $\widetilde{S}^{d}\left(\widetilde{h}\right)$ is independent upon the choice of the linear differential equation satisfied by $\widetilde{h}$. Notice that we will not use this fact.
\end{rem}
\pagebreak[3]
\begin{rem}
As we can see in Theorem 7.51 in \cite{VdPS}, the function $\widetilde{S}^{d}\left(\widetilde{h}\right)$ is~$\widetilde{\k_{s}}$-Gevrey asymptotic to~$\widetilde{h}$ on $\overline{S}\left(d-\dfrac{\pi}{2\widetilde{\k_{s}}},d+\dfrac{\pi}{2\widetilde{\k_{s}}}\right)$:
for every closed subsector~$W$ of~$\overline{S}\left(d-\dfrac{\pi}{2\widetilde{\k_{s}}},d+\dfrac{\pi}{2\widetilde{\k_{s}}}\right)$, there exist~${A_{W} \in \R}$,~${\e >0}$ such that for all~$N\in \N^{*}$ and all~$z\in W$ with~$|z|<\e$,
$$\left|\widetilde{S}^{d}\left(\widetilde{h}\right)(z)-\sum_{n= 0}^{N-1} \widetilde{h}_{n}z^{n}\right|\leq (A_{W})^{N}\G\left(1+\frac{N}{\widetilde{\k_{s}}}\right) |z|^{N}.$$ 
\end{rem}
\begin{proof}[Proof of Proposition \ref{4propo5}]
Let~$\widetilde{g}:=\rho_{\nu}\widetilde{h}\in \C[[z]]$. For all~$l\in\{0,\dots,\b\nu-1\}$, we have $$z^{l/\nu}\widetilde{h}^{(l)}(z,q)=\rho_{1/\nu}\displaystyle\sum_{j=0}^{\b\nu-1}\frac{\widetilde{g}\left(e^{2i\pi lj/\b\nu}z\right)}{e^{2i\pi lj/\b\nu}\b\nu}.$$
It follows that there exists~$\widetilde{\Sigma}_{\widetilde{h}}\subset \R$, finite modulo~$2\pi\Z$, such that for all~$l\in\{0,\dots,\b\nu-1\}$, if~$d\in \R\setminus \widetilde{\Sigma}_{\widetilde{h}}$, then 
\begin{itemize}
\item $\widetilde{f}_{1}:=\hat{\mathcal{B}}_{\widetilde{\k_{1}}}\circ\dots\circ\hat{\mathcal{B}}_{\widetilde{\k_{s}}}\left(\widetilde{h}\right)\in\widetilde{\mathbb{H}}_{\widetilde{\k_{1}}}^{d}~$ if and only if for all integers ${l\in\{0,\dots,\b\nu-1\}}$, we have~${\widetilde{f}_{1,l}:=\hat{\mathcal{B}}_{\widetilde{\k_{1}}}\circ\dots\circ\hat{\mathcal{B}}_{\widetilde{\k_{s}}}\left(\widetilde{h}^{(l)}\right)
\in\widetilde{\mathbb{H}}_{\widetilde{\k_{1}}}^{d}}$.
\item For~$j=2$ (resp.~$j=3$,~$\dots$, resp.~$j=s$),~$\widetilde{f}_{j}:=\mathcal{L}_{\widetilde{\k_{j-1}}}^{d}\left(\widetilde{f}_{j-1}\right)\in\widetilde{\mathbb{H}}_{\widetilde{\k_{j}}}^{d}$
if and only if for all~${l\in\{0,\dots,\b\nu-1\}}$,~$\widetilde{f}_{j,l}:=\mathcal{L}_{\widetilde{\k_{j-1}}}^{d}\left(\widetilde{f}_{j-1,l}\right)\in\widetilde{\mathbb{H}}_{\widetilde{\k_{j}}}^{d}$.
\end{itemize}
Let~$d\in  \R\setminus\widetilde{\Sigma}_{\widetilde{h}}$ and
let~$(\k'_{1},\dots,\k'_{r-1})$ defined as:
$$\k'_{r-1}:=k_{r-1} \hbox{ and for }i<r-1, \frac{1}{\k'_{i}}:=\frac{1}{k_{i}}-\frac{1}{k_{i+1}}.$$ 
Due to Theorem 7.51 in \cite{VdPS} and \cite{B},~$\S 7.2$,  
~$\widetilde{f}'_{1}:=\hat{\mathcal{B}}_{\k'_{1}}\circ\dots\circ\hat{\mathcal{B}}_{\k'_{r}}\left(\widetilde{h}\right)
\in\widetilde{\mathbb{H}}_{\k'_{1}}^{d}$, and for~$j=2$ (resp.~$j=3$,~$\dots$, resp.~$j=r-1$),~$\widetilde{f}'_{j}:=\mathcal{L}_{\k'_{j-1}}^{d}\left(\widetilde{f}'_{j-1}\right)\in\widetilde{\mathbb{H}}_{\k'_{j}}^{d}$. With Lemma 2 in \cite{B},~$\S 6.2$, this implies that~${\widetilde{f}_{1}:=
\hat{\mathcal{B}}_{\widetilde{\k_{1}}}\circ\dots\circ\hat{\mathcal{B}}_{\widetilde{\k_{s}}}\left(\widetilde{h}^{(l)}\right)\in\widetilde{\mathbb{H}}_{\widetilde{\k_{1}}}^{d}}$
and for~$j=2$ (resp.~$j=3$,~$\dots$, resp.~$j=s$),~$\widetilde{f}_{j}:=\mathcal{L}_{\widetilde{\k_{j-1}}}^{d}\left(\widetilde{f}_{j-1,l}\right)\in\widetilde{\mathbb{H}}_{\widetilde{\k_{j}}}^{d}$. With the equivalence we have written in the beginning of the proof, we may apply successively the Borel and Laplace transformations of the required order to each series~$\hat{h}^{(l)}$. \par 
To finish, we have to prove that~$\widetilde{S}^{d}\left(\widetilde{h}\right)$ is solution of the same linear~$\d$-equation than~$\widetilde{h}$. This is a direct consequence of Theorem 2 in \cite{B},~$\S 6.4$.
\end{proof}

 As a matter of fact, as we can see in Page $239$ of \cite{VdPS}, $\widetilde{S}^{d}\left(\widetilde{h}\right)$ belongs to $\mathcal{A}\left(d_{l}-\frac{\pi}{2k_{r}},d_{l+1}+\frac{\pi}{2k_{r}}\right)$, where ~${d_{l},d_{l+1}\in \widetilde{\Sigma}_{\widetilde{h}}}$ are chosen such that~$\big]d_{l},d_{l+1}\big[ \bigcap \widetilde{\Sigma}_{\widetilde{h}}=\varnothing$. \par
If (\ref{4eq2}) has only slope~$0$, then~$\widetilde{h}\in \C\left\{z^{1/\nu}\right\}$. In this case we set~$\widetilde{\Sigma}_{\widetilde{h}}:=\varnothing$, and for all~$d\in \R$ we set $$\widetilde{S}^{d}\left(\widetilde{h}\right):=\widetilde{h}.$$ \par 
We recall that we consider the equation~$\d \widetilde{Y}=\widetilde{B}\widetilde{Y}$, where~$\widetilde{B}\in \mathrm{M}_{m}\Big(\C(\{z\})\Big)$ is a companion matrix and~${\widetilde{H}:=\left(\widetilde{h}_{i,j}\right)\in \mathrm{M}_{m}\Big(\C\left[\left[z^{1/\nu}\right]\right]\Big)}$ is a formal matrix given by the Hukuhara-Turrittin theorem. The entries of~$\widetilde{H}$ satisfy linear~$\d$-equations with coefficients in~$\C\left[z^{1/\nu}\right]$ for some~$\nu$. We may assume that for a given entry, the coefficients of the~$\d$-equation are relatively prime. Let~$d_{0}$ be the maximum among~$2$ and the degrees of the coefficients of the equations.
Let~$\widetilde{\Sigma}_{\widetilde{H}}$ be the union of the~$\widetilde{\Sigma}_{\widetilde{h}_{i,j}}$, where~$\widetilde{\Sigma}_{\widetilde{h}_{i,j}}$ has been defined in Proposition~\ref{4propo5};~$k_{i,j}\in \Q$ be the biggest slope of the equation satisfied by~$\widetilde{h}_{i,j}$;~$k'$ be the maximum of the ~$k_{i,j}$; and~$k$ be an integer strictly bigger than~$k'$ and~$d_{0}$. Let~$d,d^{\pm}\in  \R\setminus \widetilde{\Sigma}_{\widetilde{H}}$, with $$d-\frac{\pi}{2k}<d^{-}<d<d^{+}<d+\frac{\pi}{2k},$$ and such that~$\displaystyle\left(\big[d^{-},d\big[\bigcup \big]d,d^{+}\big]\right)\bigcap \widetilde{\Sigma}_{\widetilde{H}}=\varnothing$. Let~$\widetilde{S}^{d^{\pm}}\left(\widetilde{H}\right):=\widetilde{S}^{d^{\pm}}\left(\widetilde{h}_{i,j}\right)$. 
We get two analytic solutions, $$\widetilde{S}^{d^{-}}\left(\widetilde{H}\right)\mathrm{Diag} \left(e^{\widetilde{L}_{i} \log(z)}e^{\widetilde{\l}_{i}\times\mathrm{Id}_{m_{i}}}\right)\in \mathrm{GL}_{m}\left(\mathcal{A}\left(d^{-}-\dfrac{\pi}{2k},d+\dfrac{\pi}{2k}\right)\right),$$
and
 $$\widetilde{S}^{d^{+}}\left(\widetilde{H}\right)\mathrm{Diag} \left(e^{\widetilde{L}_{i} \log(z)}e^{\widetilde{\l}_{i}\times\mathrm{Id}_{m_{i}}}\right)\in \mathrm{GL}_{m}\left( \mathcal{A}\left(d-\dfrac{\pi}{2k},d^{+}+\dfrac{\pi}{2k}\right)\right).$$
Note that by definition, the analyticity holds on a subset of~$\widetilde{\C}$.
 A computation shows that there exists a matrix~$\widetilde{ST}_{d}\in \mathrm{GL}_{m}(\C)$, we call the Stokes matrix in the direction~$d$, such that: 
$$
\widetilde{S}^{d^{+}}\left(\widetilde{H}\right)\mathrm{Diag} \left(e^{\widetilde{L}_{i} \log(z)}e^{\widetilde{\l}_{i}\times\mathrm{Id}_{m_{i}}}\right)
= \widetilde{S}^{d^{-}}\left(\widetilde{H}\right)\mathrm{Diag} \left(e^{\widetilde{L}_{i} \log(z)}e^{\widetilde{\l}_{i}\times\mathrm{Id}_{m_{i}}}\right)\widetilde{ST}_{d}.
$$
 The Stokes matrices belong to the differential Galois group, see Chapter 8 of \cite{VdPS}. 

\pagebreak[3]
\section{Local formal study of~$q$-difference equations}\label{4sec2}
In this section, we summarize results about formal classification of linear~$q$-difference equations. See in particular \cite{RSZ} for more details. Let~$q>1$ be fixed. We extend the action of~$\sq$ to~$\displaystyle\bigcup_{\nu\in \N^{*}}\C\left(\left(z^{1/\nu}\right)\right)$ by~${\sq z^{1/\nu}=e^{\log(q)/\nu}z^{1/\nu}}$, for~$\nu\in\N^{*}$.
Let~$K$ be an intermediate field extension:~$\C(z)\subset K\subset\displaystyle\bigcup_{\nu\in \N^{*}}\C\left(\left(z^{1/\nu}\right)\right)$, stable by~$\sq$. \par
Let us consider the~$q$-difference operator:
$$P=\displaystyle \sum_{i=l}^{m} b_{i}\sq^{i},$$
where~$b_{i}\in K$,~$l,m\in \Z$ and~$l<m$.
The Newton polygon of~$P$ is the convex hull of 
$$\displaystyle \bigcup_{k=l}^{m}\Big\{ (i,j)\in \Z\times\Q \Big| j\geq v_{0}\left(b_{k}\right)\Big\},$$ where~$v_{0}$ denotes the~$z$-adic valuation of~$K$.
Let~$\big\{(d_{1},n_{1}),\dots ,(d_{r},n_{r})\big\}$
be a minimal subset of~$\mathbb Z^{2}$ for the inclusion, with~$d_{1}<\dots<d_{r}$,
such that the Newton polygon is the convex hull of
$$\displaystyle \bigcup_{k=1}^{r}\Big\{ (d_{k},j)\in \Z\times\Q \Big| j\geq n_{k}\Big\}.$$
 We call slopes of the linear~$q$-difference equation the rational numbers~${\frac{n_{i+1}-n_{i}}{d_{i+1}-d_{i}}}$, and multiplicity of the slope~$\frac{n_{i+1}-n_{i}}{d_{i+1}-d_{i}}$, the integer~$d_{i+1}-d_{i}$. \par 
Like in $\S \ref{4sec1}$, let~$B\in \mathrm{GL}_{m}(K)$ be a companion matrix. As in the differential case, we can naturally associate to the linear~$\sq$-equation~${\sq Y=BY}$ a unitary~$q$-difference operator~${P=\sq^{m}+b_{m-1}\sq^{m-1}+\dots+b_{0}}$ with coefficients in~$K$. We define the Newton polygon of~$\sq Y=BY$, as the Newton polygon of~$P$. We also define the slopes and the multiplicities of the slopes of~$\sq Y=BY$ as the slopes and the multiplicities of the slopes of~$P$. Notice that if $B\in \mathrm{GL}_{m}\Big(\C((z))\Big)$ is not a companion matrix, we can still define the Newton polygon of $\sq Y=BY$, but we will not need this in this paper. \par
Let~$A,B\in \mathrm{GL}_{m}(K)$. The two~$q$-difference systems,~$\sq Y=AY$ and~$\sq Y=BY$ are equivalent over~$K$, if there exists~$P\in \mathrm{GL}_{m}(K)$, called gauge transformation, such that
$$A=P[B]_{\displaystyle\sq}:=(\sq P)BP^{-1}.$$ In particular,
$$\sq Y=BY\Longleftrightarrow\sq \left(PY\right)=APY.$$ \par
Conversely, if there exist $A,B,P\in \mathrm{GL}_{m}(K)$ such that
$\sq Y=BY$, $ \sq Z=AZ$ and $Z=PY$, then 
$$A=P\left[ B\right]_{\displaystyle\sq}.$$
If the above matrices~$A,B\in \mathrm{GL}_{m}(K)$ are companion matrices, then, see \cite{RSZ}, Theorem~2.2.1, they have the same Newton polygon. 

\pagebreak[3]
\begin{theo}[\cite{RSZ},~$\S 2.2$]\label{4theo6}
Let~$B\in \mathrm{GL}_{m}\Big(\C((z))\Big)$ be a companion matrix and let us consider~${\sq Y=BY}$. Let~$\mu_{1},\dots,\mu_{k}$ be the slopes of the~$q$-difference equation, let~$m_{1}\dots,m_{k}$ be their multiplicities and assume that the~$\mu_{i}$ belong to~$\Z$. Then, we have existence of ~$B_{i}\in \mathrm{GL}_{m_{i}}(\C)$,~$\hat{H}\in \mathrm{GL}_{m}\Big(\C((z))\Big)$, such that: 
$$B=\hat{H}\Big[\mathrm{Diag} \left(z^{-\mu_{i}}B_{i}\right)\Big]_{\displaystyle\sq}.$$ 
\end{theo}
 See \cite{VdPR} for a more general result that works for~$q$-difference equation with arbitrary slopes. Notice that for all~$n\in \Z$, we have also
$$ B=\left(z^{n}\hat{H}\right)\Big[\mathrm{Diag} \left(B_{i}\times q^{-n}z^{-\mu_{i}}\right)\Big]_{\displaystyle\sq},$$
 which allow us to reduce to the case where~$\hat{H}$ has entries in $\C[[z]]$.\par 
We want to determine the eigenvalues of the~$B_{i}$ and the~$z$-valuation of the entries of~$\hat{H}$. Let~$b_{0},\dots,b_{m-1}\in  \C((z))$, and let us consider the~$q$-difference equation:
\begin{equation}\label{4eq18}
\sq^{m}y+b_{m-1}\sq^{m-1}y+\dots+b_{0}y=0.
\end{equation}
 Let~$\big\{(d_{1},n_{1}),\dots ,(d_{r},n_{r})\big\}$
be a minimal subset of~$\mathbb Z^{2}$ for the inclusion, with~${d_{1}<\dots<d_{r}}$,
such that the Newton polygon is the convex hull of
~$\displaystyle \bigcup_{k=1}^{r}\Big\{ (d_{k},j)\in \Z\times\Z \Big| j\geq n_{k}\Big\}.$ Let~$\mu_{1},\dots,\mu_{k}$ be the slopes of the~$q$-difference equation,~$m_{1}\dots,m_{k}$ be their multiplicities and assume that the slopes~$\mu_{i}=\frac{n_{i+1}-n_{i}}{d_{i+1}-d_{i}}$ belongs to~$\Z$.\par 
For~$d_{i}\leq j\leq d_{i+1}$, let~$a_{j}$ be the value at~$z=0$ of~$b_{j}(z)z^{-n_{i}-\mu_{i}(j-d_{i})}$. We define the characteristic polynomial associated to the slope~$\mu_{i}$ as follows: 
$$
P^{(\mu_{i})}(X):=\left(a_{d_{i+1}}q^{\mu_{i}d_{i+1}(d_{i+1}-1)/2}X^{d_{i+1}-d_{i}}+\dots+a_{d_{i}}q^{\mu_{i}d_{i}(d_{i}-1)/2}\right).$$
From \cite{MZ}, Theorem 3.2.3, we deduce directly the following:
\pagebreak[3]
\begin{theo}\label{4theo5}
Let~$B\in \mathrm{GL}_{m}\Big(\C((z))\Big)$ be a companion matrix, such that~$\sq Y=BY$ is the linear~$\sq$-system equivalent to (\ref{4eq18}). There exist
\begin{itemize}
\item~$B_{i}\in \mathrm{GL}_{m_{i}}(\C)$, which are of the form~$\mathrm{Diag}_{l}\left(T_{i,l}\right)$, where~$T_{i,l}$ are upper triangular matrices with diagonal terms that are equal to the roots of the characteristic polynomial associated to the slope~$\mu_{i}$,
\item~$\hat{H}\in \mathrm{GL}_{m}\Big(\C((z))\Big)$, whose entries of the first row of~$\hat{H}$ have~$z$-valuation equal to~$0$,
\end{itemize}
such that
$$B=\hat{H}\Big[\mathrm{Diag} \left(z^{-\mu_{i}}B_{i}\right)\Big]_{\displaystyle\sq}.$$
\end{theo}

\pagebreak[3]
\section{Definition of~$q$-Borel and~$q$-Laplace transformations.}\label{4sec3}

The goal of this section is to define~$q$-analogues of the Borel and Laplace transformations. We will study their behavior as~$q$ goes to~$1$ in~$\S  \ref{4sec42}$. Remark that there are several possible definitions of~$q$-analogues of Borel and Laplace transformations. See \cite{DVZ,MZ,R92,RZ,Z99,Z00,Z01,Z02,Z03} for example. Following \cite{DVZ}, we begin by defining a~$q$-Borel transformation we are going to study. In this section,~$q>1$ is fixed. Let us recall that for all~$n\in \N$,~$[n]_{q}^{!}=\displaystyle\prod_{l=1}^{n}\frac{q^{l}-1}{q-1}$.
\pagebreak[3]
\begin{defi}\label{4defi4}
Let~$k\in \Q_{>0}$ and let~$\nu\in \N^{*}$ minimal such that~$\nu k \in \N^{*}$. We define~$\hat{\mathcal{B}}_{q,k}$ as follows:
$$
\begin{array}{llll}
\hat{\mathcal{B}}_{q,k}:& \C\left[\left[z^{\nu k}\right]\right]&\longrightarrow&\C\left[\left[\z^{\nu k}\right]\right]\\
&\displaystyle\sum_{l\in \N} a_{l}z^{l}&\longmapsto&\displaystyle\sum_{l\in \N} \frac{a_{l}}{[l/k]_{q}^{!}}\z^{l},
\end{array}
$$
Let~$k\in \Q_{>0}$, let~$\nu\in \N^{*}$ minimal such that~$\nu k \in \N^{*}$ and let~$\r_{k},\r_{1/k}$ be the maps defined in~$\S \ref{4sec1}$. We remark that we have:
$$\hat{\mathcal{B}}_{q,k}=\r_{k}\circ \hat{\mathcal{B}}_{q,1}\circ \r_{1/k}.$$
\end{defi}

\pagebreak[3]
\begin{defi}
Let~$d\in \R$ and let~$k\in \Q_{>0}$.
Let~$f$ be a function such that there exists~$\e>0$, such that~$f\in \mathcal{A}(d-\e,d+\e)$. We say that~$f$ belongs to~$\mathbb{H}_{q,k}^{d}$, if~$f$ admits an analytic continuation defined on~$\overline{S}(d-\e,d+\e)$, that we will still call~$f$, such that there exist constants~$J,L>0$, such that for~$\z\in\overline{S}(d-\e,d+\e)$ (see the introduction for the definition of $e_{q}$):
$$|f(\z)|<Je_{q}\left(L|\z|^{k}\right).$$
\end{defi}

For all~$d\in \R$, we write~$[d]:=q^{\Z}e^{id}$ the discrete logarithmic~$q$-spiral through the point~$e^{id}\in \C^{*}$. For~$d\in \R$ we set the Jackson integral: 
$$\displaystyle \int_{[d]} f(\z)d_{q}\z:=(q-1) \displaystyle \sum_{l\in \Z}f\left(q^{l}e^{id}\right)q^{l}e^{id},$$
whenever the right hand side converges. Roughly speaking, Jackson integral degenerates into classical integral when $q$ goes to $1$, which means that for a convenient choice of function~$f$, we have on a convenient domain $$\displaystyle \int_{[d]} f(\z)d_{q}\z \underset{q \to 1}{\longrightarrow} \displaystyle \int_{0}^{\infty e^{id}} f(\z)d\z.$$ \par
From now, let~$p:=1/q\in ]0,1[$. Let~$\mathcal{M}(\C^{*},0)$ be the field of functions that are meromorphic on some punctured neighborhood of~$0$ in~$\C^{*}$. We define now the~$q$-Laplace transformation. 

\pagebreak[3]
\begin{defi}\label{4defi5} Let~$k\in \Q_{>0}$ and let~$\rho_{k}\left(\mathcal{M}(\C^{*},0)\right):=\{\rho_{k}(f)|f\in\mathcal{M}(\C^{*},0)\}$. Let~$d\in \R$. As we can see in \cite{DVZ},~$\S 4.2$, the following maps are well defined and we call them the~$q$-Laplace transformation of order~$1$ and~$k$ respectively:
$$
\begin{array}{llll}
\mathcal{L}_{q,1}^{[d]}:& \mathbb{H}_{q,1}^{d}&\longrightarrow&\mathcal{M}(\C^{*},0)\\
&f&\longmapsto&\displaystyle \int_{[d]}\frac{f(\z)}{ze_{q}\left(\frac{q\z}{z} \right)}d_{q}\z,\\\\
\mathcal{L}_{q,k}^{[d]}:& \mathbb{H}_{q,k}^{d}&\longrightarrow&\rho_{k}\left(\mathcal{M}(\C^{*},0)\right)\\
&g&\longmapsto&\r_{k}\circ \mathcal{L}^{[d]}_{q,1}\circ \r_{1/k}(g).
\end{array}
$$

\end{defi} 
For~$|z|$ small, the function~$\mathcal{L}_{q,1}^{[d]}(f)(z)$ has poles of order at most~$1$ that are contained on the~$q$-spiral~${(q-1)[d+\pi]:=q^{\Z}(1-q)e^{id}}$.
The following proposition is the~$q$-analogue of Proposition~\ref{4propo4}. 

\pagebreak[3]
\begin{propo}\label{4propo2}
Let~$\hat{f}\in \C[[z]]$, let~$d\in \R$, and let~$g\in  \mathbb{H}_{q,1}^{d}$. Then 
\begin{itemize}
\item~$\hat{\mathcal{B}}_{q,1}\left(\dq \hat{f}\right)= \dq\hat{\mathcal{B}}_{q,1}\left( \hat{f}\right)$ .
\item~$\dq\hat{\mathcal{B}}_{q,1}\left( z\hat{f}\right)= \z\hat{\mathcal{B}}_{q,1}\left( \hat{f}\right)$.
\item~$ \mathcal{L}_{q,1}^{[d]}\Big(\dq g\Big)= \dq\mathcal{L}_{q,1}^{[d]}\Big( g\Big)$. 
\item~$z\mathcal{L}_{q,1}^{[d]}\Big( \dq g\Big)= p\mathcal{L}_{q,1}^{[d]}\Big( \z g\Big)-pz\mathcal{L}_{q,1}^{[d]}\Big(g\Big)$.
\end{itemize}
\end{propo}

\begin{proof}
The three first points are straightforward computations. Let us prove the last equality. Let~$z\in \C^{*}$. It is a well known fact and easy to verify that~${\sq\Big(e_{q}(z)e_{p}(-z)\Big)=e_{q}(z)e_{p}(-z)}$. Since~$e_{q}(z)e_{p}(-z)$ is a formal power series with constant term equals to~$1$,~$e_{q}(z)e_{p}(-z)=1$.
We have the equalities:
$$\begin{array}{llrll}
z\mathcal{L}_{q,1}^{[d]}\Big(\dq g\Big)&=&(q-1)e^{id} &\displaystyle \sum_{l\in \Z}&\dfrac{\dq g\left(q^{l}e^{id}\right)}{e_{q}\left(\frac{q^{l+1}e^{id}}{z}\right)}q^{l}\\
&=&(q-1)e^{id} &\displaystyle \sum_{l\in \Z}&\dq g\left(q^{l}e^{id}\right)e_{p}\left(\frac{-q^{l+1}e^{id}}{z}\right)q^{l}\\
&=& e^{id} &\displaystyle \sum_{l\in \Z}& g\left(q^{l+1}e^{id}\right)e_{p}\left(\frac{-q^{l+1}e^{id}}{z}\right)-  g\left(q^{l}e^{id}\right)e_{p}\left(\frac{-q^{l+1}e^{id}}{z}\right)q^{l}\\
&=& e^{id} &\displaystyle \sum_{l\in \Z}& g\left(q^{l}e^{id}\right)\left(e_{p}\left(\frac{-q^{l+1}e^{id}}{qz}\right)p-  e_{p}\left(\frac{-q^{l+1}e^{id}}{z}\right)\right)q^{l}\\
&=& (p-1)e^{id} &\displaystyle \sum_{l\in \Z}& g\left(q^{l}e^{id}\right)e_{p}\left(\frac{-q^{l+1}e^{id}}{z}\right)\left( \frac{-q^{l}e^{id}}{z}+1\right)q^{l}\\
&=& p(q-1)e^{id} &\displaystyle \sum_{l\in \Z}& \dfrac{g\left(q^{l}e^{id}\right)}{e_{q}\left(\frac{q^{l+1}e^{id}}{z}\right)}\left( \frac{q^{l}e^{id}}{z}-1\right)q^{l}\\
&=&p\mathcal{L}_{q,1}^{[d]}\left(\z g\Big(\z\Big)\right)&-&pz\mathcal{L}_{q,1}^{[d]}\Big(g(\z)\Big).
\end{array}$$
\end{proof}

\begin{rem}\label{4rem3}
Let~$k\in \N^{*}$ and let~$d\in \R$. If we consider~$\hat{f}\in \C\left[\left[z^{ k}\right]\right]$, solution of a linear~$\dq$-equation with coefficients in~$\C\left[z^{ k}\right]$ with~${\hat{\mathcal{B}}_{q,k}\left(\hat{f}\right)\in \mathbb{H}_{q,k}^{d}}$, then we have:
$$\dq\left(\mathcal{L}_{q,k}^{[d]}\circ\hat{\mathcal{B}}_{q,k}\left( \hat{f}\right)\right)=\mathcal{L}_{q,k}^{[d]}\circ\hat{\mathcal{B}}_{q,k}\left(\dq \hat{f}\right) 
\hbox{ and }
\dq \left(z^{k}\mathcal{L}_{q,k}^{[d]}\circ\hat{\mathcal{B}}_{q,k}\left( \hat{f}\right)\right)=\mathcal{L}_{q,k}^{[d]}\circ\hat{\mathcal{B}}_{q,k}\left(\dq\left( z^{k}\hat{f}\right)\right) .
$$
Hence,~$\mathcal{L}_{q,k}^{[d]}\circ\hat{\mathcal{B}}_{q,k}\left(\hat{f}\right)$ is solution of the same linear~$\dq$-equation than~$\hat{f}$. But in general, if~$\hat{f}\in \C\left[\left[z\right]\right]$ is solution of a linear~$\dq$-equation with coefficients in~$\C\left[z\right]$, we will have to apply successively several~$q$-Borel and~$q$-Laplace transformations in order to compute an analytic solution of the same equation than~$\hat{f}$. See Theorem~\ref{4theo1}. We are not going to give explicit examples of such sitation since the proof that we need strictly more than one~$q$-Borel and~$q$-Laplace transformation should be very technical. 
\end{rem}

In~$\S \ref{4sec7}$, we will use other~$q$-analogue of the Borel (resp. Laplace) transformation that has been originally introduced by Ramis (resp. Zhang). See \cite{Z02},~$\S 1$ for the justification of the convergence of the~$q$-Laplace transformation.

\pagebreak[3]
\begin{defi}\label{4defi6}
\begin{trivlist}
\item (1) We define~$\hat{B}_{q}$ as follows: 
$$
\begin{array}{llll}
\hat{B}_{q}:& \C[[z]]&\longrightarrow&\C[[\z]]\\
&\displaystyle\sum_{l\in \N} a_{l}z^{l}&\longmapsto&\displaystyle\sum_{l\in \N} \frac{a_{l}}{q^{l(l-1)/2}}\z^{l}.
\end{array}
$$
\item (2) Let~$d\in \R$. We define the map~$L_{q}^{[d]}$ as follows:
$$
\begin{array}{llll}
L_{q}^{[d]}:& \mathbb{H}_{q,1}^{d}&\longrightarrow&\mathcal{M}(\C^{*},0)\\
&f&\longmapsto&\displaystyle \sum_{n\in \Z}\frac{f\left(q^{n}(q-1)e^{id}\right)}{\T_{q}\left(\frac{q^{n+1}(q-1)e^{id}}{z}\right)}.
\end{array}
$$
For~$|z|$ small, the function~$L_{q}^{[d]}(f)(z)$ admits a spiral of poles of order at most~$1$ that are contained in the~$q$-spiral~${(q-1)[d+\pi]}$.
\end{trivlist}
\end{defi}

\pagebreak[3]
\begin{rem}\label{4rem5}
Let~$d\in \R$. The maps~$\hat{\mathcal{B}}_{q,1}$,~$\mathcal{L}_{q,1}^{[d]}$,~$\hat{B}_{q}$ and~$L_{q}^{[d]}$ are very similar to the ~$q$-Borel and the ``discrete''~$q$-Laplace transformations introduced in \cite{DVZ},~$\S 4.2$. Let~$\hat{f}\in \C[[z]]$ such that there exists~$d\in\R$ with~$g:=\hat{\mathcal{B}}_{q,1}\left(\hat{f}\right) \in  \mathbb{H}_{q,1}^{d}$ (resp.~$h:=\hat{B}_{q}\left(\hat{f}\right) \in  \mathbb{H}_{q,1}^{d}$). By a straightforward computation, we find that~$\mathcal{L}_{q,1}^{[d]}(g)$ and~$L_{q}^{[d]}(h)$ are respectively equal to the two ``discrete''~$q$-Borel-Laplace summation defined in \cite{DVZ}, Definition 4.12, (1).
\end{rem}

We can compare the two~$q$-Borel-Laplace summation processes for formal power series solutions of a linear~$\sq$-equation with coefficients in~$\C(\{z\})$ with only slope~$1$. From \cite{DVZ}, Theorem 4.14,  and Remark~\ref{4rem5}, we deduce directly the following: 

\pagebreak[3]
\begin{theo}\label{4theo7}
Let~$\hat{h}(z)\in \C[[z]]$ be a formal power series solution of a linear~$\sq$-equation with coefficients in~$\C(\{z\})$ with only slope~$1$ and let~$d\in \R$. Then, the series ~$\hat{B}_{q}\left(\hat{h}\right)$  converges and admits an analytic continuation~$f\in \mathbb{H}_{q,1}^{d}$ if and only if~$\hat{\mathcal{B}}_{q,1}\left(\hat{h}\right)$ converges and admits an analytic continuation~$g\in \mathbb{H}_{q,1}^{d}$. Moreover for such a~$d\in \R$,~${L_{q}^{[d]}(f)=\mathcal{L}_{q,1}^{[d]}(g)}$ on a convenient domain.
\end{theo}

\newpage
\section{Statement of the main result.}\label{4sec4}
From now, we see~$q$ as a parameter in~$]1,\infty[$. We recall that when we say that~$q$ is close to~$1$, we mean that~$q$ will be in the neighborhood of~$1$ in~$]1,\infty[$.  In~$\S \ref{4sec41}$, we prove two preliminaries lemmas that deal with the confluence of formal solutions of family of linear~$\sq$-equations. In~$\S \ref{4sec42}$, we state our main result.  We consider ~$\left(\hat{h}(z,q)\right)_{q>1}$ (resp.~$\widetilde{h}(z)$), formal power series solutions of a family of linear~$\dq$-equations (resp.~$\d$-equation) with coefficients in~$\C[z]$. We assume that~$\hat{h}(z,q)$ converges coefficientwise to~$\widetilde{h}(z)$ when~$q\to 1$. We state that under reasonable assumptions, for~$q$ close to~$1$, we may apply several~$q$-Borel and~$q$-Laplace transformations to~$\hat{h}(z,q)$, and obtain a solution of the family of linear $\dq$-equations, that is for~$q$ fixed, meromorphic on some punctured neighborhood of~$0$ in~$\C^{*}$. Moreover, the latter converges as~$q$ goes to~$1$, to the solution of the linear~$\d$-equation, computed with the classical Borel and Laplace transformations.

\pagebreak[3]
\subsection{Preliminaries on confluence of formal solutions.}\label{4sec41}

\pagebreak[3]
\begin{lem}\label{4lem9}
Let us consider 
$$\left\{\begin{array}{lllllllll}
\D_{q}&:=&b_{m}(z,q)\d_{q}^{m}&+&b_{m-1}(z,q)\d_{q}^{m-1}&+&\dots&+&b_{0}(z,q)\\\\
\widetilde{\D}&:=&\widetilde{b}_{m}(z)\d^{m}&+&\widetilde{b}_{m-1}(z)\d^{m-1}&+&\dots&+&\widetilde{b}_{0}(z),
\end{array}\right.$$
 with~$z\mapsto b_{i}(z,q),\widetilde{b}_{i}(z)\in \C[[z]]$, and  the ~$b_{i}$ converge coefficientwise to the~$\widetilde{b}_{i}$ when~$q\to 1$. We assume that the~$\C$-vector subspace~$\widetilde{F}\subset\C((z))$, of solutions of~$\widetilde{\D}\left(\widetilde{F}\right)=0$ has dimension~$1$. Let~$\k\in \Z$ be the~$z$-valuation of the elements of~$\widetilde{F}\setminus \{0\}$. Let~${\hat{h}(z,q):=\displaystyle\sum_{n=\k}^{\infty} \hat{h}_{n}(q)z^{n}}$ be a solution of~$\D_{q}\left(\hat{h}\right)=0$, such that~$\lim\limits_{q \to 1}\hat{h}_{\k}(q)=\widetilde{h}_{\k}\neq 0$. Let~${\widetilde{h}(z):=\displaystyle\sum_{n=\k}^{\infty} \widetilde{h}_{n}z^{n}\in\widetilde{F}\setminus \{0\}}$, which is uniquely determined by assumption. Then, for all~$n\geq \k$,
$$\lim\limits_{q \to 1}\hat{h}_{n}(q)=\widetilde{h}_{n}.$$
\end{lem} 

\begin{proof}
We will prove by an induction on~$n$ that for all~$n\geq \k$,~$\hat{h}_{n}(q)$ converges as~$q$ goes to~$1$ to~$\widetilde{h}_{n}$.
By assumption,~$\hat{h}_{\k}(q)$ converges to~$\widetilde{h}_{\k}$.\par 
 Let~$n\geq \k$. Induction hypothesis: assume that for all~$k\in \{\k,\dots,n-1\}$,~$\lim\limits_{q \to 1}\hat{h}_{k}(q)=\widetilde{h}_{k}$. Let us prove that~$\hat{h}_{n}(q)$ converges to~$\widetilde{h}_{n}$. Looking at the linear~$\sq$-equation (resp. the linear~$\d$-equation) satisfied by~$\hat{h}(z,q)$ (resp.~$\widetilde{h}(z)$), we find a relation of the form:
$$\begin{array}{lllllllll}
c_{n}(q)\hat{h}_{n}(q)&=&c_{n-1}(q)\hat{h}_{n-1}(q)&+&\dots&+&c_{\k}(q)\hat{h}_{\k}(q),\\
\widetilde{c}_{n}\widetilde{h}_{n}&=&\widetilde{c}_{n-1}\widetilde{h}_{n-1}&+&\dots&+&\widetilde{c}_{\k}\widetilde{h}_{\k}, \end{array}$$
where~$c_{i}(q),\widetilde{c}\in \C$.
Since the~$b_{i}$ converge coefficientwise to the~$\widetilde{b}_{i}$ when~$q\to 1$, we find that for all~${k\in \{\k,\dots,n\}}$,~$\lim\limits_{q \to 1}c_{k}(q)=\widetilde{c}_{k}$.\par 
If ~$\widetilde{c}_{n}=0$, then we obtain a formal solution of the same linear~$\d$-equation than~$\widetilde{h}$ with~$z$-valuation equal to~$n$. This is in contradiction with the assumptions of the lemma. Therefore,~$\widetilde{c}_{n}\neq 0$. Using the convergence of~$c_{n}(q)$ to~$\widetilde{c}_{n}$,~$c_{n}(q)$ is not vanishing in the neighborhood of~$1$. Because of the induction hypothesis and the convergence of the~$c_{i}(q)$, we obtain
$$\lim\limits_{q \to 1}\hat{h}_{n}(q)=\widetilde{h}_{n}.$$ 
By  induction, we have proved that for all~$n\geq \k$,~$\hat{h}_{n}(q)$ converges as~$q$ goes to~$1$ to~$\widetilde{h}_{n}$.
\end{proof}

If~$A$ and~$B$ are matrices with coefficients in~$\C$ and~$R\in \R_{>0}$, we say that~$|A|<|B|$ (resp.~$|A|<R$) if every entry of~$A$ has modulus bounded by the modulus of the corresponding entry of~$B$ (resp. by~$R$).\par 
Following~$\S 3.3.1$ of \cite{S00}, we prove:

\pagebreak[3]
\begin{lem}\label{4lem11}
Let us consider~$z\mapsto\hat{h}(z,q),\widetilde{h}(z)\in \C\{z\}$, solution of 
$$\left\{\begin{array}{lllllllllll}
b_{m}(z,q)\d_{q}^{m}\hat{h}(z,q)&+&b_{m-1}(z,q)\d_{q}^{m-1}\hat{h}(z,q)&+&\dots&+&b_{0}(z,q)\hat{h}(z,q)&=&0\\\\
\widetilde{b}_{m}(z)\d^{m}\widetilde{h}(z)&+&\widetilde{b}_{m-1}(z)\d^{m-1}\widetilde{h}(z)&+&\dots&+&\widetilde{b}_{0}(z)\widetilde{h}(z)&=&0,
\end{array}\right.$$
with~$z\mapsto b_{i}(z,q),\widetilde{b}_{i}(z)\in \C[z]$ and assume that
\begin{itemize}
\item The~$b_{i}$ converge coefficientwise to the~$\widetilde{b}_{i}$ when~$q\to 1$.
\item The series~$\hat{h}$ converges coefficientwise to~$\widetilde{h}$ when~$q\to 1$.
\end{itemize}
Then, we have
$$\lim \limits_{q \to 1} \hat{h}(z,q)=\widetilde{h}(z),$$
uniformly on a closed disk centered at~$0$.
\end{lem}

\begin{proof}
Let us consider the equations as systems:
$$\dq Y(z,q)=B(z,q)Y(z,q) \hbox{ and } \d\widetilde{Y}(z)=\widetilde{B}(z)\widetilde{Y}(z).$$
Let~$\k\in \Z$ and let us write the vector solutions~$Y(z,q)=:\displaystyle\sum_{k=\k}^{\infty} Y_{k}(q)z^{k}$,~$\widetilde{Y}(z)=:\displaystyle\sum_{k=\k}^{\infty} \widetilde{Y}_{k}z^{k}$ and the matrices~$B(z,q)=:\displaystyle\sum_{k=\k}^{\infty}B_{k}(q)z^{k}$,~$\widetilde{B}(z)=:\displaystyle\sum_{k=\k}^{\infty}\widetilde{B}_{k}z^{k}$. For all~$k\geq \k$, we have the relation:
\begin{equation}\label{4eq20}
\Big([k]_{q}\times\mathrm{Id}-B_{0}(q)\Big)Y_{k}(q)=\sum_{i\neq k}B_{i}(q)Y_{k-i}(q) \hbox{ and } \left(k\times\mathrm{Id}-\widetilde{B}_{0}\right)\widetilde{Y}_{k}=\sum_{i\neq k}\widetilde{B}_{i}\widetilde{Y}_{k-i}.
\end{equation}
There exist~$k_{0}\geq \k$,~$C\in \R_{>0}$, such that for all~$k\geq k_{0}$, for all~$q$ close to~$1$, for all~$Y\in \C^{m}$,
$$\Big([k]_{q}\times\mathrm{Id}-B_{0}(q)\Big)\in \mathrm{GL}_{m}(\C) \hbox{ and } \left|\Big([k]_{q}\times\mathrm{Id}-B_{0}(q)\Big)^{-1}Y\right|
=\left|\displaystyle\sum_{l=0}^{\infty}\left([k]_{q}\right)^{-1}\left(\frac{B_{0}(q)}{[k]_{q}}\right)^{l}Y\right|<C|Y|$$ resp.
$$\left(k\times\mathrm{Id}-\widetilde{B}_{0}\right)\in \mathrm{GL}_{m}(\C)\hbox{ and }  \left|\left(k\times\mathrm{Id}-\widetilde{B}_{0}\right)^{-1}Y\right|=\left|\displaystyle\sum_{l=0}^{\infty}k^{-1}\left(\frac{\widetilde{B}_{0}}{k}\right)^{l}Y\right|<C|Y|.$$

Since the equations have coefficients in~$\C[z]$, the first assumption implies the existence of~$C_{0}>0$ such that for all~$k\geq \k$, for all~$q$ close to~$1$,~$\left|B_{k}(q)\right|<C_{0}^{k}$ and~$\left|\widetilde{B}_{k}(q)\right|<C_{0}^{k}$.
Using additionally (\ref{4eq20}), we can prove by an induction that there exists~$C_{1}>0$, such that for all~$k\geq \k$, for all~$q$ close to~~$1$, we have:
~$$\left| Y_{k}(q)\right|=\left|\Big([k]_{q}\times\mathrm{Id}-B_{0}(q)\Big)^{-1}\sum_{i\neq k}B_{i}(q)Y_{k-i}(q)\right|<C_{1}^{k}$$
and 
~$$ \left| \widetilde{Y}_{k}\right|=\left|\left(k\times\mathrm{Id}-\widetilde{B}_{0}\right)^{-1}
\sum_{i\neq k}\widetilde{B}_{i}\widetilde{Y}_{k-i}\right|<C_{1}^{k}.$$
Using the dominated convergence theorem, and the second assumption of the lemma, we obtain the result.
\end{proof}

\pagebreak[3]
\subsection{Confluence of a ``discrete''~$q$-Borel-Laplace summation.}\label{4sec42}
The goal of the subsection is to state our main result, Theorem~\ref{4theo1}. See~$\S \ref{4sec5},\S\ref{4sec6}$ for the proof. We begin with a definition.

\pagebreak[3]
\begin{defi}\label{4defi1}
Let~$d\in \R$ and let~$k\in \Q_{>0}$.
Let~$f$ be a function such that there exists~$\e>0$, such that for~$q$ close to~$1$,~$z\mapsto f(z,q)\in\mathcal{A}(d-\e,d+\e)$. We say that~$f$ belongs to~$\overline{\mathbb{H}}_{k}^{d}$, if for~$q$ close to~$1$,~$z\mapsto f(z,q)$ admits an analytic continuation defined on~$\overline{S}(d-\e,d+\e)$, that we will still call~$f$, such that there exist constants~$J,L>0$, that do not depend upon~$q$, such that for all~$z\in \R_{>0}$:
~$$\left|f\left(e^{id}z,q\right)\right|<Je_{q}\left(Lz^{k}\right).$$
\end{defi}
Let us consider~$z\mapsto \hat{h}(z,q)\in \C[[z]]$, that converges coefficientwise to~$\widetilde{h}(z)\in \C[[z]]$ when~$q\to 1$. We make the following assumptions:
\begin{trivlist}
\item \textbf{(A1)}  There exist
$$z\mapsto b_{0}(z,q),\dots,b_{m}(z,q)\in\C[z],$$ with~$z$-coefficients that converge as~$q$ goes to~$1$, such that for all~$q$ close to~$1$,~$\hat{h}(z,q)$ is solution of:
\begin{equation}\label{4eq3}
b_{m}(z,q)\d_{q}^{m}(y(z,q))+\dots+b_{0}(z,q)y(z,q)=0.
\end{equation}
Let~$\widetilde{b}_{0}(z),\dots,\widetilde{b}_{m}(z)\in\C[z]$, be the limit as~$q$ tends to~$1$ of the~$b_{0}(z,q),\dots,b_{m}(z,q)$.
\item \textbf{(A2)} For~$q$ close to~$1$, the  slopes of (\ref{4eq3}) are independent of~$q$, and the set of slopes of (\ref{4eq3}) that are positive coincides with the set of  slopes of 
\begin{equation}\label{4eq4}
\widetilde{b}_{m}(z)\d^{m}\left(\widetilde{y}(z)\right)+\dots+\widetilde{b}_{0}(z)\widetilde{y}(z)=0.
\end{equation}
Notice that the series~$\widetilde{h}(z)$ is solution of (\ref{4eq4}).
\item \textbf{(A3)} There exists~$c_{1}>0$, such that for all~$i\leq m$ and~$q$ close to~$1$:
$$\left|b_{i}(z,q)-\widetilde{b}_{i}(z)\right|<(q-1)c_{1}\left(\left|\widetilde{b}_{i}(z)\right|+1\right).$$ 
\end{trivlist}
\pagebreak[3]
\begin{rem}\label{4rem7}
\begin{trivlist}
\item (1) Conversely, given equations like (\ref{4eq3}) and (\ref{4eq4}) that satisfies the assumptions \textbf{(A2)} and \textbf{(A3)}, we would like to know if there exists~$z\mapsto \hat{h}(z,q)\in \C[[z]]$, solution of (\ref{4eq3}), which converges coefficientwise to~$\widetilde{h}(z)\in \C[[z]]$, solution of (\ref{4eq4}). The answer is in general no, but Lemma~\ref{4lem9} gives a sufficient condition. 
\item (2) If for~$q$ close to~$1$, the only slope of (\ref{4eq4}) is~$0$ then,~$z\mapsto \hat{h}(z,q),\widetilde{h}(z)\in \C\{z\}$ and we set for all~${d\in \R}$,~$S_{q}^{[d]}\left(\hat{h}\right):=\hat{h}$. Remember that we have set in~$\S \ref{4sec1}$, 
~$\widetilde{S}^{d}\left(\widetilde{h}\right):=\widetilde{h}$. In this particular case, applying Lemma~\ref{4lem11}, we obtain that 
$$
\lim\limits_{q \to 1}S_{q}^{[d]}\left(\hat{h}\right)
=\widetilde{S}^{d}\left(\widetilde{h}\right),
$$
uniformly on a closed disk centered at~$0$. 
\end{trivlist}
\end{rem}
From now, we are going to assume that (\ref{4eq4}) has at least one slope strictly bigger than~$0$.
Let~${d_{0}:=\max\left(2,\deg \widetilde{b}_{0},\dots, \deg \widetilde{b}_{m}\right)}$.
Let~${k_{1}<\dots<k_{r-1}}$ be the slopes of (\ref{4eq4}) different from $0$, let~$k_{r}$ be an integer strictly bigger than~$k_{r-1}$ and~$d_{0}$, and set~$k_{r+1}:=+\infty$. 
Let~$(\k_{1},\dots,\k_{r})$ defined as:
$$\k_{i}^{-1}:=k_{i}^{-1}-k_{i+1}^{-1}.$$  
 As in Proposition~\ref{4propo5}, we define the~$(\widetilde{\k_{1}},\dots,\widetilde{\k_{s}})$ as follows:
we take~$(\k_{1},\dots,\k_{r})$ and  for~$i=1,\dots,i=r$, replace successively~$\k_{i}$ by~$\a_{i}$ terms~$\a_{i}\k_{i}$, where~$\a_{i}$ is the smallest integer such that~$\a_{i}\k_{i}$ is greater or equal than~$d_{0}$. See Example~\ref{4ex1}. Therefore, by construction, each of the~$\widetilde{\k_{i}}$ is a rational number greater than~$d_{0}$ and $\widetilde{\k_{s}}\in \N^{*}$.\par 
 Let~$\b\in \N^{*}$ be minimal, such that for all~$i\in \{1,\dots,s\}$,~$\b/\widetilde{\k_{i}}\in \N^{*}$.
Let us write~${\hat{h}(z,q)=:\displaystyle \sum_{n=0}^{\infty}\hat{h}_{n}(q)z^{n}}$ and, for~$l\in\{0,\dots,\b-1\}$, let~$\hat{h}^{(l)}(z,q):=\displaystyle \sum_{n=0}^{\infty} \hat{h}_{l+n\b}(q)z^{n\b}$. \par 
 The main result of the paper is the following. See~$\S \ref{4sec1},\S \ref{4sec3}$ for the notations, and~$\S \ref{4sec5},\S\ref{4sec6}$ for the proof. See also Theorem~\ref{4theo9} in the appendix for a similar result with a ``continuous''~$q$-Laplace transformation. We recall that the series~$\hat{h},\widetilde{h}$ satisfies the assumptions \textbf{(A1)} to \textbf{(A3)}.

\pagebreak[3]
\begin{theo}\label{4theo1}
There exists~$\Sigma_{\widetilde{h}}\subset \R$ finite modulo~$2\pi \Z$, that contains the set of singular directions defined in Proposition~\ref{4propo5}, such that if~$d\in  \R\setminus \Sigma_{\widetilde{h}}$ and~$l\in\{0,\dots,\b-1\}$, then 
the series~${g_{1,l}:=\hat{\mathcal{B}}_{q,\widetilde{\k_{1}}}\circ\dots\circ\hat{\mathcal{B}}_{q,\widetilde{\k_{s}}}\left(\hat{h}^{(l)}\right)}$  converges and belongs to~$\overline{\mathbb{H}}_{\widetilde{\k_{1}}}^{d}$.\par 
Moreover, for~$j=2$ (resp.~$j=3$,~$\dots$, resp.~$j=s$), ~$g_{j,l}:=\mathcal{L}_{q,\widetilde{\k_{j-1}}}^{[d]}(g_{j-1,l})$ belongs to~$\overline{\mathbb{H}}_{\widetilde{\k_{j}}}^{d}$.
 Let~${S_{q}^{[d]}\left(\hat{h}^{(l)}\right):=\mathcal{L}_{q,\widetilde{\k_{s}}}^{[d]}(g_{r,l})}$. 
The function
$$S_{q}^{[d]}\left(\hat{h}\right):=\displaystyle \sum_{l=0}^{\b-1} z^{l}S_{q}^{[d]}\left(\hat{h}^{(l)}\right)
\in\mathcal{A}\left(d-\frac{\pi}{k_{r}},d+\frac{\pi}{k_{r}}\right)$$ is solution of (\ref{4eq3}). 
Furthermore, we have $$\lim\limits_{q \to 1}S_{q}^{[d]}\left(\hat{h}\right)=\widetilde{S}^{d}\left(\widetilde{h}\right),$$
uniformly on the compacts of~$\overline{S}\left(d-\frac{\pi}{2k_{r}},d+\frac{\pi}{2k_{r}}\right)\setminus \displaystyle\bigcup \R_{\geq 1}\a_{i}$, where~$\a_{i}$ are the roots of~$\widetilde{b}_{m}\in \C[z]$,  and~$\widetilde{S}^{d}\left(\widetilde{h}\right)$ is the asymptotic solution of (\ref{4eq4}) that has been defined in Proposition~\ref{4propo5}.
\end{theo}

\pagebreak[3]
\begin{rem}\label{4rem1}
After some arrangements, we could probably state and show a similar result for~$q$ not real. As \cite{S00}, we should make~$q$ goes to~$1$ following a~$q$-spiral of the form~$\left\{q_{0}^{\l}, \l\in\R_{>0}\right\}$, for some~$q_{0}\in \C$ fixed with modulus strictly bigger than~$1$. The problem here, is that we would obtain at the limit, a solution of the differential equation that is not classic, since at the limit, we would obtain integrals of the form~$\displaystyle \int_{q_{0}^{\R}e^{id}} z^{-k}f(\z)e^{-\left(\frac{\z}{z}\right)^{k}}d\z^{k}$, instead of Laplace transformations. In order to interpret the limit as the classical Borel-Laplace summation, we have to consider~$q$ real.
\end{rem}

\pagebreak[3]
\begin{rem}\label{4rem8}
A confluence result of this nature can also be found in \cite{DVZ}, Theorem~2.6. We are going now to state \cite{DVZ}, Corollary 2.9, which is the particular case where the coefficients of the family of linear $q$-difference equations do not depend upon $q$. Let~$p=1/q$ and let~$\delta_{p}:=\frac{\sq^{-1}-\mathrm{Id}}{p-1}$, which converges formally to~$\d$ when~$p\to 1$. Let~${z\mapsto \hat{h}(z,q)\in \C\{z\}}$ that converges coefficientwise to~$\widetilde{h}(z)\in \C[[z]]$ when~$p\to 1$. Assume the existence of~$b_{0},\dots,b_{m}\in \C[z]$, such that for all~$p$ close to~$1$, we have
$$
\left \{ \begin{array}{lllll} 
b_{m}(z)\delta_{p}^{m}\hat{h}(z,q)&+&\dots+b_{0}(z)\hat{h}(z,q)&=&0\\\\
b_{m}(z)\d^{m}\widetilde{h}(z)&+&\dots+b_{0}(z)\widetilde{h}(z)&=&0.
\end{array}\right.$$
Moreover, assume that the series~$\hat{\mathcal{B}}_{1}\left(\widetilde{h}\right)$ belongs to~$\C\{z\}$ and is solution of a linear differential equation which is Fuchsian at~$0$ and infinity and has non resonant exponents at~$\infty$.\par 
Let~$\widetilde{\Sigma}_{\widetilde{h}}\subset \R$ be the set of singular directions that has been defined in Proposition \ref{4propo5}. The authors of \cite{DVZ} conclude that for all~$d\notin \widetilde{\Sigma}_{\widetilde{h}}$, the series 
~$\hat{\mathcal{B}}_{1}\left(\widetilde{h}\right)$ belongs to~$ \widetilde{\mathbb{H}}_{1}^{d}$, and $$\lim\limits_{p \to 1}\hat{h}(z,q)=\widetilde{S}^{d}\left(\widetilde{h}\right)(z),$$
uniformly on the compacts of~$\overline{S}\left(d-\frac{\pi}{2},d+\frac{\pi}{2}\right)$, where~$\widetilde{S}^{d}\left(\widetilde{h}\right)$ is the asymptotic solution of the linear differential equation that has been defined in Proposition~\ref{4propo5}. Notice that Theorem \ref{4theo1} and this theorem have not the same setting, since we consider~$\dq$-equations and not~$\delta_{p}$-equations. In particular, in our case~$z\mapsto \hat{h}(z,q)$ might be divergent and we have to replace~$\hat{h}$ by~$S_{q}^{[d]}\left(\hat{h}\right)$ in order to have the convergence.
\end{rem}

%référence boulot changguy ici

\pagebreak[3]
\section{Lemmas on meromorphic solutions.}\label{4sec5}
The goal of this section is to prove lemmas on meromorphic solutions that will be used in the proof of Theorem~\ref{4theo1} in~$\S \ref{4sec6}$. See~$\S\ref{4sec4}$ for the notations. If~$D(z)\in \mathrm{GL}_{m}\Big(\C(z)\Big)$, we define~$ \mathbf{S}^{q} (D(z))$ as the union of the~$q^{\N^{*}}x_{i}$, where the~$x_{i}$ are the poles of~$D(z)$ or~$D(z)^{-1}$. 

\pagebreak[3]
\begin{lem}\label{4lem4} 
Let~$a<b$. Let us consider~$\sq M(z)=D(z)M(z)$ with~$D(z)\in \mathrm{GL}_{m}\Big(\C(z)\Big)$ and~$M(z)$ is s solution in $\Big(\mathcal{A}(a,b)\Big)^{m}$.
Then, the entries of~$M(z)$ are meromorphic on~$\overline{S}(a,b)$, with poles contained in ~$\mathbf{S}^{q} (D(z))$.
\end{lem}

\begin{proof}
 Let~$z\in\C^{*}\setminus\mathbf{S}^{q} (D(z))$. We use the fact that~$M(qz)=D(z)M(z)$ to deduce that if the entries of~$M$ are analytic on a domain $U$, then there are analytic  on the domain $qU:=\{qz,z\in U\}$. We use the existence of~$\e>0$, such that the entries of~$M(z)$ are analytic for all~$|z|<\e$ and~$z\in \overline{S}(a,b)$, to obtain that the entries of~$M(z)$ are meromorphic on~$\overline{S}(a,b)$, with poles contained in~$\mathbf{S}^{q} (D(z))$.
\end{proof}

If~$\widetilde{D}(z)\in \mathrm{M}_{m}\Big(\C(z)\Big)$, we define ~$ \mathbf{S}^{1} \left(\widetilde{D}(z)\right)$ as the union of the~$\R_{\geq 1}x_{i}$, where the~$x_{i}$ are the poles of~$\widetilde{D}(z)$. We define also~$\R_{>0}[z]$ as the set of polynomials with coefficients that are strictly positive real numbers. We recall that if~$A$ and~$B$ are matrices with coefficients in~$\C$ and~$R\in \R_{>0}$, we say that~${|A|<|B|}$ (resp.~$|A|<R$) if every entry of~$A$ has modulus bounded by the modulus of the corresponding entry of~$B$ (resp. by~$R$).

\pagebreak[3] 
\begin{propo}\label{4propo3}
Let~$a<b$,~${z\mapsto \mathrm{Id}+(q-1)D(z,q)\in\mathrm{GL}_{m}\Big(\C(z)\Big),\widetilde{D}(z)\in\mathrm{M}_{m}\Big(\C(z)\Big)}$ and let~$C$ be a convex set with non empty interior contained in~$\overline{S}(a,b)\setminus  \mathbf{S}^{1}\left( \widetilde{D}(z)\right)$ such that $0$ does not belong to its closure. Let us consider~${z\mapsto M\left(z,q\right),\widetilde{M}\left(z\right)}$,~$1\times m$ matrices with entries continuous on $C$ and analytic in the interior of~$C$, solutions of 
$$\left\{\begin{array}{lll}
\dq M(z,q)&=&D(z,q)M(z,q)\\ \\
\d \widetilde{M}(z)&=&\widetilde{D}(z)\widetilde{M}(z).
\end{array}\right.$$
We assume that:
\begin{trivlist}
\item~$(i)$ There exists~$c_{1}>0$, such that for all~$q$ close to~$1$ and for all ~$z\in C$,
$$\left|D(z,q)-\widetilde{D}(z)\right|<(q-1)c_{1}\left(\left|\widetilde{D}(z)\right|+\left|\mathbf{1}_{m}\right|\right),$$
where~$\mathbf{1}_{m}$ denotes the square matrix of size~$m$ with~$1$ entries everywhere.
Notice that this condition implies that for~$q$ close to~$1$, the entries of~$D(z,q)$ have no poles in~$C$.
\item~$(ii)$ There exists~$w_{0}\in C$, such that for all~$q$ close to~$1$,~$M\left(w_{0},q\right)=\widetilde{M}\left(w_{0}\right)$. Moreover, we have~$\lim \limits_{q \to 1} 
M(w,q)=\widetilde{M}\left(w\right)$ uniformly on a compact~$K$ contained in~$C$. 
\item~$(iii)$ There exists~$R\in \C[z]$, such that for all~$z\in C$,~$\left|\widetilde{M}\left(z\right)\right|<\left|R\left(z\right)\right|$. 
\end{trivlist}
Let~$\k$ be the maximum of the degrees of the numerators and the denominators of the entries of ~$\widetilde{D}(z)$, written as the quotient of two coprime polynomials. Let
~$S\in \R_{>0}[z]$ be a polynomial of degree~$\k$, such that for all~$z\in C$,~${S(|z|)>\left|\widetilde{D}(z)\right|+\left|\mathbf{1}_{m}\right|}$.
Under those assumptions, there exist
\begin{itemize}
\item~$\d(q)>0$ that converges to~$1$ as~$q\to 1$,
\item~$\e(q)>0$ that converges to~$0$ as~$q\to 1$,
\item~$S_{0}\in \R_{>0}[z]$ which has degree~$\k$ and satisfies~$S_{0}(|z|)>S(|z|)$ for all~$z\in C$,
\end{itemize}
 such that for all~${z\in C\cap \R_{\geq 1}K:=\left\{ xw\in C\Big|x\in [1,\infty[,w\in K\right\}}$, we have
$$
\left|M(z,q)-\widetilde{M}(z)\right|<(q-1)\d(q)e_{q^{\k}}\Big(S_{0}\left(\left|z\right|\right)\Big)
+\e(q)\left|R(z)\right|.$$
In particular,
$$\lim \limits_{q \to 1} M\left(z,q\right)=\widetilde{M}\left(z\right),$$
uniformly on the compacts of~${C\cap \R_{\geq 1}K}$. 
\end{propo}
\pagebreak[3]
\begin{rem}
The polynomial~$S_{0}$ does not depend upon~$w$ and~$q$.
\end{rem}
Before proving the proposition, we need to prove a technical lemma.
\pagebreak[3]
\begin{lem}\label{4lem10}
Let~${z\mapsto \mathrm{Id}+\left(q-1\right)D(z,q)\in\mathrm{GL}_{m}\Big(\C(z)\Big),\widetilde{D}(z)\in\mathrm{M}_{m}\Big(\C(z)\Big)}$ that satisfies assumption~$(i)$ of Proposition~\ref{4propo3}. Let~$C$ be the corresponding convex set and let~$K$ be the corresponding compact set defined in Proposition~\ref{4propo3}.
 Let~${\Big(z\mapsto M_{w}\left(z,q\right)\Big)_{w\in K},\left(\widetilde{M}_{w}(z)\right)_{w\in K}}$, be a family of~$1\times m$ matrices with entries continuous on $C$ and analytic in the interior of~$C$, solutions of 
$$\left\{\begin{array}{lll}
\dq M_{w}(z,q)&=&D(z,q)M_{w}(z,q)\\ \\
\d \widetilde{M}_{w}(z)&=&\widetilde{D}(z)\widetilde{M}_{w}(z).
\end{array}\right.$$
We assume that the matrices~$\Big(M_{w}(z,q)\Big)_{w\in K},\left(\widetilde{M}_{w}(z)\right)_{w\in K}$ satisfy:
\begin{trivlist}
\item \hspace*{1cm}(a) For all~$q$ close to~$1$, for all~$w\in K$,~$M_{w}\left(w,q\right)=\widetilde{M}_{w}\left(w\right)$.
\item \hspace*{1cm}(b) There exists~$R\in \C[z]$, such that for all~$z\in C$, for all~$w\in K$:
$$\left|\widetilde{M}_{w}\left(z\right)\right|<\left|R\left(z\right)\right|.$$ 
\end{trivlist}
 Under those assumptions, there exists a polynomial~$S_{0}$ that satisfies the same properties than the one in Proposition~\ref{4propo3}, such that for all~$w\in K$, for all~$q$ close to~$1$, for all~$N\in \N$ with~$q^{N}w\in C$:
\begin{equation}\label{4eq6}
\left|\frac{M_{w}\left(q^{N}w,q\right)-\widetilde{M}_{w}\left(q^{N}w\right)}{q-1}\right|<e_{q^{\k}}\Big(S_{0}\left(\left|q^{N}w\right|\right)\Big).
\end{equation}
\end{lem}

\begin{proof}[Proof of Lemma~\ref{4lem10}]
For the reader's convenience, we will decompose the proof in four steps. 
\vspace{1 mm}\pagebreak[3]
\begin{center}
\textbf{Step 1: Find another expression of~$\frac{\widetilde{M}_{w}\left(q^{n}w\right)-M_{w}\left(q^{n}w,q\right)}{q-1}$}.\end{center}
\vspace{1 mm}
Let~$f$ be a function continuous on~$C$, that is analytic in the interior of $C$, and let~$z_{0},z_{1}\in C$. The generalized mean value theorem (see~$\S 1.4$ of \cite{KM}) says that there exists~$c\in C$ that belongs to the convex hull of 
$$
\Big\{f'\big(z_{0}+x(z_{1}-z_{0})\big)\Big|x\in \left[0,1\right] \Big\} ,
$$
such that:
$$\frac{f(z_{1})-f(z_{0})}{z_{1}-z_{0}}=c.$$
For all~$q>1,w\in K,n\in \N$ with~$q^{n}w\in C$, let us define the~$A_{w,q,n-1}$ as the convex hull of
$$
\left\{\left.\frac{\widetilde{D}\left(q^{n-1}wx\right)\widetilde{M}\left(q^{n-1}wx\right)}{q^{n-1}wx}\right|x\in \big[1,q\big] \right\} .
$$
Because of the generalized mean value theorem, for all~$n\in \N$, for all~$q>1$, for all~$w\in K$, with~$q^{n}w\in C$, there exists~$\widetilde{D}_{w,q,n-1}$ that belongs to~$A_{w,q,n-1}$, 
such that:
$$\frac{\widetilde{M}_{w}\left(q^{n}w\right)-\widetilde{M}_{w}\left(q^{n-1}w\right)}{q^{n-1}w(q-1)}=\widetilde{D}_{w,q,n-1}.$$ 
The linear~$\dq$-equation satisfied by~$M_{w}(z,q)$ gives that for all~$n\in \N$, for all~$q>1$, for all~$w\in K$, with~$q^{n}w\in C$:
$$\frac{M_{w}\left(q^{n}w,q\right)-M_{w}\left(q^{n-1}w,q\right)}{q-1}= D\left(q^{n-1}w,q\right)M_{w}\left(q^{n-1}w,q\right).$$
In particular, we have 
\begin{equation}\label{4eq23}
\begin{array}{lll}
\frac{\widetilde{M}_{w}\left(q^{n}w\right)-M_{w}\left(q^{n}w,q\right)}{q-1}&=&\frac{\widetilde{M}_{w}\left(q^{n-1}w\right)-M_{w}\left(q^{n-1}w,q\right)}{q-1}\\\\
&+&
q^{n-1}w\widetilde{D}_{w,q,n-1}-D\left(q^{n-1}w,q\right)M_{w}\left(q^{n-1}w,q\right).
\end{array}
\end{equation}

\begin{center}\pagebreak[3]
\textbf{Step 2: Bound the expression of~$\frac{\widetilde{M}_{w}\left(q^{n}w\right)-M_{w}\left(q^{n}w,q\right)}{q-1}$.}\end{center}
\vspace{1 mm}
Let~$q_{0}>1$ sufficiently close to~$1$. Let us prove the existence of~$R_{1},R_{2}\in \C[z]$, such that for all~$n\in \N$,~$q\in]1,q_{0}[$,~$w\in K$, with~$q^{n}w\in C$, 
\begin{equation}\label{4eq16}
\begin{array}{llll}
&\left|\frac{\widetilde{M}_{w}\left(q^{n}w\right)-M_{w}\left(q^{n}w,q\right)}{q-1} \right|&&\\\\
\leq&\left|\frac{\widetilde{M}_{w}\left(q^{n-1}w\right)-M_{w}\left(q^{n-1}w,q\right)}{q-1}\right| 
&+&(q-1)\Big(\left|R_{1}\left(q^{n}w\right)\right|+\left|R_{2}\left(q^{n}w\right)\right|\Big)\\\\
+&\left|\frac{\widetilde{M}_{w}\left(q^{n-1}w\right)-M_{w}\left(q^{n-1}w,q\right)}{q-1}\right|&\times&
(q-1)(1+(q-1)c_{1})mS\left(\left|q^{n}w\right|\right),
\end{array}
\end{equation}
where~$S,c_{1}>0$ are given by Proposition~\ref{4propo3}. 
Using the triangular inequality and (\ref{4eq23}), it is sufficient to prove the existence of~$R_{1},R_{2}\in \C[z]$, such that for all~$n\in \N$,~$q\in]1,q_{0}[$,~$w\in K$, with~$q^{n}w\in C$, 
\begin{equation}\label{4eq15}
\begin{array}{l}
\left|q^{n-1}w\widetilde{D}_{w,q,n-1}-D\left(q^{n-1}w,q\right)M_{w}\left(q^{n-1}w,q\right)\right|\\\\
\leq(q-1)\Big(\left|R_{1}\left(q^{n}w\right)\right|+\left|R_{2}\left(q^{n}w\right)\right|\Big)\\ \\
+\left|\frac{\widetilde{M}_{w}\left(q^{n-1}w\right)-M_{w}\left(q^{n-1}w,q\right)}{q-1}\right|\times
(q-1)(1+(q-1)c_{1})mS\left(\left|q^{n}w\right|\right).
\end{array}
\end{equation}

We have for all~$n\in \N$,~$q\in]1,q_{0}[$,~$w\in K$, with~$q^{n}w\in C$, 
$$\begin{array}{lll}
&\left|q^{n-1}w\widetilde{D}_{w,q,n-1}\right.-\left.D\left(q^{n-1}w,q\right)M_{w}\left(q^{n-1}w,q\right)\right|\\\\
\leq&\left|q^{n-1}w\widetilde{D}_{w,q,n-1}\right.-\left.\widetilde{D}\left(q^{n-1}w\right)\widetilde{M}_{w}\left(q^{n-1}w\right)\right|\\ \\
+&\left|\widetilde{D}\left(q^{n-1}w\right)\widetilde{M}_{w}\left(q^{n-1}w\right)\right.-\left.D\left(q^{n-1}w,q\right)\widetilde{M}_{w}\left(q^{n-1}w\right)\right|\\\\
+&\left|D\left(q^{n-1}w,q\right)\widetilde{M}_{w}\left(q^{n-1}w\right)\right.-D\left(q^{n-1}w,q\right)M_{w}\left(q^{n-1}w,q\right)\Big|&.
\end{array}$$
Let 
$$\begin{array}{lll}
\tau_{1}&:=&\left|q^{n-1}w\widetilde{D}_{w,q,n-1}\right.-\left.\widetilde{D}\left(q^{n-1}w\right)\widetilde{M}_{w}\left(q^{n-1}w\right)\right|,\\\\
\tau_{2}&:=&\left|\widetilde{D}\left(q^{n-1}w\right)\widetilde{M}_{w}\left(q^{n-1}w\right)\right.-\left.D\left(q^{n-1}w,q\right)\widetilde{M}_{w}\left(q^{n-1}w\right)\right|,\\\\
\tau_{3}&:=&\left|D\left(q^{n-1}w,q\right)\widetilde{M}_{w}\left(q^{n-1}w\right)\right.-D\left(q^{n-1}w,q\right)M_{w}\left(q^{n-1}w,q\right)\Big| .
\end{array}$$
Let us bound~$\tau_{1}$. The entries of~$q^{n-1}w\widetilde{D}_{w,q,n-1}$ and~$\widetilde{D}\left(q^{n-1}w\right)\widetilde{M}_{w}\left(q^{n-1}w\right)$ belong to the convex hull of~$
\left\{\left.\frac{\widetilde{D}\left(q^{n-1}wx\right)\widetilde{M}\left(q^{n-1}wx\right)}{x}\right|x\in [1,q] \right\}$. The entries of the elements of this set of matrices are bounded by a polynomial, because of the assumption~$(b)$ and the fact that the entries of~$\widetilde{D}$ are bounded by polynomials. This provides~$R_{1}\in \C[z]$, such that for all~$q\in]1,q_{0}[$, for all~$n\in \N$, for all~$w\in K$, with~$q^{n}w\in C$:
$$\tau_{1}=\left|q^{n-1}w\widetilde{D}_{w,q,n-1}-\widetilde{D}\left(q^{n-1}w\right)\widetilde{M}_{w}\left(q^{n-1}w\right)\right|<(q-1)
\left|R_{1}\left(q^{n}w\right)\right|.$$
Let us bound~$\tau_{2}$. Due to the assumptions~$(i)$ and~$(b)$, there exists~$R_{2}\in \C[z]$ such that for all~$q\in]1,q_{0}[$, for all~$n\in \N$, for all~$w\in K$, with~$q^{n}w\in C$:
$$\tau_{2}=\displaystyle \left|\left(\widetilde{D}\left(q^{n}w\right)-D\left(q^{n}w,q\right)\right)\widetilde{M}_{w}\left(q^{n}w\right)\right|<(q-1)
\left|R_{2}\left(q^{n}w\right)\right|.$$

Let us bound the quantity~$\tau_{3}$. By assumption~$(i)$ and the fact that for all~${z\in\C}$,~${\left|\widetilde{D}(z)\right|+\left|\mathbf{1}_{m}\right|<S\left(\left|z\right|\right)}$, we obtain that for all~$q\in]1,q_{0}[$,~$n\in \N$,~$w\in K$, with~$q^{n-1}w\in C$:

$$
\begin{array}{lll}
\tau_{3}&=&\left|D\left(q^{n-1}w,q\right)\left(\widetilde{M}_{w}\left(q^{n-1}w\right)-M_{w}\left(q^{n-1}w,q\right)\right)\right|\\\\
&\leq&\left(\left|\widetilde{D}\left(q^{n-1}w\right)\right|+(q-1)c_{1}\left(\left|\widetilde{D}\left(q^{n-1}w\right)\right|+\left|\mathbf{1}_{m}\right|\right)\right)\left|\widetilde{M}_{w}\left(q^{n-1}w\right)-M_{w}\left(q^{n-1}w,q\right)\right|\\\\
&\leq&(1+(q-1)c_{1})mS\left(\left|q^{n-1}w\right|\right)\left|\widetilde{M}_{w}\left(q^{n-1}w\right)-M_{w}\left(q^{n-1}w,q\right)\right|.
\end{array}
$$ 

Since the polynomial~$S$ has real positive coefficients, 
~$S\left(\left|q^{n-1}w\right|\right)\leq S\left(\left|q^{n}w\right|\right)$. In particular, for all~$q\in]1,q_{0}[$,~$n\in \N$,~$w\in K$ with~$q^{n-1}w\in C$:
$$\tau_{3}\leq (q-1)(1+(q-1)c_{1})mS\left(\left|q^{n}w\right|\right)\left|\frac{\widetilde{M}_{w}\left(q^{n-1}w\right)-M_{w}\left(q^{n-1}w,q\right)}{q-1}\right|.$$

This concludes the proof of (\ref{4eq15}) and yields (\ref{4eq16}), because of the triangular inequality.
\vspace{1 mm}\pagebreak[3]
\begin{center}
\textbf{Step 3: Construction of~$S_{0}$.}\end{center}
\vspace{1 mm}
We recall that ~$\k\in \N$ is the degree of~$S$. Before constructing~$S_{0}$, we are going to prove that for all~$b>0$ sufficiently big, for all~$z\in C\cap \R_{\geq 1}K$ and for all~$q$ close to~$1$

\begin{equation}\label{4eq10}
\begin{array}{ll}
e_{q^{\k}}\Big(b\left|z\right|^{\k}\Big)+(q-1)\Big(\left|R_{1}(qz)\right|+\left|R_{2}(qz)\right|\Big)&\\\\+ 
(q-1)(1+(q-1)c_{1})mS(q|z|)
e_{q^{\k}}\Big(b\left|z\right|^{\k}\Big)
&\leq e_{q^{\k}}\Big(b\left|qz\right|^{\k}\Big).
\end{array}\end{equation}

Using the~$q$-difference equation satisfied by the~$q$-exponential, we find that this inequality is equivalent to:
$$\begin{array}{lll}
1+(q-1)\frac{\left|R_{1}(qz)\right|+\left|R_{2}(qz)\right|}{e_{q^{\k}}\Big(b\left|z\right|^{\k}\Big)}+(q-1)(1+(q-1)c_{1})mS(q|z|)&\leq&
\frac{e_{q^{\k}}\Big(b\left|qz\right|^{\k}\Big)}{e_{q^{\k}}\Big(b\left|z\right|^{\k}\Big)}\\\\
&=&1+(q^{\k}-1)b\left|z\right|^{\k}.
\end{array}
$$
This inequality is equivalent to the following:
$$\frac{\left|R_{1}(qz)\right|+\left|R_{2}(qz)\right|}{e_{q^{\k}}\Big(b\left|z\right|^{\k}\Big)}+(1+(q-1)c_{1})mS(q|z|)\leq b[\k]_{q}\left|z\right|^{\k}.$$
Since~$R_{1},R_{2}$ are polynomials, for all~$b>0$ sufficiently big, for all~$q\in ]1,q_{0}[$ and for all~$z\in C\cap \R_{\geq 1}K$, this latter inequality is true. This proves (\ref{4eq10}).
\par 
We recall that by assumption, $0$ does not belong to the closure of~$C$. Using (\ref{4eq10}), we obtain the existence of a polynomial~$S_{0}\in \R_{>0}[z]$ of degree~$\k$, such that for all~${z\in C}$,~${S_{0}(|z|)>S(|z|)}$, and such that for all~$z\in C\cap \R_{\geq 1}K$, for all $q$ close to $1$

\begin{equation}\label{4eq17}
\begin{array}{l}
e_{q^{\k}}\Big(S_{0}(|z|)\Big)+(q-1)\Big(\left|R_{1}(qz)\right|+\left|R_{2}(qz)\right|\Big)+\\\\
(q-1)(1+(q-1)c_{1})mS(q|z|)e_{q^{\k}}\Big(S_{0}(|z|)\Big)
\leq e_{q^{\k}}\Big(S_{0}(q|z|)\Big).
\end{array}\end{equation}

\vspace{1 mm}\pagebreak[3]
\begin{center}
\textbf{Step 4 : Conclusion.}\end{center}
\vspace{1 mm}
We are going now to prove (\ref{4eq6}) with the polynomial~$S_{0}$ we have defined in Step 3. We will proceed by an induction on~$n$. The step~$n=0$ is true because of the assumption~$(a)$. \par
Induction hypothesis: let us fix~$n\in \N$, and assume that if~$q\in]1,q_{0}[$,~$w\in K$, with~$q^{n+1}w\in C$, 
$$
\left|\frac{\widetilde{M}_{w}\left(q^{n}w\right)-M_{w}\left(q^{n}w,q\right)}{q-1}\right|<e_{q^{\k}}\Big(S_{0}\left(\left|q^{n}w\right|\right)\Big).
$$
 From (\ref{4eq16}), we obtain that 
$$
\begin{array}{lll}
\left|\frac{\widetilde{M}_{w}\left(q^{n+1}w\right)-M_{w}\left(q^{n+1}w,q\right)}{q-1}\right| &
\leq& e_{q^{\k}}\Big(S_{0}\left(\left|q^{n}w\right|\right)\Big)
+(q-1)\Big(\left|R_{1}\left(q^{n+1}w\right)\right|+\left|R_{2}\left(q^{n+1}w\right)\right|\Big)\\\\
&+&e_{q^{\k}}\Big(S_{0}\left(\left|q^{n}w\right|\right)\Big)\times (q-1)(1+(q-1)c_{1})mS\left(\left|q^{n+1}w\right|\right)
.
\end{array}
$$
Using additionally (\ref{4eq17}), we find that
$$\left|\frac{\widetilde{M}_{w}\left(q^{n+1}w\right)-M_{w}\left(q^{n+1}w,q\right)}{q-1}\right|<e_{q^{\k}}\Big(S_{0}\left(\left|q^{n+1}w\right|\right)\Big).~$$
This concludes the proof of (\ref{4eq6}).
\end{proof}

\begin{proof}[Proof of Proposition~\ref{4propo3}]
Let~$K$ be the compact considered in hypothesis~$(ii)$, with~${w_{0}\in K\subset C}$, so that we have 
$$\lim \limits_{q \to 1} \Big(M(w,q)\Big)=\Big(\widetilde{M}(w)\Big) ,$$
uniformly on~$ K$. 
Let~$N(w,q)$ be the matrix, such that~$ N(w,q)$ has entries that are equal to the entrywise division of~$\widetilde{M}(w)$  by~$M(w,q)$. Due to the uniform convergence on~$K$ (assumption~~$(ii)$), the entries of~$N(w,q)$ converge uniformly on~$K$ to~$1$, as~$q$ goes to~$1$. We are going to apply Lemma~\ref{4lem10}, with 
\begin{equation}\label{4eq21}
\Big(M_{w}(z,q)\Big)_{w\in K}:=\Big(M(z,q)\times_{h}N(w,q)\Big)_{w\in K}, \hbox{ and} \left(\widetilde{M}_{w}(z)\right)_{w\in K}:=\left(\widetilde{M}(z)\right)_{w\in K},
\end{equation}
where~$\times_{h}$ denotes the Hadamard product, that is~$(a_{i})\times_{h}(b_{i}):=(a_{i}b_{i})$.
If~$a,b,c\in \C$, we have:
$$\left|a-b\right|<\left|c\right|^{-1}\left| a\times c-b\right|+\left|c^{-1}-1\right|\times\left|b\right|.$$
We are going to apply this inequality entrywise, to the entries of~$M\left(q^{n}w,q\right)$,~$\widetilde{M}\left(q^{n}w\right)$ and~$N(w,q)$. Since the entries of~$N(w,q)$ tend to~$1$, we find that there exists~$\d(q)>0$, (resp.~$\e(q)>0$) that converges to~$1$ (resp. converges to~$0$) as~$q$ goes to~$1$, such that for all~$w\in K$ and~$n\in \N$, with~$q^{n}w\in C$:
$$
\left|M\left(q^{n}w,q\right)-\widetilde{M}\left(q^{n}w\right)\right|<\d(q)\left|M\left(q^{n}w,q\right)\times_{t}N(w,q)-\widetilde{M}\left(q^{n}w\right)\right|
+\e(q)\left|\widetilde{M}(q^{n}w)\right|.$$
Using the assumption~$(iii)$, there exists~$R\in \C[z]$, such that for all~$z\in C\cap \R_{\geq 1}K$, ${\left|\widetilde{M}(z)\right|<\left|R(z)\right|}$.
Lemma~\ref{4lem10} applied to (\ref{4eq21}), gives the existence of a polynomial~$S_{0}$, that does not depend upon~$w$, such that for all~$q$ close to~$1$, for all~$w\in K$, for all~$n\in \N$, with~$q^{n}w\in C$, we obtain:
$$
\left|M\left(q^{n}w,q\right)-\widetilde{M}\left(q^{n}w\right)\right|<(q-1)\d(q)e_{q^{\k}}\Big(S_{0}\left(\left|q^{n}w\right|\right)\Big)
+\e(q)\left|R\left(q^{n}w\right)\right|.$$
In other words, for~$q$ close to~$1$ and for all~$z\in C\cap \R_{\geq 1}K$, we have
$$
\left|M(z,q)-\widetilde{M}(z)\right|<(q-1)\d(q)e_{q^{\k}}\Big(S_{0}\left(\left|z\right|\right)\Big)
+\e(q)\left|R(z)\right|.$$
The uniform convergence follows immediately.
\end{proof}

\pagebreak[3]
\section{Proof of Theorem~\ref{4theo1}.}\label{4sec6}
The goal of this section is to prove Theorem~\ref{4theo1}. In~$\S \ref{4sec61}$, we treat the confluence of the ``discrete''~$q$-Laplace transformation. In ~$\S \ref{4sec62}$ we prove Theorem~\ref{4theo1} in a particular case. In ~$\S \ref{4sec63}$, we prove Theorem~\ref{4theo1} in the general case. 

\pagebreak[3]
\subsection{Confluence of the ``discrete''~$q$-Laplace transformation.}\label{4sec61}

\pagebreak[3]
\begin{lem}\label{4lem7}
 Let~$a\in \C$ and~$k\in \Q$. Then, for any~$q>1$ and~$z\in \C^{*}$, the following inequality is true
$
{\left|e_{q}\left(az^{k}\right)\right|\leq\exp\left|az^{k}\right|}$. Moreover, we have $$\lim\limits_{q \to 1}e_{q}\left(az^{k}\right)=\exp\left(az^{k}\right),
$$
uniformly on the compacts of~$\C^{*}$.
\end{lem}

\begin{proof}
The coefficients of  the series of function defining~$e_{q}\left(az^{k}\right)$ depend upon the parameter~$q$. By construction, we have for all $n\in \N$ and all $q>1$, $n\leq [n]_{q}$, and therefore~$n!\leq[n]_{q}^{!}$. Then, for all~$q>1$ and~$z\in \C^{*}$, we have the following inequalities:
$$\left|e_{q}\left(az^{k}\right)\right|\leq \displaystyle\sum_{n=0}^{ \infty} \left|\displaystyle  \dfrac{a^{n}z^{nk}}{[n]_{q}^{!}} \right|\leq \displaystyle\sum_{n=0}^{ \infty} \left|\displaystyle  \dfrac{a^{n}z^{nk}}{n!} \right|=\exp\left|az^{k}\right|.$$
The convergence is then a direct consequence of the dominated convergence theorem, since the series defining $e_{q}\left(az^{k}\right)$ is termwise dominated by the series defining $\exp\left|az^{k}\right|$.
\end{proof}

Let~$d\in \R$, let~$k\in \Q_{>0}$ and let~$f$ be a function that belongs to~$\overline{\mathbb{H}}_{k}^{d}$, see Definition~\ref{4defi1},
~${g:=\rho_{1/k}( f)}$,~${\widetilde{f}\in \widetilde{\mathbb{H}}_{k}^{d}}$, see Definition \ref{4defi3}, and~${\widetilde{g}:=\rho_{1/k}\left(\widetilde{f}\right)}$. For the reader's convenience, we recall the expressions of the Laplace transformations of order~$1$ and~$k$ that come from Definitions \ref{4defi3} and \ref{4defi5}: 
$$
\begin{array}{ll}
\mathcal{L}_{q,1}^{[d]}(g)(z,q)=(1-q)e^{id}\displaystyle \sum_{l\in \Z}\frac{q^{l}g\left(q^{l}e^{id},q\right)}{ze_{q}\left(\frac{q^{l+1}e^{id}}{z}\right)},&\mathcal{L}_{1}^{d}\left( \widetilde{g} \right)(z)=\displaystyle\int_{0}^{\infty e^{id}}\frac{\widetilde{g}(\z)}{z\exp\left(\frac{\z}{ z}\right)}d\z,\\
\mathcal{L}_{q,k}^{[d]}( f)=\rho_{k}\circ\mathcal{L}_{q,1}^{[d]}\circ\rho_{1/k}( f),&\mathcal{L}_{k}^{d}\left( \widetilde{f}\right)=\rho_{k}\circ\mathcal{L}_{1}^{d}\circ\rho_{1/k}\left( \widetilde{f}\right).
\end{array}
$$
Since~$f\in \overline{\mathbb{H}}_{k}^{d}$, there exist~$\e>0$, constants~$J,L>0$, such that for all~$q$ close to~$1$,~${\z\mapsto f(\z,q)}$ is analytic on~$\overline{S}(d-\e,d+\e)$, and for all~$\z\in \R_{>0}$:
\begin{equation}\label{4eq24}
\left|f\left(e^{id}\z,q\right)\right|<Je_{q}\left(L\z^{k}\right).
\end{equation}
\pagebreak[3]
\begin{lem}\label{4lem5}
In the notation introduced above, let us assume that we have 
~$\lim\limits_{q \to 1}f:=\widetilde{f}$,
uniformly on the compacts of~$\overline{S}(d-\e,d+\e)$. 
Then, we have 
$$\lim\limits_{q \to 1}\mathcal{L}_{q,k}^{[d]}\Big( f\Big)(z,q)=\mathcal{L}_{k}^{d}\left( \widetilde{f}\right)(z),$$
uniformly on the compacts of~${\Big\{z\in \overline{S}\left(d-\frac{\pi}{2k\pi},d+\frac{\pi}{2k\pi}\right)\Big||z|<1/L\Big\}}$.
\end{lem}

\begin{proof}
The expressions of the Laplace transformations of order~$k$ allow us to reduce to the case~$k=1$. The variable change~$\z\mapsto \z e^{-id}$ allows us to reduce to the case~$d=0$.  Let us fix a an arbitrary compact subset~$K$ of ~${\Big\{z\in \overline{S}\left(-\frac{\pi}{2\pi},+\frac{\pi}{2\pi}\right)\Big||z|<1/L\Big\}}$, and let us prove the uniform convergence on~$K$.\par 
The~$q$-Laplace transformation can be seen as a Riemann sum with associated partition~$\left(q^{l}\right)_{l\in \Z}$. Moreover, on every compact of~$]0,\infty[$, the mesh of the partition tends to~$0$ as~$q$ goes to~$1$.  Using the dominated convergence theorem, it is sufficient to prove the existence of~$\left(h_{l}\right)\in \left(\R_{>0}\right)^{\Z}$ that satisfies~$\displaystyle \sum_{l\in \Z} h_{l}<\infty$, such that for all~$q$ close to~$1$,~$l\in \Z$ and~$z\in K$,
$$\left|\frac{(q-1) q^{l}f\left(q^{l},q\right)}{ze_{q}\left(q^{l+1} /z\right)}\right|<h_{l}.$$
By definition of the~$q$-Laplace transformation  and (\ref{4eq24}), we have for all~$z\in K$,
$$\left|\mathcal{L}_{q,k}^{[d]}(f)(z,q)\right|\leq (q-1)\displaystyle \sum_{l\in \Z} \left|\frac{q^{l}J}{z}\frac{e_{q}\left(Lq^{l}\right)}{e_{q}\left(q^{l+1}/z\right)}\right|
.$$
 For all~$l\in \Z$,~$z\in K$,~$q>1$, we have: 
\begin{equation}\label{4eq19}
\left|\frac{q^{l+1}J}{z}\frac{e_{q}\left(Lq^{l+1}\right)}{e_{q}\left(q^{l+2}/z\right)}\right|=\left|q
\frac{1+(q-1)Lq^{l}}{1+(q-1)q^{l+1}/z}\right|\left|\frac{q^{l}J}{z}\frac{e_{q}\left(Lq^{l}\right)}{e_{q}\left(q^{l+1}/z\right)}\right|.
\end{equation}
Let~$R\in \R_{>0}$,~$M_{1}<1$,~$q_{0}>1$,  such that for all~$x\geq R$, for all~$z\in K$, and for all~$q\in ]1,q_{0}[$,
\begin{equation}\label{4eq13}
\left|q
\frac{1+(q-1)Lx}{1+(q-1)qx/z}\right|<M_{1}.
\end{equation}
Let~$q\mapsto l_{0}(q)\in \Z$ be the smallest integer that satisfies 
$$q^{l_{0}(q)}\geq R.$$
We will break the series into two parts, and start by treating the convergence of~${(q-1)\displaystyle \sum_{l=l_{0}(q)}^{\infty} \frac{q^{l}}{z}\frac{f\left(q^{l},q\right)}{e_{q}\left(q^{l+1}/z\right)}}$ to~$\displaystyle \int_{R}^{\infty }z^{-1}\widetilde{f}(\z)e^{-\frac{\z}{z}}d\z$.
Because of (\ref{4eq19}) and (\ref{4eq13}), for all~$q\in ]1,q_{0}[$,~$l\geq l_{0}(q)$ and~$z\in K$, we have 
$$\left|\frac{q^{l+1}J}{z}\frac{e_{q}\left(Lq^{l+1}\right)}{e_{q}\left(q^{l+2}/z\right)}\right|<M_{1}\left|\frac{q^{l}J}{z}\frac{e_{q}\left(Lq^{l}\right)}{e_{q}\left(q^{l+1}/z\right)}\right|.$$
By iteration, we find that 
~$$
(q-1)\displaystyle \sum_{l=l_{0}(q)}^{\infty} \left|\frac{q^{l}J}{z}\frac{e_{q}\left(Lq^{l}\right)}{e_{q}\left(q^{l+1}/z\right)}\right|\leq(q-1)\displaystyle \sum_{l=0}^{\infty}\left|\dfrac{q^{l_{0}(q)}J}{z}\dfrac{e_{q}\left(Lq^{l_{0}(q)}\right)}{e_{q}\left(q^{l_{0}(q)+1}/z\right)}\right|\Big(M_{1}\Big)^{l}$$
Using Lemma~\ref{4lem7}, we obtain that~$|e_{q}(z)|$ can be bounded for~$(z,q)\in K_{0}\times ]1,q_{0}[$, where~$K_{0}$ is an arbitrary compact of~$\C$. Moreover, the fact that~$e_{q}(z)$ vanishes only on~$\frac{q^{\N^{*}}}{1-q}$, implies that~$\frac{1}{|e_{q}(z)|}$ can also be bounded for~$(z,q)\in K_{1}\times ]1,q_{0}[$, where~$K_{1}$ is an arbitrary compact of~$\C\setminus\R_{<0}$.
In particular, we find that for all~$R_{0}\in \R_{>0}$
\begin{equation}\label{4eq12}
\displaystyle\sup_{\substack{x\in [1,R_{0}],\\q\in ]1,q_{0}[,\\z\in K}}\left|\frac{e_{q}(Lx)}{e_{q}(qx/z)}\right|<\infty .
\end{equation}
Then, we obtain that for all~$q\in ]1,q_{0}[$ and for all~$l\geq l_{0}(q)$,
~$$
\begin{array}{lll}
(q-1)\displaystyle \sum_{l=0}^{\infty}\left|\dfrac{q^{l_{0}(q)}J}{z}\dfrac{e_{q}\left(Lq^{l_{0}(q)}\right)}{e_{q}\left(q^{l_{0}(q)+1}/z\right)}\right|\Big(M_{1}\Big)^{l} &\leq&\displaystyle\sup_{\substack{x\in [1,q_{0}],\\q\in ]1,q_{0}[,\\z\in K}}\left|(q-1)\dfrac{RxJ}{z}\dfrac{e_{q}(LRx)}{e_{q}(qRx/z)}\right|\sum_{l=0}^{\infty}\Big(M_{1}\Big)^{l}\\\\
&=&\dfrac{M_{2}}{1-M_{1}},\end{array}
$$ 
where~$M_{2}:=\displaystyle\sup_{\substack{x\in [1,q_{0}],\\q\in ]1,q_{0}[,\\z\in K}}\left|(q-1)\dfrac{RxJ}{z}\dfrac{e_{q}(LRx)}{e_{q}(qRx/z)}\right|$ is a real positive constant. Hence, we have 
~$$(q-1)\displaystyle \sum_{l=l_{0}(q)}^{\infty} \left|\frac{ q^{l}f\left(q^{l},q\right)}{ze_{q}\left(q^{l+1} /z\right)}\right|\leq \frac{M_{2}}{1-M_{1}}<\infty,$$ and the dominated convergence theorem gives
\begin{equation}\label{4eq22}
\lim\limits_{q \to 1} (q-1)\displaystyle \sum_{l=l_{0}(q)}^{\infty} \frac{q^{l}}{z}\frac{f\left(q^{l},q\right)}{e_{q}\left(q^{l+1}/z\right)}=\displaystyle \int_{R}^{\infty }z^{-1}\widetilde{f}(\z)e^{-\frac{\z}{z}}d\z,
\end{equation}
uniformly on~$K$.\par 
Let us now treat~$(q-1)\displaystyle \sum_{l=-\infty}^{l_{0}(q)-1} \frac{q^{l}}{z}\frac{f\left(q^{l},q\right)}{e_{q}\left(q^{l+1}/z\right)}$. 
Because of (\ref{4eq12}), we may define 
~$$M_{3}:=\sup_{\substack{x\in [0,R],\\q\in ]1,q_{0}[,\\z\in K}}\left|\frac{J}{z}\frac{e_{q}(Lx)}{e_{q}(qx/z)}\right|<\infty.$$ 
Therefore, for all~$q$ close to~$1$ and for all~$z\in K$, we have
~$$\left|(q-1)\displaystyle \sum_{l=-\infty}^{l_{0}(q)-1} \frac{q^{l}}{z}\frac{f\left(q^{l},q\right)}{e_{q}\left(q^{l+1}/z\right)}\right|\leq (q-1)\displaystyle \displaystyle \sum_{l=-\infty}^{l_{0}(q)-1}  q^{l}M_{3}\leq \frac{(q-1)RM_{3}}{1-1/q}\leq qRM_{3}.$$
Consequently, due to the dominated convergence theorem, we have
~$$\lim\limits_{q \to 1} (q-1)\displaystyle \sum_{l=-\infty}^{l_{0}(q)-1} \frac{q^{l}}{z}\frac{f\left(q^{l},q\right)}{e_{q}\left(q^{l+1}/z\right)}=\displaystyle \int_{0}^{R}z^{-1}\widetilde{f}(\z)e^{-\frac{\z}{z}}d\z,$$
uniformly on~$K$.
This limit combined with (\ref{4eq22}) yields the result.
\end{proof}

\pagebreak[3]
\subsection{Proof of Theorem~\ref{4theo1} in a particular case.}\label{4sec62}
In this subsection, we are going to prove Theorem~\ref{4theo1} in a particular case. Let us start by recalling some notations. See~$\S \ref{4sec1}$ to~$\S \ref{4sec4}$ for rest of the notations. We consider (\ref{4eq4}), that admits~$\widetilde{h}\in \C[[z]]$ as solution and~$\widetilde{b}_{0},\dots,\widetilde{b}_{m}$ as coefficients. In other words, we have 
~$$\widetilde{b}_{m}(z)\d^{m}\left(\widetilde{h}(z)\right)+\dots+\widetilde{b}_{0}(z)\widetilde{h}(z)=0.$$ Let~$d_{0}:=\max\left(2,\deg \left(\widetilde{b}_{0}\right),\dots, \deg \left(\widetilde{b}_{m}\right)\right)$.
Let~$k_{1}<\dots<k_{r-1}$ be the slopes of (\ref{4eq4}) different from $0$, let~$k_{r}$ be an integer strictly bigger than~$k_{r-1}$ and~$d_{0}$, and set~$k_{r+1}:=+\infty$. 
Let~$(\k_{1},\dots,\k_{r})$ defined as:
~$$\k_{i}^{-1}:=k_{i}^{-1}-k_{i+1}^{-1}.$$  
We define the~$(\widetilde{\k_{1}},\dots,\widetilde{\k_{s}})$ as follows:
We take~$(\k_{1},\dots,\k_{r})$ and for~$i=1,...,r$, replace successively~$\k_{i}$ by~$\a_{i}$ terms~$\a_{i}\k_{i}$, where~$\a_{i}$ is the smallest integer such that~$\a_{i}\k_{i}$ is greater or equal than~$d_{0}\geq 2$. See Example~\ref{4ex1}. Let~$\b\in \N^{*}$ be minimal, such that for all~${i\in \{1,\dots,s\}}$,~$\b/\widetilde{\k_{i}}\in \N^{*}$.\par 
In this subsection \ref{4sec62}, we are going to assume that~${z\mapsto\hat{h}(z,q),\widetilde{h}(z)\in \C\left[\left[z^{\b}\right]\right]}$. Ramification in $\S \ref{4sec63}$ will allow us to reduce to the case tackled in the present subsection. Note that in this case, we have~$\hat{h}=\hat{h}^{(0)}$. For the reader's convenience, we will decompose the proof of Theorem~\ref{4theo1} in four steps.
\vspace{2 mm}\pagebreak[3]
\begin{center}
\textbf{Step 1: Construction of~$\Sigma_{\widetilde{h}}$.}\end{center}
\vspace{2 mm}
Let us consider a general formal power series~$\hat{f}\in \C\left[\left[z^{\b}\right]\right]$ (resp.~$\widetilde{f}\in \C\left[\left[z^{\b}\right]\right]$) that satisfy  a linear~$\dq$-equation (resp.~$\d$-equation) of order~$m_{0}$ with coefficients in~$\C\left[z^{\b}\right]$.
 Then, for all~$i\in \{1,\dots,s\}$, ~$\r_{1/\widetilde{\k_{i}}}\left(\hat{f}\right)$ (resp.~$\r_{1/\widetilde{\k_{i}}}\left(\widetilde{f}\right)$) satisfies a linear~$\dq$-equation (resp.~$\d$-equation) with coefficients in~$\C[z]$. Therefore, Propositions~\ref{4propo4} and \ref{4propo2}, combined with the definition of the Borel transformations (see Definitions~\ref{4defi3} and \ref{4defi4}) imply that for all~${i\in\{1,\dots,s\}}$,~$\hat{\mathcal{B}}_{q,\widetilde{\k_{i}}}\left(\hat{f}\right)$ (resp.~$\hat{\mathcal{B}}_{\widetilde{\k_{i}}}\left(\widetilde{f}\right)$) satisfies a linear~$\dq$-equation (resp.~$\d$-equation) of order independent of~$q$ (resp. of the same order than the~$\dq$-equation satisfied by~$\hat{\mathcal{B}}_{q,\widetilde{\k_{i}}}\left(\hat{f}\right)$) with coefficients in~$\C\left[z^{\b}\right]$. \par 
In particular, for~$j\in \{1,\dots,s\}$,
~$\hat{\mathcal{B}}_{\widetilde{\k_{j}}}\circ\dots\circ\hat{\mathcal{B}}_{\widetilde{\k_{s}}}\left(\widetilde{h}\right)$
satisfies a linear~$\d$-equation that we will see as a system. We define~$\Sigma_{\widetilde{h}}$ as the union of~$\widetilde{\Sigma}_{\widetilde{h}}$, the set of its singular direction that has been defined in Proposition~\ref{4propo5}, and the argument of the poles of the differential system satisfied by the successive Borel transformations. The set~$\Sigma_{\widetilde{h}}\subset \R$ is finite modulo~$2\pi \Z$. 
\vspace{2 mm}\pagebreak[3]
\begin{center}
\textbf{Step 2: Local convergence of the~$q$-Borel transformations.}\end{center}
\vspace{2 mm}
From what is preceding,
~$\hat{\mathcal{B}}_{q,\widetilde{\k_{1}}}\circ\dots\circ\hat{\mathcal{B}}_{q,\widetilde{\k_{s}}}\left(\hat{h}\right)$ (resp.~$\hat{\mathcal{B}}_{\widetilde{\k_{1}}}\circ\dots\circ\hat{\mathcal{B}}_{\widetilde{\k_{s}}}\left(\widetilde{h}\right)$)
satisfies a linear~$\dq$-equation of order ~$m_{1}\in \N$, that we will see as a system~$\dq Y(\z,q)=E(\z,q)Y(\z,q)$ with~${\z\mapsto \mathrm{Id}+(q-1)E(\z,q)\in \mathrm{GL}_{m_{1}}\Big(\C\left(\z^{\b}\right)\Big)}$ (resp. a linear~$\d$-equation of order~$m_{1}$ we will see as a system~$\d \widetilde{Y}(\z)=\widetilde{E}(\z)\widetilde{Y}(\z)$  with~$\widetilde{E}(\z)\in \mathrm{M}_{m_{1}}\Big(\C\left(\z^{\b}\right)\Big)$).\par 
Because of Proposition~\ref{4propo5},~$\hat{\mathcal{B}}_{\widetilde{\k_{1}}}\circ\dots\circ\hat{\mathcal{B}}_{\widetilde{\k_{s}}}\left(\widetilde{h}\right)$ is convergent. Let us prove that~${\z\mapsto \hat{\mathcal{B}}_{q,\widetilde{\k_{1}}}\circ\dots\circ\hat{\mathcal{B}}_{q,\widetilde{\k_{s}}}\left(\hat{h}\right) \in\C\left\{\z^{\b}\right\}}$. Due to \textbf{(A2)}, the slopes of the~$\sq$-equation satisfied by~$\hat{h}$ are independent of~$q$, and the smallest positive slope is~$k_{1}$. As we can see in \cite{R92}, Theorem 4.8, (see also \cite{Be}), there exist~$C_{1}(q),C_{2}(q)>0$, such that for all~$l\in \N$, for all~$q>1$
~$$\Big| \hat{h}_{l}(q) \Big| <C_{1}(q)C_{2}(q)^{l}\Big([l]_{q}^{!}\Big)^{1/k_{1}},$$
where~$\hat{h}(z,q)=\sum \hat{h}_{l}(q)z^{l}$.
By construction of the~$\widetilde{\k_{i}}$, we have~$\displaystyle\sum_{i=1}^{s} \widetilde{\k_{i}}^{-1}=\sum_{i=1}^{r} \k_{i}^{-1}=k_{1}^{-1}$. Since for all~$l,k\in \N^{*}$, and for all~$q>1$,~$\Big([kl]_{q}^{!}\Big)^{1/k}\leq [l]_{q}^{!}$, we find that for all~$l\in \N$, for all~$q>1$, 
~$$\Big| \hat{h}_{l}(q) \Big| <C_{1}(q)C_{2}(q)^{l}\displaystyle \prod_{i=1}^{s}[l/\widetilde{\k_{i}}]_{q}^{!}.$$
Hence, we obtain that~${\z\mapsto \hat{\mathcal{B}}_{q,\widetilde{\k_{1}}}\circ\dots\circ\hat{\mathcal{B}}_{q,\widetilde{\k_{s}}}\left(\hat{h}\right) \in\C\left\{\z^{\b}\right\}}$.  Applying Lemma~\ref{4lem11}, we find
\begin{equation}\label{4eq5}
\lim\limits_{q \to 1}\hat{\mathcal{B}}_{q,\widetilde{\k_{1}}}\circ\dots\circ\hat{\mathcal{B}}_{q,\widetilde{\k_{s}}}\left(\hat{h}\right)
=\hat{\mathcal{B}}_{\widetilde{\k_{1}}}\circ\dots\circ\hat{\mathcal{B}}_{\widetilde{\k_{s}}}\left(\widetilde{h}\right),
\end{equation}
uniformly on a closed disk centered at~$0$.
\vspace{1 mm}\pagebreak[3]
\begin{center}
\textbf{Step 3: Local convergence of the~$q$-Borel-Laplace summation.}\end{center}
\vspace{1 mm}
Let~$d\in \R\setminus \Sigma_{\widetilde{h}}$. The variable change~$\z\mapsto \z e^{-id}$ allows us to reduce to the case~$d=0$. By construction of~$\Sigma_{\widetilde{h}}$,~$\widetilde{E}(\z)$ has no poles for~$\z\in\overline{S}(-\e,+\e)$.  Because of the assumption \textbf{(A3)}, Propositions~\ref{4propo2} and \ref{4propo4}, we deduce that~$E(\z,q)$ has no poles for~$\z\in\overline{S}(-\e,\e)$ and for~$q$ close to~$1$.
Because of Lemma~\ref{4lem4}, the series~$z\mapsto\hat{\mathcal{B}}_{q,\widetilde{\k_{1}}}\circ\dots\circ\hat{\mathcal{B}}_{q,\widetilde{\k_{s}}}\left(\hat{h}\right)(z,q)$ admits, for~$q$ close to~$1$, an analytic continuation~$f_{1}(\z,q)$ defined on~$\overline{S}(-\e,\e)$. We want now to prove that~$f_{1}(\z,q)$  converges to~$\widetilde{f}_{1}(\z)$ on a convenient domain.  \par 
 Due to Proposition~\ref{4propo5}, there exists~$B_{1}>0$ such that the functions $\dfrac{\widetilde{f}_{1}(\z)}{\exp\left(B_{1}\z^{\widetilde{\k_{1}}}\right)},\dots,
\dfrac{\d^{m_{1}-1}\widetilde{f}_{1}(\z)}{\exp\left(B_{1}\z^{\widetilde{\k_{1}}}\right)}$ tend to~$0$ as~$\z\in \overline{S}(-\e,\e)$ tends to infinity. Using
$$\dq \left(e_{q^{\widetilde{\k_{1}}}}\left(B_{1}\z^{\widetilde{\k_{1}}}\right)^{-1}\right)=
\frac{-\left[\, \widetilde{\k_{1}}\,\right]_{q}B_{1}\z^{\widetilde{\k_{1}}}}{1+(q-1)\left[\, \widetilde{\k_{1}}\,\right]_{q}B_{1}\z^{\widetilde{\k_{1}}}}e_{q}\left(B_{1}\z^{\widetilde{\k_{1}}}\right)^{-1},$$
we obtain that 
~$\dfrac{f_{1}(\z,q)}{e_{q^{\widetilde{\k_{1}}}}\left(B_{1}\z^{\widetilde{\k_{1}}}\right)}$ (resp.~$
\widetilde{f}_{1}(\z)\exp\left(-B_{1}\z^{\widetilde{\k_{1}}}\right)$) satisfies a linear~$\dq$-equation of order~$m_{1}+1$ (resp. a linear~$\d$-equation of order~$m_{1}+1$) with coefficients in~$\C(z)$. Because of (\ref{4eq5}), there exists~$\z_{0}>0$, such that~$f_{1}(\z_{0},q)$ converges to~$\widetilde{f}_{1}(\z_{0})$ as~$q$ goes to~$1$. Let
$$Y(\z,q):=\begin{pmatrix}
\dfrac{f_{1}(\z,q)}{e_{q^{\widetilde{\k_{1}}}}\left(B_{1}\z^{\widetilde{\k_{1}}}\right)} F_{0}(q)   \\ 
\vdots \\
\dfrac{\dq^{m_{1}} f_{1}(\z,q)}{e_{q^{\widetilde{\k_{1}}}}\left(B_{1}\z^{\widetilde{\k_{1}}}\right)} F_{m_{1}}(q)
\end{pmatrix},\quad
\widetilde{Y}_{B_{1}}(\z):=\begin{pmatrix}
\dfrac{\widetilde{f}_{1}(\z)}{\exp\left(B_{1}\z^{\widetilde{\k_{1}}}\right)}   \\ 
\vdots \\
 \dfrac{\d^{ m_{1}}\widetilde{f}_{1}(\z)}{\exp\left(B_{1}\z^{\widetilde{\k_{1}}}\right)}
\end{pmatrix},$$
 where the~$F_{i}(q)\in \C$ are defined by:
\begin{equation}\label{4eq14}
\left. \dfrac{\dq^{i}f_{1}(\z,q)}{e_{q^{\widetilde{\k_{1}}}}\left(B_{1}\z^{\widetilde{\k_{1}}}\right)}F_{i}(q)\right| _{\z=\z_{0}}=\left. \dfrac{\d^{i}\widetilde{f}_{1}(\z)}{\exp\left(B_{1}\z^{\widetilde{\k_{1}}}\right)}\right|_{\z=\z_{0}}.
\end{equation}
From what is preceding, there exist~$\z\mapsto \mathrm{Id}+(q-1)D(\z,q)\in \mathrm{GL}_{m_{1}+1}\Big(\C(\z)\Big)$ and ${\widetilde{D}(\z)\in \mathrm{M}_{m_{1}+1}\Big(\C(\z)\Big)}$, such that 
$$\left\{\begin{array}{lll}
\dq Y(\z,q)&=&D(\z,q)Y(\z,q)\\ \\
\d\widetilde{Y}(\z)&=&\widetilde{D}(\z)\widetilde{Y}(\z).
\end{array}\right.$$

\pagebreak[3]
\begin{lem}
Let us consider~$C$, a convex subset of~$\overline{S}(-\e,\e)$, that contains~${\Big\{\z\in \overline{S}(-\e_{1},\e_{1})\Big| |\z|>\z_{0}/2 \Big\}}$, for some~$\e_{1}\in ]0,\e[$, such that $0$ does not belong to its closure. Then, the systems 
$$\left\{\begin{array}{lll}
\dq Y(\z,q)&=&D(\z,q)Y(\z,q)\\ \\
\d \widetilde{Y}(\z)&=&\widetilde{D}(\z)\widetilde{Y}(\z),
\end{array}\right.$$
satisfy the assumptions of Proposition~\ref{4propo3}, with ${K:=\left\{\z\in  \widetilde{\C}\Big| |\z-\z_{0}|\leq\e_{0} \right\}}$ and~$\e_{0}>0$ is a real positive constant sufficiently small.
\end{lem}
\begin{proof}[Proof of the lemma]
We are going to check separately the three assumptions of Proposition~\ref{4propo3}. 
\begin{trivlist}
\item~$(i)$ 
Because of the assumption \textbf{(A3)}, Propositions~\ref{4propo2} and \ref{4propo4}, we obtain the existence of~$c_{1}>0$, such that for all~$\z\in C$,
$$
\left|E(\z,q)-\widetilde{E}(\z)\right|<(q-1)c_{1}\left(\left|\widetilde{E}(\z)\right|+\mathbf{1}_{m_{1}}\right).
$$
With the~$q$-difference equation satisfied by~$e_{q^{\widetilde{\k_{1}}}}\left(B_{1}\z^{\widetilde{\k_{1}}}\right)$, this implies that  we have the existence of~$c_{2}>0$, such that for~$q$ close to~$1$, for~$\z\in C$,
$$\left|D(\z,q)-\widetilde{D}(\z)\right|<(q-1)c_{2}\left(\left|\widetilde{D}(\z)\right|+\mathbf{1}_{m_{1}+1}\right).$$ 
\item~$(ii)$  Let~$i\in \{0,\dots,m_{1}\}$.  Due to (\ref{4eq5}) and Lemma~\ref{4lem7},~$F_{i}(q)$ converges to~$1$ as~$q$ goes to~$1$. Then, we have for all~$i\in \{0,\dots,m_{1}\}$
$$
\lim\limits_{q \to 1}
\dfrac{\dq^{i}f_{1}(\z,q)}{e_{q^{\widetilde{\k_{1}}}}\left(B_{1}\z^{\widetilde{\k_{1}}}\right)} F_{i}(q)=
\dfrac{\d^{i}\widetilde{f}_{1}(\z)}{\exp\left(B_{1}\z^{\widetilde{\k_{1}}}\right)},$$
uniformly on a compact set with non empty interior containing~$\z_{0}$. Let us choose $\e_{0}>0$ small enough, such that we have the uniform convergence on~${K:=\left\{\z\in  \widetilde{\C}\Big| |\z-\z_{0}|\leq\e_{0} \right\}}$ and such that $K$ is included in $C$. Because of (\ref{4eq14}), ~$
\dfrac{\dq^{i}f_{1}(\z,q)}{e_{q^{\widetilde{\k_{1}}}}\left(B_{1}\z^{\widetilde{\k_{1}}}\right)} F_{i}(q)$ and~$
\dfrac{\d^{i}\widetilde{f}_{1}(\z)}{\exp\left(B_{1}\z^{\widetilde{\k_{1}}}\right)}$ are equal at~$\z_{0}$.
\\
\item~$(iii)$ From the choice of~$B_{1}$, we have the existence of~$R\in \C[\z]$, such that for~$\z\in C$, for all~${i\in \{0,\dots, m_{1}-1\}}$:
$$\left|\d^{i}
\left(\widetilde{f}_{1}(\z)\right)\exp\left(-B_{1}\z^{\widetilde{\k_{1}}}\right)\right|
<|R(\z)|.$$
\end{trivlist}
\end{proof}
We need now the following elementary lemma. 

\pagebreak[3]
\begin{lem}\label{4lem12}
For all~$z\in \C$, for all~$q>1$, we have~$e_{q^{2}}\left(|z|\right)^{2}\leq e_{q}\left(|(1+q)z|\right)$. 
\end{lem} 

\begin{proof}[Proof of the lemma]
Let us remark that the two functions are equal at $z=0$. The lemma is now a direct consequence of the~$q$-difference equation
$$\sq^{2}\left(\frac{e_{q^{2}}\left(|z|\right)^{2}}{e_{q}\left(|(1+q)z|\right)}\right)=\frac{1+2(q^{2}-1)|z|+(q^{2}-1)^{2}|z|^{2}}{1+(1+q)(q^{2}-1)|z|+(q^{2}-1)^{2}q|z|^{2}}
\frac{e_{q^{2}}\left(|z|\right)^{2}}{e_{q}\left(|(1+q)z|\right)},$$
since~$\frac{1+2(q^{2}-1)|z|+(q^{2}-1)^{2}|z|^{2}}{1+(1+q)(q^{2}-1)|z|+(q^{2}-1)^{2}q|z|^{2}}\leq 1$.
\end{proof}

We finish now the proof of Theorem~\ref{4theo1}, in the particular case ${z\mapsto\hat{h}(z,q),\widetilde{h}(z)\in \C\left[\left[z^{\b}\right]\right]}$.  Let us define~$\widetilde{d}$ (resp.~$\widetilde{e}$) as the maximum of the degrees of the numerators and the denominators of the entries of ~$\widetilde{D}(\z)$ (resp.~$\widetilde{E}(\z)$), written as the quotient of two coprime polynomials. Using the differential equation satisfied by~$\exp\left(-B_{1}\z^{\widetilde{\k_{1}}}\right)$, we find that~$\widetilde{d}\leq \max(\widetilde{e},\widetilde{\k_{1}})$. Because of Remark~\ref{4rem6}, and the definition of~$d_{0}$ and~$\widetilde{\k_{1}}$ (see the beginning of the subsection),~$\widetilde{e} \leq d_{0}\leq\widetilde{\k_{1}}$. Hence~$\widetilde{d}\leq \widetilde{\k_{1}}$. Proposition~\ref{4propo3} applied to the systems
$$\left\{\begin{array}{lll}
\dq Y(\z,q)&=&D(\z,q)Y(\z,q)\\ \\
\d\widetilde{Y}(\z)&=&\widetilde{D}(\z)\widetilde{Y}(\z),
\end{array}\right.$$
implies that there exist~$R,S_{0}\in \C[z]$,~$\d(q),\e(q)$ that converge respectively to~$1$ and~$0$ as~$q\to 1$, such that
$$\left|\dfrac{f_{1}(\z,q)F_{0}(q)}{e_{q^{\widetilde{\k_{1}}}}\left(B_{1}\z^{\widetilde{\k_{1}}}\right)}-\dfrac{\widetilde{f}_{1}(\z)}{\exp\left(B_{1}\z^{\widetilde{\k_{1}}}\right)}\right|<
(q-1)\d(q)e_{q^{\widetilde{d}}}\Big(S_{0}\left(|\z|\right)\Big)
+\e(q)\left|R(\z)\right|.$$

There exists a polynomial~$S_{1}$ with degree~$\widetilde{\k_{1}}$, such that for~$|\z|$ sufficiently big and for all~$q$ close to~$1$,
$$\left|e_{q^{\widetilde{\k_{1}}}}\left(B_{1}\z^{\widetilde{\k_{1}}}\right)e_{q^{\widetilde{d}}}\Big(S_{0}\left(|\z|\right)\Big)\right|\leq e_{q^{\widetilde{\k_{1}}}}\Big(\left|S_{1}(\z)\right|\Big)^{2}.$$ 
By construction,~$\widetilde{\k_{1}}\geq 2$, (see the beginning of the subsection). Using Lemma~\ref{4lem12}, we obtain that for~$|\z|$ sufficiently big,
$$e_{q^{\widetilde{\k_{1}}}}\Big(\left|S_{1}(\z)\right|\Big)^{2}\leq e_{q^{2}}\Big(\left|S_{1}(\z)\right|\Big)^{2}\leq 
e_{q}\Big((1+q)\left|S_{1}(\z)\right|\Big).$$ 
Since~$F_{0}(q)$ converges to~$1$ and the fact that~$\widetilde{f}_{1}(\z)\exp\left(-B_{1}\z^{\widetilde{\k_{1}}}\right)$ is bounded by a polynomial, the triangular inequality yields
$$f_{1} \in\overline{\mathbb{H}}_{\widetilde{\k_{1}}}^{0}.$$
Moreover, due to Proposition~\ref{4propo3}, we have~$\lim\limits_{q \to 1} f_{1}=\widetilde{f}_{1}$,
uniformly on the compacts of~${C\cap \R_{\geq 1}K:=\left\{ xw\in C\Big|x\in [1,\infty[,w\in K\right\}}$. Hence, we find that there exists~$\e_{2}\in ]0,\e_{1}[$, such that~$\lim\limits_{q \to 1} f_{1}=\widetilde{f}_{1}$,
uniformly on the compacts of~$\overline{S}(-\e_{2},\e_{2})$. 
We may now apply Lemma~\ref{4lem5} to obtain the existence of~$L_{0}>0$, such that we have
$$\lim\limits_{q \to 1}\mathcal{L}_{q,\widetilde{\k_{1}}}^{[0]}\Big(f_{1}\Big)(\z,q) =\mathcal{L}_{\widetilde{\k_{1}}}^{0}\left(\widetilde{f}_{1}\right)(\z),$$
uniformly on the compacts of
~$\left\{\z\in \overline{S}\left(-\dfrac{\pi}{2\widetilde{\k_{1}}},+\dfrac{\pi}{2\widetilde{\k_{1}}}\right)\Big| |\z|<L_{0}\right\}$. \par 
If~$s>1$, we apply for~$j=2$ (resp.~$j=3$,~$\dots$, resp.~$j=s$) the same reasoning with the analytic continuation of
$$f_{j}(\z,q)e_{q^{\widetilde{\k_{j}}}}\left(B_{j}\z^{\widetilde{\k_{j}}}\right)^{-1}:=\mathcal{L}_{q,\widetilde{\k_{j-1}}}^{[0]}\Big(f_{j-1}\Big)e_{q^{\widetilde{\k_{j}}}}\left(B_{j}\z^{\widetilde{\k_{j}}}\right)^{-1}$$
and~$$ \widetilde{f}_{j}(\z)\exp\left(-B_{j}\z^{\widetilde{\k_{j}}}\right):=\mathcal{L}_{\widetilde{\k_{j-1}}}^{0}\left(\widetilde{f}_{j-1}\right)\exp\left(-B_{j}\z^{\widetilde{\k_{j}}}\right),$$
where~$B_{j}>0$ are chosen sufficiently large. 
 We again use Propositions~\ref{4propo4} and \ref{4propo2} to prove that they satisfy linear~$\dq$ and~$\d$-equations with coefficients in~$\C(\z)$, which are the same as the linear~$\dq$ and~$\d$-equations satisfied by~$\hat{\mathcal{B}}_{q,\widetilde{\k_{j}}}\circ\dots\circ\hat{\mathcal{B}}_{q,\widetilde{\k_{s}}}\left(\hat{h}\right)e_{q^{\widetilde{\k_{j}}}}\left(B_{j}\z^{\widetilde{\k_{j}}}\right)^{-1}$ and 
~$\hat{\mathcal{B}}_{\widetilde{\k_{j}}}\circ\dots\circ\hat{\mathcal{B}}_{\widetilde{\k_{s}}}\left(\widetilde{h}\right)\exp\left(-B_{j}\z^{\widetilde{\k_{j}}}\right)$.\par
We have proved the existence of~$L_{1}>0$, such that we have
$$\lim\limits_{q \to 1}S_{q}^{[0]}\left(\hat{h}\right)=\widetilde{S}^{0}\left(\widetilde{h}\right),$$ 
uniformly on the compacts of~$\Big\{z\in \overline{S}\left(-\frac{\pi}{2\k_{r}},+\frac{\pi}{2\k_{r}}\right)\Big| |z|<L_{1}\Big \}$.
\vspace{3 mm}\pagebreak[3]
\begin{center}
\textbf{Step 4: Global convergence of the~$q$-Borel-Laplace summation.}\end{center}
\vspace{3 mm}
To finish the proof in the particular case~$z\mapsto\hat{h}(z,q),\widetilde{h}(z)\in \C\left[\left[z^{\b}\right]\right]$, we have to prove that
$$\lim\limits_{q \to 1}S_{q}^{[0]}\left(\hat{h}\right)=\widetilde{S}^{0}\left(\widetilde{h}\right),$$
uniformly on the compacts of~$\overline{S}\left(-\frac{\pi}{2k_{r}},+\frac{\pi}{2k_{r}}\right)\setminus \bigcup \R_{\geq 1}\a_{i}$,  
where~$\a_{i}$ are the roots of~${\widetilde{b}_{m}\in \C[z]}$.
Let~$K_{0}$ be an arbitrary compact of~$\overline{S}\left(-\frac{\pi}{2k_{r}},+\frac{\pi}{2k_{r}}\right)\setminus \bigcup \R_{\geq 1}\a_{i}$, and let us prove the uniform convergence on~$K_{0}$. Without loss of generality, we may assume that~$K_{0}$ is convex and has non empty intersection with the open disk of radius~$L_{1}$ (we recall that $L_{1}$ was defined in the end of Step $3$) centered at~$0$.\par 
 From Remark~\ref{4rem3} (resp. Proposition \ref{4propo5}), we deduce that~$S_{q}^{[0]}\left(\hat{h}\right)$ (resp.~$\widetilde{S}^{0}\left(\widetilde{h}\right)$) is solution of the same linear~$\dq$-equation than~$\hat{h}$ (resp. the same linear~$\d$-equation than~$\widetilde{h}$). \par 
Let~$|z_{0}|<L_{1}$ with~$z_{0}\in K_{0}$. We are going to use Proposition~\ref{4propo3} with~$C=K_{0}$ and with the systems  
$$\left\{\begin{array}{lll}
\dq Y(z,q)&=&F(z,q)Y(z,q)\\ \\
\d  \widetilde{Y}(z)&=&\widetilde{F}(z)\widetilde{Y}(z),
\end{array}\right.$$
where 
$$Y(\z,q):=\Big(\dq^{i}S_{q}^{[0]}\left(\hat{h}\right) G_{i}(q)\Big)_{i\in \{0,\dots, m-1\}},\quad  \widetilde{Y}(\z):=\Big(\d^{i}\widetilde{S}^{0}\left(\widetilde{h}\right)\Big)
_{i\in \{0,\dots,m-1\}},$$
$$z\mapsto \mathrm{Id}+(q-1)F(z,q)\in \mathrm{GL}_{m}\Big(\C(z)\Big),\widetilde{F}(z)\in \mathrm{M}_{m}\Big(\C(z)\Big),$$
and ~$G_{i}(q)\in \C$ are defined such that:
$$\left.\dq^{i}S_{q}^{[0]}\left(\hat{h}\right)G_{i}(q)\right|_{z=z_{0}}=\left. \d^{i}
\widetilde{S}^{0}\left(\widetilde{h}\right)\right|_{z=z_{0}}.$$
The assumption~$(i)$ of Proposition~\ref{4propo3} is satisfied because of the assumption \textbf{(A3)}, and the two others are trivially satisfied, since~$K_{0}$ is bounded. \par 
 This yields~${\lim\limits_{q \to 1}S_{q}^{[0]}\left(\hat{h}\right)=\widetilde{S}^{0}\left(\widetilde{h}\right)}$
uniformly on~$K_{0}$, and completes the proof in the particular case~${z\mapsto\hat{h}(z,q),\widetilde{h}(z)\in \C\left[\left[z^{\b}\right]\right]}$.

\pagebreak[3]
\subsection{Proof of Theorem~\ref{4theo1} in the general case.}\label{4sec63}
In this subsection, we are going to prove Theorem~\ref{4theo1} in the general case. See~$\S \ref{4sec1}$ to~$\S \ref{4sec4}$ for the notations. We recall that for all~$l\in \{0,\dots,\b-1\}$, we define $\hat{h}^{(l)}\in \C\big[\big[z^{\b}\big]\big]$, so that~$\hat{h}=\displaystyle \sum_{l=0}^{\b-1} z^{l}\hat{h}^{(l)}$. Let us set~$\Sigma_{\widetilde{h}}:=\displaystyle \bigcup_{l=0}^{\b-1}\Sigma_{\widetilde{h}^{(l)}}$ (see Step 1 in~$\S \ref{4sec62}$).
Let~$d\in  \R\setminus \Sigma_{\widetilde{h}}$.  After considering~$z\mapsto ze^{-id}$, we may assume that~$d=0$. \par 
  Looking at the term with~$z$-degree congruent to~$j$ modulo~$\b$ , for~$j=0,\dots,j=\b-1$, we find that the equation satisfied by~$\hat{h}$ is equivalent to the following family of~$\dq$-linear equations:
$$\left\{\begin{array}{lcl}
0&=&\displaystyle \sum_{k,l} d_{0,k,l}(z,q)\dq^{k}\hat{h}^{(l)}(z,q)\\
&\vdots&\\
0&=&\displaystyle \sum_{k,l} d_{\b-1,k,l}(z,q)\dq^{k}\hat{h}^{(l)}(z,q),
\end{array}\right.$$
where ~$z\mapsto d_{j,k,l}\in \C\left[z^{\b}\right]~$.
Let~$l\in \{0,\dots,\b-1\}$. Following the equalities
$$
z^{l}\hat{h}^{(l)}(z,q)=\displaystyle \displaystyle\sum_{j=0}^{\b-1}\frac{\hat{h}\left(e^{2i\pi lj/\b}z,q\right)}{e^{2i\pi lj/\b}\b},\quad
z^{l}\widetilde{h}^{(l)}(z)=\displaystyle\sum_{j=0}^{\b-1}\frac{\widetilde{h}\left(e^{2i\pi lj/\b}z\right)}{e^{2i\pi lj/\b}\b},
$$
we obtain that for all~$l\in \{0,\dots,\b-1\}$,~$\hat{h}^{(l)}(z,q)$ (resp.~$\widetilde{h}^{(l)}$) satisfies a linear~$q$-difference (resp. differential) equation with coefficients in~$\C\big[z^{\b}\big]$. Moreover, for all~${l\in \{0,\dots,\b-1\}}$,~$\hat{h}^{(l)}$, converges coefficientwise to~$\widetilde{h}^{(l)}$ and the equations they satisfy have coefficients that check the assumptions \textbf{(A2)} and \textbf{(A3)}. \par 
Because of the fact that~$0\in  \R\setminus \widetilde{\Sigma}_{\widetilde{h}}$, Proposition~\ref{4propo5} implies that for all~${l\in \{0,\dots,\b-1\}}$, there exists~$\widetilde{S}^{0}\left(\widetilde{h}^{(l)}\right)$, asymptotic solution of the same linear~$\d$-equation than~$\widetilde{h}^{(l)}$. These latter can be computed with Laplace and Borel transformations.\par 
Using the case~$z\mapsto\hat{h}(z,q),\widetilde{h}(z)\in \C\left[\left[z^{\b}\right]\right]$ (see~$\S \ref{4sec62}$), we can compute  for~$q$ close to~$1$, and~${l\in \{0,\dots,\b-1\}}$,~${z\mapsto S_{q}^{[0]}\left(\hat{h}^{(l)}\right)\in\mathcal{M}(\C^{*},0)}$, solution of the same family of linear~$\dq$-equations than~$\hat{h}^{(l)}$. Because of Remark~\ref{4rem3}, we find:
$$\left\{\begin{array}{lcl}
0&=&\displaystyle \sum_{k,l} d_{0,k,l}(z,q)\dq^{k}S_{q}^{[0]}\left(\hat{h}^{(l)}\right)\\
&\vdots&\\
0&=&\displaystyle \sum_{k,l} d_{\b-1,k,l}(z,q)\dq^{k}S_{q}^{[0]}\left(\hat{h}^{(l)}\right).
\end{array}\right.$$
Hence, we obtain that for~$q$ close to~$1$,~$S_{q}^{[0]}\left(\hat{h}\right):=\displaystyle \sum_{l=0}^{\b-1} z^{l}S_{q}^{[0]}\left(\hat{h}^{(l)}\right)$ satisfies the same linear~$\dq$-equation than~$\hat{h}$.
We apply now the theorem in the case~${z\mapsto\hat{h}(z,q),\widetilde{h}(z)\in \C\left[\left[z^{\b}\right]\right]}$ previously treated, to prove the existence of~$L_{2}>0$, such that we have
$$\lim\limits_{q \to 1}S_{q}^{[0]}\left(\hat{h}\right)=\widetilde{S}^{0}\left(\widetilde{h}\right):=\displaystyle \sum_{l=0}^{\b-1} z^{l}\widetilde{S}^{0}\left(\widetilde{h}^{(l)}\right),$$
uniformly on the compacts of~$\Big\{z\in \overline{S}\left(-\frac{\pi}{2\k_{r}},+\frac{\pi}{2\k_{r}}\right)\Big| |z|<L_{2}\Big \}$.
To conclude, we have to prove
$$\lim\limits_{q \to 1}S_{q}^{[0]}\left(\hat{h}\right)=\widetilde{S}^{0}\left(\widetilde{h}\right),$$
uniformly on the compacts of~$\overline{S}\left(-\frac{\pi}{2k_{r}},+\frac{\pi}{2k_{r}}\right)\setminus \bigcup \R_{\geq 1}\a_{i}$,  
where~$\a_{i}$ are the roots of ${\widetilde{b}_{m}\in \C[z]}$. This is the same reasoning than for the particular case
${z\mapsto\hat{h}(z,q),\widetilde{h}(z)\in \C\left[\left[z^{\b}\right]\right]}$ (see Step 4 in~$\S \ref{4sec62}$).
This completes the proof of our main result, Theorem~\ref{4theo1}.
\pagebreak[3]
\section{Basic hypergeometric series.}\label{4sec7}

We refer the reader to \cite{GR} for more details about basic hypergeometric series. We recall that~${p=1/q}$. In this section, we will say that two functions are equal if their analytic continuations coincide. Let~$r,s\in \N$, let~$a_{1},\dots,a_{r},b_{1},\dots,b_{s}\in \C\setminus q^{\N}$, with different images in~$\C^{*}/q^{\Z}$, and let us consider the formal power series
$$_{r}\varphi_{s}\left(\begin{array}{ll}
a_{1},\dots,a_{r}&\\
&;p,z\\
b_{1},\dots,b_{s}&
\end{array} \right):=\displaystyle \sum_{n=0}^{\infty}\dfrac{(a_{1};p)_{n}\dots(a_{r};p)_{n}}{(p;p)_{n}(b_{1};p)_{n}\dots(b_{s};p)_{n}}\left((-1)^{n}p^{n(n-1)/2}\right)^{1+s-r}z^{n},$$
where~$(a;p)_{n+1}:=(1-ap^{n})(a;p)_{n}$ and~$(a;p)_{0}:=1$, for~$a\in \C$. Assume now that~$r>s+1$ and~$\displaystyle\prod_{i=1}^{r}a_{i}\neq 0$. In this case, the formal power series is divergent. Let us put~${\underline{p}:=q^{-1/(r-s-1)}}$ and~$\underline{q}:=q^{1/(r-s-1)}$.
\pagebreak[3]
\begin{lem}\label{4lem13}
The series~$_{r}\varphi_{s}\left(\begin{array}{ll}
a_{1},\dots,a_{r}&\\
&;\underline{p},z\\
b_{1},\dots,b_{s}&
\end{array} \right)$ satisfies the linear~$\sigma_{\underline{q}}$-equation
$$\left((\sigma_{\underline{q}}-1)\displaystyle\prod_{i=1}^{s}(\sigma_{\underline{q}}-b_{i}\underline{q})
+z(-1)^{s-r}\underline{q}^{1+s}\sigma_{\underline{q}}^{2+s-r}\displaystyle\prod_{i=1}^{r}(\sigma_{\underline{q}}-a_{i})\right)
\left(_{r}\varphi_{s}\left(\begin{array}{ll}
a_{1},\dots,a_{r}&\\
&;\underline{p},z\\
b_{1},\dots,b_{s}&
\end{array} \right)\right)=0,$$ 
 which admits~$0$ and~$1$ as non negative slopes. 
\end{lem}
\begin{figure}[ht]
\begin{center}
\includegraphics[width=0.3\linewidth]{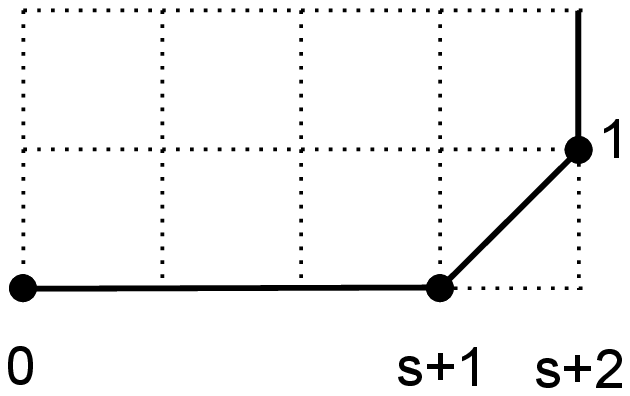}
%\caption{Newton polygon of the $\underline{q}$-difference equation satisfied by  $ _{r}\varphi_{s}$.}
\end{center}
\end{figure}
\begin{proof}
Let us define $(u_{n})_{n\in \N}\in \C^{\N}$ such that $_{r}\varphi_{s}\left(\begin{array}{ll}
a_{1},\dots,a_{r}&\\
&;\underline{p},z\\
b_{1},\dots,b_{s}&
\end{array} \right)=\displaystyle \sum_{n\in \N} u_{n}z^{n}$. Let us fix $n\in \N$. We find that 
$$
u_{n+1}(1-\underline{p}^{n+1})\displaystyle\prod_{i=1}^{s}(1-b_{i}\underline{p}^{n})=u_{n}(-1)^{1+s-r}\underline{q}^{n}\displaystyle\prod_{i=1}^{r}(1-a_{i}\underline{p}^{n}) $$ and then 
$$u_{n+1}(\underline{q}^{n+1}-1)\displaystyle\prod_{i=1}^{s}(\underline{q}^{n+1}-b_{i}\underline{q})=u_{n}(-1)^{1+s-r}\underline{q}^{1+s}\underline{q}^{n(2+s-r)}\displaystyle\prod_{i=1}^{r}(\underline{q}^{n}-a_{i}) .$$
Multiplying the two sides of the equality by $z^{n+1}$, we obtain the result.
\end{proof}

Let~$\a_{1},\dots,\a_{r},\b_{1},\dots,\b_{s}\in \C\setminus (-\N)$ with different images in~$\C/\Z$. If we put~${x:=z\left(1-\underline{p}\right)^{1+s-r}}$, we have the following convergence coefficientwise when~$\underline{p}$ goes to~$1$:
$$ _{r}\varphi_{s}\left(\begin{array}{ll}
\underline{p}^{\a_{1}},\dots,\underline{p}^{\a_{r}}&\\
&;\underline{p},x\\
\underline{p}^{\b_{1}},\dots,\underline{p}^{\b_{s}}&
\end{array} \right)\rightarrow \hbox{} 
_{r}F_{s}\left(
\begin{array}{ll}
\a_{1},\dots,\a_{r}&\\
&;(-1)^{{1+s-r}}z\\
\b_{1},\dots,\b_{s}&
\end{array} 
\right)
,$$
where 
~$$_{r}F_{s}\left(
\begin{array}{ll}
\a_{1},\dots,\a_{r}&\\
&;z\\
\b_{1},\dots,\b_{s}&
\end{array} 
\right):=\displaystyle \sum_{n=0}^{\infty}\dfrac{(\a_{1})_{n}\dots(\a_{r})_{n}}{n!(\b_{1})_{n}\dots(\b_{s})_{n}}z^{n} ,$$
and
~$(\a)_{0}:=1; (\a)_{n+1}:=(\a+n)(\a)_{n}$ for~$\a\in \C$. Applying Lemma \ref{4lem13}, we obtain that the first series satisfies $\D_{\underline{q}}\left(_{r}\varphi_{s}\left(\begin{array}{ll}
\underline{p}^{\a_{1}},\dots,\underline{p}^{\a_{r}}&\\
&;\underline{p},x\\
\underline{p}^{\b_{1}},\dots,\underline{p}^{\b_{s}}&
\end{array} \right)\right)=0$ where

$$\D_{\underline{q}}:=\delta_{\underline{q}}\displaystyle\prod_{i=1}^{s}\left(\delta_{\underline{q}}+\frac{1-\underline{p}^{\b_{i}-1}}{1-\underline{q}}\right)
+z(-1)^{s-r}\underline{q}^{2+2s-r}\sigma_{\underline{q}}^{2+s-r}\displaystyle\prod_{i=1}^{r}\left(\delta_{\underline{q}}+\frac{1-\underline{p}^{\a_{i}}}{1-\underline{q}}\right).$$
Using the same reasoning, one can prove that the second series satisfies  $\widetilde{\D}\left(_{r}F_{s}\left(
\begin{array}{ll}
\a_{1},\dots,\a_{r}&\\
&;(-1)^{1+s-r}z\\
\b_{1},\dots,\b_{s}&
\end{array} 
\right)\right)=0$,
where 
$$\widetilde{\D}:=\d\displaystyle\prod_{i=1}^{s}\left(\d+\b_{i}-1\right)
+z(-1)^{s-r}\displaystyle\prod_{i=1}^{r}\left(\d+\a_{i}\right). $$
\emph{The above series do not satisfy the assumptions of Theorem~\ref{4theo1}}, since the slopes of $\widetilde{\D}$, do not correspond to the slopes of $\D_{\underline{q}}$ that are positive. Notice that the assumptions of \cite{DVZ}, Theorem~2.6, are not satisfied as well.\par 
The goal of this section is to show that if~$d\not \equiv (r-s-1)\pi [2\pi]$, we may apply successively $\hat{B}_{\underline{q}}$ and~$L_{\underline{q}}^{[d]}$, see Definition \ref{4defi6}, to $_{r}\varphi_{s}
\left(
\begin{array}{ll}
\underline{p}^{\a_{1}},\dots,\underline{p}^{\a_{r}}\\
&;\underline{p},x\\
\underline{p}^{\b_{1}},\dots,\underline{p}^{\b_{s}}&
\end{array}\right)$, and prove, by making explicitly the computations, that we obtain a function that converges, as $\underline{q}$  goes to $1$, to $\widetilde{S}^{d}\left(_{r}F_{s}\left(
\begin{array}{ll}
\a_{1},\dots,\a_{r}&\\
&;(-1)^{1+s-r}z\\
\b_{1},\dots,\b_{s}&
\end{array} 
\right) \right)$.
 The case~$r=2$ and~$s=0$ has been treated in \cite{Z02},~$\S 2$. \\ \par 
First, we are going to consider~$ _{r}\varphi_{s}\left(\begin{array}{ll}
a_{1},\dots,a_{r}&\\
&;\underline{p},z\\
b_{1},\dots,b_{s}&
\end{array} \right)$, which satisfies, see Lemma~\ref{4lem13}, a linear~$\sigma_{\underline{q}}$-equation with non negative slopes $0$ and $1$.
 As we can see in \cite{Z02},~$\S 1$, if~$d\not \equiv (r-s-1)\pi [2\pi]$, we can compute a solution of the same linear~$\sigma_{\underline{q}}$-equation than~$ _{r}\varphi_{s}\left(\begin{array}{ll}
a_{1},\dots,a_{r}&\\
&;\underline{p},z\\
b_{1},\dots,b_{s}&
\end{array} \right)$ applying successively to it~$\hat{B}_{\underline{q}}$ and~$L_{\underline{q}}^{[d]}$. Applying ~$\hat{B}_{\underline{q}}$ to~$ _{r}\varphi_{s}\left(\begin{array}{ll}
a_{1},\dots,a_{r}&\\
&;\underline{p},z\\
b_{1},\dots,b_{s}&
\end{array} \right)$, we obtain for all $d \not \equiv (r-s-1)\pi [2\pi]$
$$ h(\z):=\hbox{} 
_{r}\varphi_{r-1}\left(\begin{array}{ll}
a_{1},\dots,a_{r}&\\
&;\underline{p},(-1)^{1+s-r}\z\\
b_{1},\dots,b_{s},0,\dots,0&
\end{array} \right)\in\mathbb{H}_{\underline{q},1}^{d}.$$

 For~$a,a_{1},\dots,a_{k}\in \C$, let~${(a;\underline{p})_{\infty}:=\displaystyle\prod_{n=0}^{\infty}(1-a\underline{p}^{n})}$, and~${(a_{1},\dots,a_{k};\underline{p})_{\infty}:=\displaystyle\prod_{i=1}^{k}(a_{i};\underline{p})_{\infty}}$. For all~${j\in \{1,\dots,k\}}$ and~$a_{1},\dots,a_{k}\in \C$, let~$(a_{1},\dots,\widehat{a_{j}},\dots,a_{k})$ be equals to the finite sequence~$(a_{1},\dots,a_{k})$ after the withdrawn of the element~$a_{j}$.
As we can see in Page 121 of \cite{GR}, the convergent series~$_{r}\varphi_{r-1}$ may be expressed with connection formula at infinity:
$$ \begin{array}{lll}
_{r}\varphi_{r-1}\left(\begin{array}{ll}
a_{1},\dots,a_{r}&\\
&;\underline{p},z\\
b_{1},\dots,b_{r-1}&
\end{array} \right)=\displaystyle \sum_{j=1}^{r}\dfrac{(a_{1},\dots,\widehat{a_{j}},\dots,a_{r},b_{1}/a_{j},\dots,b_{r-1}/a_{j},a_{j}z,\underline{p}/a_{j}z;\underline{p})_{\infty}}
{(b_{1},\dots,b_{r-1},a_{1}/a_{j},\dots,\widehat{a_{j}/a_{j}},\dots,a_{r}/a_{j},z,\underline{p}/z;\underline{p})_{\infty}}&&\\\\
\times \hbox{}
_{r}\varphi_{r-1}\left(
\begin{array}{ll}
a_{j},a_{j}\underline{p}/b_{1},\dots,a_{j}\underline{p}/b_{r-1}&\\
&;\underline{p},\dfrac{\underline{p}\prod_{i=1}^{r-1}b_{i}}{z\prod_{i=1}^{r}a_{i}}\\
a_{j}\underline{p}/a_{1},\dots,\widehat{a_{j}\underline{p}/a_{j}},\dots,a_{j}\underline{p}/a_{r}&
\end{array} 
\right).
\end{array}
$$
Making~$b_{s+1},\dots,b_{r-1}$ goes to~$0$, we find:
$$ \begin{array}{lll}
h(\z)=\displaystyle \sum_{j=1}^{r}\dfrac{(a_{1},\dots,\widehat{a_{j}},\dots,a_{r},b_{1}/a_{j},\dots,b_{s}/a_{j};\underline{p})_{\infty}\T_{\underline{q}}\Big((-1)^{s-r}a_{j}\z\Big)}
{(b_{1},\dots,b_{s},a_{1}/a_{j},\dots,\widehat{a_{j}/a_{j}},\dots,a_{r}/a_{j};\underline{p})_{\infty}\T_{\underline{q}}\Big((-1)^{s-r}\z\Big)}\\\\
\times \hbox{}
_{s+1}\varphi_{r-1}
\left(
\begin{array}{ll}
a_{j},a_{j}\underline{p}/b_{1},\dots,a_{j}\underline{p}/b_{s}\\
&\\
&;\underline{p},\dfrac{(-1)^{1+s-r}\underline{p}a_{j}^{r-s-1}\prod_{i=1}^{s}b_{i}}{\z \prod_{i=1}^{r}a_{i}}\\
a_{j}\underline{p}/a_{1},\dots,\widehat{a_{j}\underline{p}/a_{j}},\dots,a_{j}\underline{p}/a_{r}&
\end{array} \right).
\end{array}
$$
The next lemma gives the expression of the~$\underline{q}$-Laplace transformation of the first term of the sum of~$h$. The expression of the~$\underline{q}$-Laplace transformation of~$h$ will follows directly.
\pagebreak[3]
\begin{lem} 
Let~$d\not \equiv (r-s-1)\pi [2\pi]$,~$\l:=(\underline{q}-1)e^{id}$ and~$\a:=\dfrac{(-1)^{1+s-r}\underline{p}a_{1}^{r-s-2}b_{1}\dots b_{s}}{ a_{2}\dots a_{r}}$. Then,$$L_{\underline{q}}^{[d]}\left(\frac{\T_{\underline{q}}\Big((-1)^{s-r}a_{1}\z\Big)}{\T_{\underline{q}}\Big((-1)^{s-r}\z\Big)}\hbox{}_{s+1}\varphi_{r-1}
\left(
\begin{array}{ll}
a_{1},a_{1}\underline{p}/b_{1},\dots,a_{1}\underline{p}/b_{s}&\\
&;\underline{p},\frac{\a}{\z}\\
a_{1}\underline{p}/a_{2},\dots,a_{1}\underline{p}/a_{r}&
\end{array} \right)\right)$$
is equal to 
$$ \frac{\T_{\underline{q}}\Big((-1)^{s-r}a_{1}\l\Big)}{\T_{\underline{q}}\Big((-1)^{s-r}\l\Big)}\frac{\T_{\underline{q}}\Big(a_{1}z/\l\Big)}{\T_{\underline{q}}\Big(z/\l\Big)}\hbox{}_{s+2}\varphi_{r-1}
\left(
\begin{array}{ll}
a_{1},a_{1}\underline{p}/b_{1},\dots,a_{1}\underline{p}/b_{s},0&\\
&;\underline{p},-\frac{\a}{ a_{1}\underline{p}z}\\
a_{1}\underline{p}/a_{2},\dots,a_{1}\underline{p}/a_{r}&
\end{array}
 \right).$$
\end{lem}

\begin{proof}
Using the expression of~$\T_{\underline{q}}$, we find that for all~$k\in \Z$,$$\T_{\underline{q}}(\underline{q}^{k}z)=\underline{q}^{k(k-1)/2}z^{k}\T_{\underline{q}}(z).$$ Let us write$$f(\z):=\hbox{}_{s+1}\varphi_{r-1}
\left(
\begin{array}{ll}
a_{1},a_{1}\underline{p}/b_{1},\dots,a_{1}\underline{p}/b_{s}&\\
&;\underline{p},\frac{\a}{\z}\\
a_{1}\underline{p}/a_{2},\dots,a_{1}\underline{p}/a_{r}&
\end{array} \right)=:\displaystyle\sum_{l=0}^{\infty}f_{l}\z^{-l}$$
 and$$g(z):=\hbox{}_{s+2}\varphi_{r-1}
\left(
\begin{array}{ll}
a_{1},a_{1}\underline{p}/b_{1},\dots,a_{1}\underline{p}/b_{s},0&\\
&;\underline{p},-\frac{\a}{ a_{1}\underline{p}z}\\
a_{1}\underline{p}/a_{2},\dots,a_{1}\underline{p}/a_{r}&
\end{array}
 \right)=:\displaystyle\sum_{l=0}^{\infty}g_{l}z^{-l}.$$ 
 Then, 
$$
\begin{array}{ll}
&L_{\underline{q}}^{[d]}\left(\dfrac{\T_{\underline{q}}\Big((-1)^{s-r}a_{1}\z\Big)}{\T_{\underline{q}}\Big((-1)^{s-r}\z\Big)}\hbox{}_{s+1}\varphi_{r-1}
\left(
\begin{array}{ll}
a_{1},a_{1}\underline{p}/b_{1},\dots,a_{1}\underline{p}/b_{s}&\\
&;\underline{p},\frac{\a}{\z}\\
a_{1}\underline{p}/a_{2},\dots,a_{1}\underline{p}/a_{r}&
\end{array} \right)\right)\\\\
=&\dfrac{\T_{\underline{q}}\Big((-1)^{s-r}a_{1}\l\Big)}{\T_{\underline{q}}\Big((-1)^{s-r}\l\Big)}
\dfrac{1}{\T_{\underline{q}}\Big(\l\underline{q}/z\Big)}\displaystyle \sum_{n\in \Z}\left(\frac{a_{1}z}{\l}\right)^{n}\underline{q}^{-n(n-1)/2}f(\underline{q}^{n}\l)\\\\
=&\dfrac{\T_{\underline{q}}\Big((-1)^{s-r}a_{1}\l\Big)}{\T_{\underline{q}}\Big((-1)^{s-r}\l\Big)}
\dfrac{1}{\T_{\underline{q}}\Big(\l\underline{q}/z\Big)}
\displaystyle \sum_{n\in \Z}\displaystyle \sum_{l=0}^{\infty}\left(\frac{a_{1}z}{\l}\right)^{n}\underline{q}^{-n(n-1)/2}f_{l}\underline{q}^{-ln}\l^{-l}.
\end{array}
$$

We apply now Fubini's Theorem to conclude that
$$\begin{array}{lll}
&\dfrac{\T_{\underline{q}}\Big((-1)^{s-r}a_{1}\l\Big)}{\T_{\underline{q}}\Big((-1)^{s-r}\l\Big)}
\dfrac{1}{\T_{\underline{q}}\Big(\l\underline{q}/z\Big)}
&\displaystyle \sum_{n\in \Z}\displaystyle \sum_{l=0}^{\infty}\left(\frac{a_{1}z}{\l}\right)^{n}\underline{q}^{-n(n-1)/2}f_{l}\underline{q}^{-ln}\l^{-l}\\\\
=&\dfrac{\T_{\underline{q}}\Big((-1)^{s-r}a_{1}\l\Big)}{\T_{\underline{q}}\Big((-1)^{s-r}\l\Big)}\dfrac{1}{\T_{\underline{q}}\Big(\l\underline{q}/z\Big)}&\displaystyle \sum_{l=0}^{\infty}\displaystyle \sum_{n\in \Z}\left(\frac{a_{1}z}{\l}\right)^{n}\underline{q}^{-n(n-1)/2}f_{l}\underline{q}^{-ln}\l^{-l}\\\\
=&\dfrac{\T_{\underline{q}}\Big((-1)^{s-r}a_{1}\l\Big)}{\T_{\underline{q}}\Big((-1)^{s-r}\l\Big)}\dfrac{1}{\T_{\underline{q}}\Big(\l\underline{q}/z\Big)}&\displaystyle \sum_{l=0}^{\infty}\T_{\underline{q}}\left(\frac{a_{1}z\underline{q}^{-l}}{\l}\right)   f_{l}\l^{-l}\\\\
=&\dfrac{\T_{\underline{q}}\Big((-1)^{s-r}a_{1}\l\Big)}{\T_{\underline{q}}\Big((-1)^{s-r}\l\Big)}\dfrac{\T_{\underline{q}}\Big(a_{1}z/\l\Big)}{\T_{\underline{q}}\Big(\l\underline{q}/z\Big)}&\displaystyle \sum_{l=0}^{\infty}f_{l}\underline{p}^{-l(l-1)/2}a_{1}^{-l}\underline{q}^{l}z^{-l}\\\\
=&\dfrac{\T_{\underline{q}}\Big((-1)^{s-r}a_{1}\l\Big)}{\T_{\underline{q}}\Big((-1)^{s-r}\l\Big)}\dfrac{\T_{\underline{q}}\Big(a_{1}z/\l\Big)}{\T_{\underline{q}}\Big(z/\l\Big)}&\displaystyle \sum_{l=0}^{\infty}g_{l}z^{-l}.
\end{array}$$
\end{proof}
We have proved:
\pagebreak[3]
\begin{theo}
Let~$d\not \equiv (r-s-1)\pi [2\pi]$ and let 
~$\mathbb{S}_{\underline{q}}^{[d]}\left(_{r}\varphi_{s}\right)$ be the function obtained applying successively~$\hat{B}_{\underline{q}}$ and 
~$L_{\underline{q}}^{[d]}$ to~$_{r}\varphi_{s}
\left(
\begin{array}{ll}
a_{1},\dots,a_{r}&\\
&;\underline{p},z\\
b_{1},\dots,b_{s}&
\end{array} \right)$.  Then
$$ \begin{array}{lll}
\mathbb{S}_{\underline{q}}^{[d]}\left(_{r}\varphi_{s}\right)=\displaystyle\sum_{j=1}^{r}
\dfrac{(a_{1},\dots,\widehat{a_{j}},\dots,a_{r},b_{1}/a_{j},\dots,b_{s}/a_{j};\underline{p})_{\infty}\T_{\underline{q}}\Big((-1)^{s-r}a_{j}(1-\underline{q})e^{id}\Big) \T_{\underline{q}}\Big(\frac{a_{j}z}{(1-\underline{q})e^{id}}\Big)}
{(b_{1},\dots,b_{s},a_{1}/a_{j},\dots,\widehat{a_{j}/a_{j}},\dots,a_{r}/a_{j};\underline{p})_{\infty}\T_{\underline{q}}\Big((-1)^{s-r}(1-\underline{q})e^{id}\Big)\T_{\underline{q}}\Big(\frac{z}{(1-\underline{q})e^{id}}\Big)}&&\\\\
\times\hbox{}_{s+2}\varphi_{r-1}
\left(
\begin{array}{ll}
a_{j},a_{j}\underline{p}/b_{1},\dots,a_{j}\underline{p}/b_{s},0&\\
&;\underline{p},\dfrac{(-1)^{s-r}a_{1}^{r-s-2}\prod_{i=1}^{s}b_{i}}{z\prod_{i=1}^{r}a_{i}}\\
a_{j}\underline{p}/a_{1},\dots,\widehat{a_{j}\underline{p}/a_{j}},\dots,a_{1}\underline{p}/a_{r}&
\end{array} \right).
\end{array}$$
\end{theo}
Let~$\a_{1},\dots,\a_{r},\b_{1},\dots,\b_{s}\in \C\setminus -\N$ with different images in~$\C/\Z$. We replace now~$a_{i}$ by~$\underline{p}^{\a_{i}}$,~$b_{i}$ by~$\underline{p}^{\b_{i}}$,~$z$ by~$x=z\left(1-\underline{p}\right)^{1+s-r}$ and consider the limit as~$\underline{p}$ goes to~$1$. It is clear that for all~$j\in \{1,\dots,r\}$, we have the uniform convergence on the compacts of~$\C^{*}$

$$\begin{array}{l}
\lim\limits_{\underline{p} \to 1}\hbox{} 
_{s+2}\varphi_{r-1}
\left(
\begin{array}{ll}
\underline{p}^{\a_{j}},\underline{p}^{\a_{j}-\b_{1}+1},\dots,\underline{p}^{\a_{j}-\b_{s}+1},0&\\
&;\underline{p},\dfrac{(-1)^{s-r}\underline{p}^{\a_{j}(r-s-2)+\b_{1}+\dots +\b_{s}}}{x\underline{p}^{\a_{1}+\dots +\a_{r}}}\\
\underline{p}^{\a_{j}-\a_{1}+1},\dots,\widehat{\underline{p}^{\a_{j}-\a_{j}+1}},\dots,\underline{p}^{\a_{j}-\a_{r}+1}&
\end{array} \right)\\\\
=_{s+1}F_{r-1}\left(
\begin{array}{ll}
\a_{j},\a_{j}-\b_{1}+1,\dots,\a_{j}-\b_{s}+1&\\
&;\frac{(-1)^{s-r}}{z}\\
\a_{j}-\a_{1}+1,\dots,\widehat{\a_{j}-\a_{j}+1},\dots,\a_{j}-\a_{r}+1&
\end{array} \right).
\end{array}
$$
\par 
 As we can see in \cite{Z02},~$\S 2.3$,
\begin{itemize}
\item  For~$\g\in \C$,
$$\lim\limits_{\underline{p} \to 1}\frac{(\underline{p}^{\g},\underline{p})_{\infty}(1-\underline{p})^{\g -1}}{(\underline{p},\underline{p})_{\infty}}=\G (\g)^{-1}.$$
\item We have$$\lim\limits_{\underline{p} \to 1}\frac{\T_{\underline{q}}(\underline{p}^{\g}u)}{\T_{\underline{q}}(u)}=u^{-\g},$$
uniformly on the compacts of~$\{z\in \C^{*}| \arg(-z)\neq \pi\}$.
\end{itemize}
We have proved: 
\pagebreak[3]
\begin{theo}\label{4theo8}
Let~$d\not \equiv (r-s-1)\pi [2\pi]$. Then,
$$ \begin{array}{lll}
\lim\limits_{\underline{p} \to 1}\mathbb{S}_{\underline{q}}^{[d]}\left(_{r}\varphi_{s}
\left(
\begin{array}{ll}
\underline{p}^{\a_{1}},\dots,\underline{p}^{\a_{r}}\\
&;\underline{p},x\\
\underline{p}^{\b_{1}},\dots,\underline{p}^{\b_{s}}&
\end{array}\right) \right)=\displaystyle\sum_{j=1}^{r}
\dfrac{\displaystyle\prod_{i=1}^{s}\Gamma(\b_{i})\displaystyle\prod_{\substack{i=1\\i\neq j}}^{r}\Gamma(\a_{i}-\a_{j})\Big((-1)^{s-r}z\Big)^{-\a_{j}}}
{\displaystyle\prod_{\substack{i=1\\i\neq j}}^{r}\Gamma(\a_{i})\displaystyle\prod_{i=1}^{s}\Gamma(\b_{i}-\a_{j})}&&\\\\
\times
_{s+1}F_{r-1}\left(
\begin{array}{ll}
\a_{j},\a_{j}-\b_{1}+1,\dots,\a_{j}-\b_{s}+1&\\
&;\frac{(-1)^{s-r}}{z}\\
\a_{j}-\a_{1}+1,\dots,\widehat{\a_{j}-\a_{j}+1},\dots,\a_{j}-\a_{r}+1&
\end{array} \right),
\end{array}$$
uniformly on the compacts of~$\{z\in \C^{*}| \arg(-z)\neq d\}$. 
\end{theo}
\pagebreak[3]
\begin{rem}
The right hand side of the limit equals to the function~$
\widetilde{S}^{d}\left(_{r}F_{s}\left(
\begin{array}{ll}
\a_{1},\dots,\a_{r}&\\
&;(-1)^{1+s-r}z\\
\b_{1},\dots,\b_{s}&
\end{array} 
\right) \right)$.
\end{rem}

\pagebreak[3]
\section{Application: Confluence of a basis of meromorphic solutions}\label{4sec8}
 We study a family of linear~$\dq$-equations that discretize a linear~$\d$-equation, and the behavior of the solutions as~$q$ goes to~$1$. 
After introducing some notations in~$\S \ref{4sec81}$, we prove in~$\S \ref{4sec82}$, that a basis of local formal solutions of the family of linear~$\dq$-equations converges to the Hukuhara-Turrittin solution of the differential equation in a sense that we are going to explain. 
We apply this and our main result, Theorem~\ref{4theo1}, to prove the convergence of the~$q$-Stokes matrices to the Stokes matrices of the linear differential equation in~$\S \ref{4sec83}$. In~$\S \ref{4sec84}$, we show how to find the monodromy matrices of the differential equation, as limit of~$q$-solutions when~$q$ tends to~$1$. When~$q$ is real, this extends the results in~$\S 4$ of \cite{S00} in the irregular singular case\footnote{Notice that the results of this section do not allow us to recover Sauloy's Theorem, but are to be considered as an analogous result in a different situation.}. 

\pagebreak[3]
\subsection{Notations}\label{4sec81}
Some of the notations below were already introduced before, see the introduction, but we recall them for the reader's convenience. For~$a\in \C^{*}$ and~$n\in \N^{*}$, let us consider~$\T_{q}(z)=\displaystyle \sum_{n \in \Z} q^{\frac{-n(n+1)}{2}}z^{n}$,~$l_{q}(z):=\dfrac{\d\left(\T_{q}(z)\right)} {\T_{q}(z)}$, ~${\L_{q,a}(z):=\frac{\T_{q}(z)}{\T_{q}(z/a)}}$,~$e_{q^{n}}\left(az^{n}\right)$ and~$e_{q^{n}}\left(az^{-n}\right)$. 
They satisfy the~$q$-difference equations: \\
\begin{itemize}
\item~$\sq \T_{q}(z)=z\T_{q}(z)$.\\
\item~$\sq l_{q}=l_{q} +1$.\\
\item~$\sq \L_{q,a}(z)=a\L_{q,a}(z)$.\\
\item~$\dq e_{q^{n}}\left(az^{n}\right)= a[n]_{q}z^{n}e_{q^{n}}\left(az^{n}\right)$.\\
\item~$\dq e_{q^{n}}\left(az^{-n}\right)= \dfrac{-a[n]_{q}q^{-n}z^{-n}}{1+(q-1)a[n]_{q}q^{-n}z^{-n}}e_{q^{n}}\left(az^{-n}\right)$.
\end{itemize}
Let~$A$ be an invertible matrix with complex coefficients and consider now the decomposition in Jordan normal form~$A=P(DN)P^{-1}$, where~$D:=\mathrm{Diag}(d_{i})$ is diagonal,~$N$ is a nilpotent upper triangular matrix with~$DN=ND$, and~$P$ is an invertible matrix with complex coefficients. Following \cite{S00}, we construct the matrix:
$$
\L_{q,A}:=P\left(\mathrm{Diag}\left(\L_{q,d_{i}}\right)e^{\log(N)l_{q}}\right)P^{-1}\in \mathrm{GL}_{m}\Big(\C\left( l_{q},\left(\L_{q,a}\right)_{a \in \C^{*}}\right)\Big)$$
that satisfies:
$$
\sq \L_{q,A}=A\L_{q,A}=\L_{q,A}A.$$
Let~$a\in \C^{*}$ and consider the corresponding matrix~$(a)\in \mathrm{GL}_{1}(\C)$. By construction, we have~$\L_{q,a}=\L_{q,(a)}$.\par
We now introduce the $q$-exponential of matrices. For $A\in \mathrm{M}_{m}\Big(\C(z)\Big)$, we define:
$$e_{q}(A):=\displaystyle \sum_{n\in \N} \frac{A^{n}}{[n]_{q}^{!}} \in \mathrm{GL}_{m}\Big(\mathcal{M}(\C^{*},0)\Big).$$
\pagebreak[3]
\subsection{Confluence of a basis of local formal solutions}\label{4sec82}
Formally, we have the convergence~$\lim\limits_{q \to 1}\dq =\d$. We want to prove the formal convergence of a basis of solutions of a family of linear~$\dq$-equations to the Hukuhara-Turrittin solution of the corresponding linear~$\d$-equation. First, we will consider the family of equations,
$$\left\{\begin{array}{lllllllll}
\D_{q}&:=&b_{m}(z,q)\d_{q}^{m}&+&b_{m-1}(z,q)\d_{q}^{m-1}&+&\dots&+&b_{0}(z,q)\\\\
\widetilde{\D}&:=&\widetilde{b}_{m}(z)\d^{m}&+&\widetilde{b}_{m-1}(z)\d^{m-1}&+&\dots&+&\widetilde{b}_{0}(z),
\end{array}\right.$$
that satisfies the following assumptions: 
\begin{trivlist}
\item \textbf{(H1)} For all~$i$, for all~$q$ close to~$1$,~$z\mapsto b_{i}(z,q),\widetilde{b}_{i}(z)\in \C[[z]]$.
\item \textbf{(H2)} For all~$i$,~$b_{i}(z,q)$ converges coefficientwise to~$\widetilde{b}_{i}(z)$ when~$q\to 1$.
\item \textbf{(H3)} Viewed as a linear~$\sq$-equation,~$\D_{q}$ has slopes that belongs to~$\Z$. For~$q$ close to~$1$, the Newton polygon of~$\D_{q}$ is independent of~$q$. 
\item \textbf{(H4)} The slopes of~$\widetilde{\D}$ belongs to~$\N$.
\end{trivlist}
We consider now the associated systems 
\begin{equation}\label{4eq7}
\left\{\begin{array}{lll}
\dq Y(z,q)&=&B(z,q)Y(z,q)\\\\
\d \widetilde{Y}(z)&=&\widetilde{B}(z)\widetilde{Y}(z),
\end{array}\right.
\end{equation}
with~$z\mapsto \mathrm{Id}+(q-1)B(z,q)\in\mathrm{GL}_{m}\Big(\C((z))\Big),\widetilde{B}(z)\in \mathrm{M}_{m}\Big(\C((z))\Big)$.
From Theorem~\ref{4theo5} and the Hukuhara-Turrittin theorem (see~$\S \ref{4sec1}$), we have the existence of
\begin{itemize}
\item~$z\mapsto \hat{H}(z,q),\widetilde{H}(z)\in \mathrm{GL}_{m}\Big(\C((z))\Big)$, such that the entries of the first row of~$\hat{H}(z,q)$ have~$z$-valuation equal to~$0$,
\item~$\mu_{i}\in \Z$, and matrices 
~$ B_{i}(q)\in \mathrm{GL}_{m'_{i}}(\C)$, which are of the form ~$\mathrm{Diag}_{l}\Big(T_{i,l}(q)\Big)$ where~$T_{i,l}(q)$ are upper triangular matrices with diagonal terms equal to the roots of the characteristic polynomial associated to the slope~$\mu_{i}$,
\item~$\widetilde{\l}_{i}(z)\in z^{-1}\C[z^{-1}]$,~$\widetilde{L}_{i}\in \mathrm{M}_{m_{i}}(\C)$,
\end{itemize}
 such that 
\begin{equation}\label{4eq11}
\left\{\begin{array}{lll}
\hat{H}(z,q)\left[\mathrm{Diag}\Big(B_{i}(q)z^{-\mu_{i}}\Big)\right]_{\displaystyle\sq}=\mathrm{Id}+(q-1)B(z,q)\\\\
\widetilde{H}(z)\left[\mathrm{Diag} \left(\widetilde{L}_{i}+\d\widetilde{\l}_{i}(z)\times\mathrm{Id}_{m_{i}}\right)\right]_{\displaystyle\d}=\widetilde{B}(z).
\end{array}\right.\end{equation}
We make two more assumptions:
\begin{trivlist}
\item \textbf{(H5)} For~$q$ close to~$1$,~$\mathrm{Diag}\Big(B_{i}(q)z^{-\mu_{i}}\Big)$ commutes with~$\mathrm{Diag} \left(\widetilde{L}_{i}+\d\widetilde{\l}_{i}(z)\times\mathrm{Id}_{m_{i}}\right)$.
\item \textbf{(H6)}
If~$\widetilde{H}'(z)$ is any formal matrix solution of the differential system of (\ref{4eq11}), then the entries of the first row of~$\widetilde{H}'(z)$ have necessarily ~$z$-valuation equal to~$0$.  Moreover, we assume that the term of lower degree of each entry of the first row of~$\hat{H}(z,q)$ converges as~$q$ goes to~$1$, to the term of lower degree of the corresponding entry of~$\widetilde{H}(z)$.
\end{trivlist}

\pagebreak[3]
\begin{rem}\label{4rem2}
\begin{trivlist}
\item (1) Assumptions \textbf{(H1)} to \textbf{(H4)} are satisfied if the~$b_{i}(z,q)\in \C[[z]]$ are independent of~$q$ and if the slopes of~$\D_{q}$, viewed as a linear~$\sq$-equation, belong to~$\Z$.
\item (2) As we can see in \cite{RSZ}, Theorem 2.2.1, up to a ramification, we may always reduce to the case where the slopes of~$\D_{q}$, viewed as a linear~$\sq$-equation, belong to~$\Z$. Up to a ramification, we may also reduce to the case where \textbf{(H4)} is satisfied.
\item (3) Assumption \textbf{(H5)} is satisfied if and only if, for all~$q$ close to~$1$, 
~$\mathrm{Diag}\Big(B_{i}(q)\Big)$ commutes with~$\mathrm{Diag} \left(\widetilde{L}_{i}\right)$.  If Assumption \textbf{(H5)} is satisfied, then the the blocks of~$\mathrm{Diag}\Big(B_{i}(q)z^{-\mu_{i}}\Big)$ and~$\mathrm{Diag} \left(\widetilde{L}_{i}+\d\widetilde{\l}_{i}(z)\times\mathrm{Id}_{m_{i}}\right)$ have the same size.
\item (4) If, for~$q$ close to~$1$, the~$B_{i}(q)$ and~$\widetilde{L}_{i}$ are diagonal, we may perform shearing transformations on the differential system (resp. a diagonal gauge transformation that depends only upon~$q$ on the~$q$-difference system), in order to change the entries~$\widetilde{l}_{i,j}$ of~$\widetilde{L}_{i}$  by~$\widetilde{l}_{i,j}+k_{i,j}$ where~$k_{i,j}\in \Z$ (resp. multiply to the right~$\hat{H}(z,q)$ by a diagonal complex matrix), and to reduce to the case where \textbf{(H6)} is satisfied. Notice that in this case, \textbf{(H5)} was already satisfied because of the point (3) of the remark. The~$B_{i}(q)$ and~$\widetilde{L}_{i}$ are diagonal if for~$q$ close to~$1$, the multiplicities of the slopes of~$\D_{q}$, viewed as a linear~$\sq$-equation, (resp. the multiplicities of the slopes of~$\widetilde{\D}$) are all equal to~$1$. A weaker condition for~$B_{i}(q)$ and~$\widetilde{L}_{i}$ being diagonal is to assume that~$\D_{q}$, viewed as a linear~$\sq$-equation, and~$\widetilde{\D}$ have exponents at~$0$ which are not resonant. 
\item (5) If Assumption \textbf{(H6)} is satisfied, then the~$\C$-vectorial subspace of~$\mathrm{M}_{m}\Big(\C((z))\Big)$ of solutions of the differential system of (\ref{4eq11}) has dimension~$1$. Remark that the converse in not true.
\item (6) The slopes of~$\widetilde{\D}$ and~$\D_{q}$, viewed as a linear~$\sq$-equation, may be different. We will make assumptions on the slopes in~$\S \ref{4sec83}$ and~$\S\ref{4sec84}$.
\end{trivlist} 
\end{rem}

\pagebreak[3]
\begin{defi}
 We say that the~$m\times m$ invertible square matrix~$F(z,q)$ belongs to~$\mathcal{O}_{m}^{*}$, if for~$q$ close to~$1$, the entries of~${z\mapsto F(z,q)}$ are meromorphic on~$\C^{*}$, and~$F(z,q)$ satisfies
\begin{itemize}
\item We have the uniform convergence~$\lim\limits_{q \to 1} \Big(\dq F(z,q)\Big)F(z,q)^{-1}=0,$ on the compacts of~$\C^{*}$.
\item We have the uniform convergence~$\lim\limits_{q \to 1} F(z,q)=\mathrm{Id},$
 on the compacts of~$\C^{*}$.
\end{itemize}
\end{defi}

\pagebreak[3]
\begin{rem}\label{4rem4}
Roughly speaking, the matrices ~$\hat{H}(z,q)\mathrm{Diag}\Big(\L_{q,B_{i}(q)}\T_{q}(z)^{-\mu_{i}}\Big)$ and 
~$\widetilde{H}(z)\mathrm{Diag}\left(e^{\log(z)\widetilde{L}_{i}}e^{\widetilde{\l}_{i}(z)\times\mathrm{Id}_{m_{i}}}\right)$ are fundamental solutions of the systems (\ref{4eq7}). Let us write~${\widetilde{\l}_{i}(z):=\displaystyle \sum_{j=1}^{k_{i}} \widetilde{\l}_{i,j}z^{-j}}$ with~${k_{i}\in \N}$. The next theorem says that there exists a fundamental solution of~${\dq Y(z,q)=B(z,q)Y(z,q)}$ of the form:
$$\hat{H}(z,q)F_{1}(z,q)F_{2}(z,q)\mathrm{Diag}_{i}\left(\L_{q,\mathrm{Id}+(q-1)\widetilde{L}_{i}}\displaystyle \prod_{j=1}^{k_{i}}e_{q^{j}}\left(\widetilde{\l}_{i,j}z^{-j}\times\mathrm{Id}_{m_{i}}\right)\right),$$
such that:
\begin{itemize}
\item~$z\mapsto F_{1}(z,q)\in \mathrm{GL}_{m}\Big(\C\{z\}\Big)$ and the matrix~$\hat{H}(z,q)F_{1}(z,q)\in \mathrm{GL}_{m}\Big(\C((z))\Big)$ converge entrywise and coefficientwise to~$\widetilde{H}(z)$ when~$q\to 1$.
\item The matrix~$F_{2}(z,q)$ belongs to~$\mathcal{O}_{m}^{*}$ and therefore, for~$z\in \C^{*}$,~$\lim\limits_{q \to 1} F_{2}(z,q)=\mathrm{Id}$.
\item Because of what is written in Page 1048 of \cite{S00} and Lemma~\ref{4lem7}, for all~${z\in \C^{*}\setminus \R_{<0}}$, we have the convergence 
$$\lim\limits_{q \to 1}\mathrm{Diag}_{i}\left(\L_{q,\mathrm{Id}+(q-1)\widetilde{L}_{i}}\displaystyle \prod_{j=1}^{k_{i}}e_{q^{j}}\left(\widetilde{\l}_{i,j}z^{-j}\times\mathrm{Id}_{m_{i}}\right)\right)=\mathrm{Diag}\left(e^{\log(z)\widetilde{L}_{i}}e^{\widetilde{\l}_{i}(z)\times\mathrm{Id}_{m_{i}}}\right).$$ 
\end{itemize}
In other words, the above fundamental solution of~$\dq Y(z,q)=B(z,q)Y(z,q)$ formally converges to the fundamental solution~$\widetilde{H}(z)\mathrm{Diag}\left(e^{\log(z)\widetilde{L}_{i}}e^{\widetilde{\l}_{i}(z)\times\mathrm{Id}_{m_{i}}}\right)$ of~$\d \widetilde{Y}(z)=\widetilde{B}(z)\widetilde{Y}(z)$ given by the Hukuhara-Turrittin Theorem. Of course, written like this, this statement is not rigorous since the matrices can not be multiplied among them, see~$\S \ref{4sec1}$.
\end{rem}

\pagebreak[3]
\begin{theo}\label{4theo2}
Let us consider the systems (\ref{4eq7}) that satisfies the assumptions \textbf{(H1)} to \textbf{(H6)}. Let~$B_{i}(q)$,~$\widetilde{L}_{i}$ and~${\widetilde{\l}_{i}(z)=\displaystyle \sum_{j=1}^{k_{i}} \widetilde{\l}_{i,j}z^{-j}}$ that come from (\ref{4eq11}).
\begin{trivlist}
\item  (1) There exist~$z\mapsto F_{1}(z,q)\in \mathrm{GL}_{m}\Big(\C\{z\}\Big)$,~$F_{2}(z,q)\in \mathcal{O}_{m}^{*}$,~$z\mapsto N(z,q)\in \mathrm{M}_{m}\Big(\C(z)\Big)$ such that 
$$F_{1}(z,q)\Big[\mathrm{Id}+(q-1)N(z,q)\Big]_{\displaystyle\sq}=\mathrm{Diag}\Big(B_{i}(q)z^{-\mu_{i}}\Big),$$
where~$N(z,q)$ satisfies:
$$\begin{array}{ll}
\dq \left(F_{2}(z,q)\mathrm{Diag}_{i}\left(\L_{q,\mathrm{Id}+(q-1)\widetilde{L}_{i}}\displaystyle \prod_{j=1}^{k_{i}}e_{q^{j}}\left(\widetilde{\l}_{i,j}z^{-j}\times\mathrm{Id}_{m_{i}}\right)\right)\right)&=\\\\
N(z,q)F_{2}(z,q)\mathrm{Diag}_{i}\left(\L_{q,\mathrm{Id}+(q-1)\widetilde{L}_{i}}\displaystyle \prod_{j=1}^{k_{i}}e_{q^{j}}\left(\widetilde{\l}_{i,j}z^{-j}\times\mathrm{Id}_{m_{i}}\right)\right).&
\end{array}$$
\item (2) The matrix~$\hat{H}(z,q)F_{1}(z,q)$ converges entrywise to~$\widetilde{H}(z)$ when~$q\to 1$. Moreover, there exists~$N\in \N$, such that for~$q$ close to~$1$,~$z\mapsto z^{N}\hat{H}(z,q)F_{1}(z,q)$ belongs to~$\mathrm{M}_{m}\Big(\C[[z]]\Big)$. 
\end{trivlist}
\end{theo}

Notice that the point (2) implies in particular that~$z^{N}\widetilde{H}(z)\in\mathrm{M}_{m}\Big(\C[[z]]\Big)$. Before proving the theorem, we state and prove a lemma:

\pagebreak[3]
\begin{lem}\label{4lem1}
Let us consider an invertible complex matrix that depends upon~$q$,~$A(q)$, and assume the existence of~$k\in \N^{*}$, such that we have the simple convergence~${\lim\limits_{q \to 1}A(q)^{-1}(q-1)^{k}=0\in \mathrm{M}_{m}(\C)}$. Let~$n\in \Z$. There exist 
\begin{itemize}
\item~$z \mapsto E_{1}(z,q)\in \mathrm{GL}_{m}\Big(\C\{z\}\Big)$
\item ~$ F_{2}(z,q)\in \mathcal{O}_{m}^{*}$
\end{itemize}
 such that$$\sq \Big(E_{1}(z,q)F_{2}(z,q)\Big)=z^{n}A(q)E_{1}(z,q)F_{2}(z,q)=E_{1}(z,q)F_{2}(z,q)A(q)z^{n}.$$
\end{lem}

\pagebreak[3]
\begin{ex}
Let us solve~$\sq Y(z,q)=\frac{z}{(q-1)^{2}}Y(z,q)$ with solution in the same form as in the lemma. The trick of the proof of the lemma is the following identity that is valid for all~$z\in \C^{*}$:
$$\frac{z}{(q-1)^{2}}=\frac{1+\frac{z}{(q-1)^{2}}}{1+\frac{(q-1)^{2}}{z}}.$$
We may take~$E_{1}(z,q):=e_{q}\left(\frac{z}{(q-1)^{3}}\right)$ that satisfies~$\sq \left(\frac{z}{(q-1)^{3}}\right)=\left(1+\frac{z}{(q-1)^{2}}\right)e_{q}\left(\frac{z}{(q-1)^{3}}\right)$ and~${F_{2}(z,q):=e_{q}\left(\frac{q(q-1)}{z}\right)}$ that satisfies~$\sq e_{q}\left(\frac{q(q-1)}{z}\right)=\frac{1}{1+\frac{(q-1)^{2}}{z}}e_{q}\left(\frac{q(q-1)}{z}\right)$.
\end{ex}

\begin{proof}[Proof of Lemma~\ref{4lem1}]
For,~$l,d\in \N^{*}$ with~$l\geq 2$, let us define the function ${f_{l,d}:=e_{q^{d}}\left(\frac{z^{d}}{(q-1)^{l+1}[d]_{q}}\right)e_{q^{d}}\left(\frac{q^{d}(q-1)^{l-1}}{[d]_{q}z^{d}}\right)}$,
that satisfies:
$$\sq f_{l,d}=\frac{z^{d}}{(q-1)^{l}} f_{l,d}=f_{l,d}\frac{z^{d}}{(q-1)^{l}},$$
with ~$z \mapsto e_{q^{d}}\left(\frac{z^{d}}{(q-1)^{l+1}[d]_{q}}\right)\in \C\{z\}$ and 
~$e_{q^{d}}\left(\frac{q^{d}(q-1)^{l-1}}{[d]_{q}z^{d}}\right)\in \mathcal{O}_{1}^{*}$. 
Let us also consider~${z\mapsto e_{q}\left(\frac{zA(q)}{(q-1)^{k+2}}\right) \in \mathrm{GL}_{m}\Big(\C\{z\}\Big)}$ and~$e_{q}\left(\frac{q(q-1)^{k}A^{-1}(q)}{z}\right)\in \mathrm{GL}_{m}\Big(\C\{z^{-1}\}\Big)$. We can prove that they satisfy 
$$\sq e_{q}\left(\frac{zA(q)}{(q-1)^{k+2}}\right)=e_{q}\left(\frac{zA(q)}{(q-1)^{k+2}}\right)\left(\mathrm{Id}+\frac{zA(q)}{(q-1)^{k+1}}\right)
=\left(\mathrm{Id}+\frac{zA(q)}{(q-1)^{k+1}}\right)e_{q}\left(\frac{zA(q)}{(q-1)^{k+2}}\right)$$
and 
$$\begin{array}{lcl}
\sq e_{q}\left(\dfrac{q(q-1)^{k}A^{-1}(q)}{z}\right)&=&\left(\mathrm{Id}+\dfrac{(q-1)^{k+1}A^{-1}(q)}{z}\right)^{-1}e_{q}\left(\frac{q(q-1)^{k}A^{-1}(q)}{z}\right)\\\\
&=&e_{q}\left(\frac{q(q-1)^{k}A^{-1}(q)}{z}\right)\left(\mathrm{Id}+\dfrac{(q-1)^{k+1}A^{-1}(q)}{z}\right)^{-1}.
\end{array}$$
Hence, $e_{q}\left(\frac{q(q-1)^{k}A^{-1}(q)}{z}\right)\in \mathcal{O}_{m}^{*}$ and we have:
$$\begin{array}{ll}
&\sq \left(e_{q}\left(\dfrac{zA(q)}{(q-1)^{k+2}}\right)e_{q}\left(\dfrac{q(q-1)^{k}A^{-1}(q)}{z}\right)\right)\\\\
=&\dfrac{zA(q)}{(q-1)^{k+1}}e_{q}\left(\dfrac{zA(q)}{(q-1)^{k+2}}\right)e_{q}\left(\dfrac{q(q-1)^{k}A^{-1}(q)}{z}\right)\\\\
=&e_{q}\left(\dfrac{zA(q)}{(q-1)^{k+2}}\right)e_{q}\left(\dfrac{q(q-1)^{k}A^{-1}(q)}{z}\right)\dfrac{zA(q)}{(q-1)^{k+1}}.
\end{array}$$
Let us choose~$d_{1},d_{2},l_{1},l_{2}\in \N^{*}$ with~$l_{1},l_{2}\geq 2$, such that~$d_{1}-d_{2}+1=n$ and~$l_{1}+(k+1)=l_{2}$. 
Then,$$f_{l_{1},d_{1}}\left(f_{l_{2},d_{2}}\right)^{-1}e_{q}\left(\frac{zA(q)}{(q-1)^{k+2}}\right)e_{q}\left(\frac{q(q-1)^{k}A^{-1}(q)}{z}\right)
 ,$$
is solution of~$\sq Y(z,q)=z^{n}A(q)Y(z,q)=Y(z,q)A(q)z^{n}$ and admits a decomposition that has the required property.
\end{proof}

\begin{proof}[Proof of Theorem~\ref{4theo2}]
\begin{trivlist}
\item (1)
Let us define
$$W_{1}(z,q):=\mathrm{Diag}_{i}\left(\displaystyle \prod_{j=1}^{k_{i}}e_{q^{j}}\left(\frac{q^{j}z^{j}}{\widetilde{\l}_{i,j}(q-1)^{2}[j]_{q}^{2}}\times\mathrm{Id}_{m_{i}}\right)\right)$$
and 
$$W_{2}(z,q):= \mathrm{Diag}_{i}\left(\displaystyle \prod_{j=1}^{k_{i}}e_{q^{j}}\left(\widetilde{\l}_{i,j}z^{-j}\times\mathrm{Id}_{m_{i}}\right)\right),$$
which satisfy 
$$\sq \Big(W_{1}(z,q)W_{2}(z,q)\Big)=\mathrm{Diag}_{i}\left(
\displaystyle \prod_{j=1}^{k_{i}}\frac{q^{j}z^{j}}{(q-1)[j]_{q}\widetilde{\l}_{i,j}} \times\mathrm{Id}_{m_{i}}\right)W_{1}(z,q)W_{2}(z,q).$$
Because of (\ref{4eq11}),~$\mathrm{Diag}_{i}\left(
\displaystyle \prod_{j=1}^{k_{i}}\frac{q^{j}z^{j}}{(q-1)[j]_{q}\widetilde{\l}_{i,j}}\times\mathrm{Id}_{m_{i}} \right)$ commutes with~$\mathrm{Diag}\left(\L_{q,\mathrm{Id}+(q-1)\widetilde{L}_{i}}\right)$ and we obtain that:
$$
\begin{array}{l}
\sq \Bigg(\mathrm{Diag}\left(\L_{q,\mathrm{Id}+(q-1)\widetilde{L}_{i}}\right)W_{1}(z,q)W_{2}(z,q)\Bigg)=\\\\
\mathrm{Diag}_{i}\left(\left(\mathrm{Id}+(q-1)\widetilde{L}_{i}\right)\displaystyle \prod_{j=1}^{k_{i}}\frac{q^{j}z^{j}}{(q-1)[j]_{q}\widetilde{\l}_{i,j}}\right)\mathrm{Diag}\left(\L_{q,\mathrm{Id}+(q-1)\widetilde{L}_{i}}\right)W_{1}(z,q)W_{2}(z,q).
\end{array}
$$
Let
$$
\begin{array}{lll}
C(z,q)&:=&
\mathrm{Diag}\Big(B_{i}(q)z^{-\mu_{i}}\Big)\mathrm{Diag}_{i}\left(\left(\left(\mathrm{Id}+(q-1)\widetilde{L}_{i}\right)\displaystyle \prod_{j=1}^{ k_{i}}\frac{q^{j}z^{j}}{(q-1)[j]_{q}\widetilde{\l}_{i,j}}\right)^{-1}\right)\\\\
&=:&\mathrm{Diag}\Big(C_{i}(q)z^{n_{i}}\Big).
\end{array}$$
If we are able to construct ~$z\mapsto E_{1}(z,q)\in \mathrm{GL}_{m}\Big(\C\{z\}\Big)$ and~$ F_{2}(z,q)\in \mathcal{O}_{m}^{*}$, that commute with~$\mathrm{Diag}\Big(B_{i}(q)z^{-\mu_{i}}\Big)$ and are solution of$$\sq \Big(E_{1}(z,q)F_{2}(z,q)\Big)=C(z,q)E_{1}(z,q)F_{2}(z,q)=E_{1}(z,q)F_{2}(z,q)C(z,q),$$ then the following matrix would be a fundamental solution of the linear~$\sq$-equation~${\sq Y(z,q)=\mathrm{Diag}\Big(B_{i}(q)z^{-\mu_{i}}\Big)Y(z,q)}$:
\begin{equation}\label{4eq26}
E_{1}(z,q)F_{2}(z,q)\mathrm{Diag}\left(\L_{q,\mathrm{Id}+(q-1)\widetilde{L}_{i}}\right)W_{1}(z,q)W_{2}(z,q).
\end{equation}
Let us construct the matrices~$E_{1}$ and~$F_{2}$ using Lemma~\ref{4lem1} applied on each block~$C_{i}(q)$.
Let us check that the matrices~$q\mapsto C_{i}(q)$ satisfies the assumptions of Lemma~\ref{4lem1}. Since the matrices~$\left(\left(\mathrm{Id}+(q-1)\widetilde{L}_{i}\right)\displaystyle \prod_{j=1}^{k_{i}}\frac{q^{j}z^{j}}{(q-1)[j]_{q}\widetilde{\l}_{i,j}}\right)^{-1}$ satisfy the assumptions of Lemma~\ref{4lem1}, it is sufficient to prove that the matrices~$B_{i}(q)$ satisfy the assumptions of Lemma~\ref{4lem1}. Using Theorem~\ref{4theo5}, the~$B_{i}(q)$ are of the form~$\mathrm{Diag}_{l}\Big(T_{i,l}(q)\Big)$ where~$T_{i,l}(q)$ are upper triangular matrices with diagonal terms equal to the roots of the characteristic polynomial associated to the slope~$\mu_{i}$. We recall that the linear~$\dq$-equation is$$\D_{q}:=b_{m}(z,q)(\d_{q})^{m}+b_{m-1}(z,q)(\d_{q})^{m-1}+\dots+b_{0}(z,q),$$
where the~$b_{i}$ converge coefficientwise when~$q\to 1$. Since for all~$n\in \N$,$$\dq^{n}=(q-1)^{-n}\sum_{k=0}^{n}\binom{n}{k}(-1)^{n-k}\sq^{k},$$ a straightforward computation shows that each root of the characteristic polynomial associated to a slope different from zero (resp. to the slope zero) is of the form~$\a(q)(q-1)$ (resp. $\a(q)$), where~$\a(q)$ converges to a non zero complex number. Therefore, each diagonal term of a~$B_{i}(q)$ is of the form $\a(q)(q-1)$ or $\a(q)$, where~$\a(q)$ converges to a non zero complex number. We recall, see (\ref{4eq11}), that the matrix $\hat{H}(z,q)$ satisfies 
$$\sq \left(\hat{H}(z,q)\right)\mathrm{Diag}\Big(B_{i}(q)z^{-\mu_{i}}\Big)=\Big(\mathrm{Id}+(q-1)B(z,q)\Big)\hat{H}(z,q).$$
Using the convergence of the constant terms of the entries in the first row of $\hat{H}(z,q)$, see the assumption \textbf{(H6)}, and the behavior of the diagonal terms of the $B_{i}(q)$, we find that each non diagonal term of a triangular matrix $B_{i}(q)$ is of the form $\a(q)(q-1)$ or $\a(q)$, where~$\a(q)$ converges to a non zero complex number. Hence, for all~$i$,~$B_{i}(q)^{-1}(q-1)^{2}$ simply converges to~$0$ as~$q$ goes to~$1$.\par
Applying Lemma~\ref{4lem1}, there exist~$z\mapsto E_{1}(z,q)\in \mathrm{GL}_{m}\Big(\C\{z\}\Big)$ and~$ F_{2}(z,q)\in \mathcal{O}_{m}^{*}$
that satisfy
$$\sq \Big(E_{1}(z,q)F_{2}(z,q)\Big)=\mathrm{Diag}\Big(C_{i}(q)z^{n_{i}}\Big)E_{1}(z,q)F_{2}(z,q)=E_{1}(q)F_{2}(z,q)\mathrm{Diag}\Big(C_{i}(q)z^{n_{i}}\Big).$$
Because of \textbf{(H5)} and the construction of~$E_{1}(z,q)$ and~$F_{2}(z,q)$ (see the proof of  Lemma~\ref{4lem1}), we obtain that they commute with~$\mathrm{Diag}\Big(B_{i}(q)z^{-\mu_{i}}\Big)$.\par 
We have proved that the matrix (\ref{4eq26}) is a fundamental solution of the system $${\sq Y(z,q)=\mathrm{Diag}\Big(B_{i}(q)z^{-\mu_{i}}\Big)Y(z,q).}$$
We have the following relation:
$$\begin{array}{l}
\sq \left(\mathrm{Diag}
\left(\L_{q,\mathrm{Id}+(q-1)\widetilde{L}_{i}}\right)W_{2}(z,q)\right)=\\\\
\mathrm{Diag}_{i}\left(\left(\mathrm{Id}+(q-1)\widetilde{L}_{i}\right)\displaystyle \prod_{j=1}^{k_{i}}\left(1+\frac{q^{j}z^{j}}{(q-1)[j]_{q}\widetilde{\l}_{i,j}} \right)\right) \mathrm{Diag}
\left(\L_{q,\mathrm{Id}+(q-1)\widetilde{L}_{i}}\right)W_{2}(z,q).
\end{array}$$
Using \textbf{(H5)} and the construction of~$F_{2}(z,q)$, we find that~$F_{2}(z,q)$ commutes with~$\mathrm{Diag}_{i}\left(\left(\mathrm{Id}+(q-1)\widetilde{L}_{i}\right)\displaystyle \prod_{j=1}^{k_{i}}\left(1+\frac{q^{j}z^{j}}{(q-1)[j]_{q}\widetilde{\l}_{i,j}}\right)\right)$. From the construction of~$F_{2}(z,q)$, we find also that~$\sq \Big(F_{2}(z,q)\Big)F_{2}(z,q)^{-1}\in \mathrm{GL}_{m}\Big(\C(z)\Big)$.
Let 
$$U(z,q):=F_{2}(z,q)\mathrm{Diag}
\left(\L_{q,\mathrm{Id}+(q-1)\widetilde{L}_{i}}W_{2}(z,q)\right).$$
From what is preceding, we obtain the existence of~${z\mapsto N(z,q)\in \mathrm{M}_{m}\Big(\C(z)\Big)}$, such that~${\dq U(z,q)=N(z,q)U(z,q)}$.\par 
Because of (\ref{4eq11}),~$W_{1}(z,q)$ commutes with~$\mathrm{Diag}\left(\L_{q,\mathrm{Id}+(q-1)\widetilde{L}_{i}}\right)$. Because of \textbf{(H5)}, and the construction of~$F_{2}(z,q)$,~$W_{1}(z,q)$ commutes also with~$F_{2}(z,q)$. Let~${F_{1}(z,q):=E_{1}(z,q)W_{1}(z,q)}$.
 Then, by construction,
$$F_{1}(z)\Big[\mathrm{Id}+(q-1)N(z,q)\Big]_{\displaystyle\sq}=\mathrm{Diag}\Big(B_{i}(q)z^{-\mu_{i}}\Big),$$
and the matrices~$N(z,q)$,~$F_{1}(z,q)$ and~$F_{2}(z,q)$ have entries in the good fields. \\
\item (2)
We recall that the matrix~$U(z,q)$ satisfies the linear~$\dq$-equation:
$$\dq U(z,q)=N(z,q)U(z,q).$$
Let~$\widetilde{N}(z):=\mathrm{Diag}\left(\widetilde{L}_{i}+\d \widetilde{\l}_{i}(z)\times\mathrm{Id}_{m_{i}}\right)$ which satisfies$$\d\left(\mathrm{Diag}\left(e^{\log(z)\widetilde{L}_{i}}e^{\widetilde{\l}_{i}(z)\times\mathrm{Id}_{m_{i}}}\right)\right)=\widetilde{N}(z)\mathrm{Diag}\left(e^{\log(z)\widetilde{L}_{i}}e^{\widetilde{\l}_{i}(z)\times\mathrm{Id}_{m_{i}}}\right).$$
\par 
 From what is preceding, we deduce the following relations:
$$\begin{array}{rcl}
\d \widetilde{H}(z)&=&\widetilde{B}(z)\widetilde{H}(z)-\widetilde{H}(z)\widetilde{N}(z)\\
\sq \left(\hat{H}(z,q)F_{1}(z,q)\right)\Big(\mathrm{Id}+(q-1)N(z,q)\Big)&=&\Big(\mathrm{Id}+(q-1)B(z,q)\Big)\left(\hat{H}(z,q)F_{1}(z,q)\right). 
\end{array}$$
This implies that
$$
\dq \left( \hat{H}(z,q)F_{1}(z,q)\right)=B(z,q)\hat{H}(z,q)F_{1}(z,q)-\sq\left(\hat{H}(z,q)F_{1}(z,q)\right)N(z,q) ,$$
and finally 
\begin{equation}\label{4eq1}
\dq\left( \hat{H}(z,q)F_{1}(z,q)\right)\Big( \mathrm{Id}+(q-1)N(z,q)\Big)=\widetilde{B}(z)\hat{H}(z,q)F_{1}(z,q)
-\hat{H}(z,q)F_{1}(z,q)N(z,q).
\end{equation}\\
We are going now to prove that the entries that belong to the first row of~$\hat{H}(z,q)F_{1}(z,q)$ converge coefficientwise to the corresponding entries of~$\widetilde{H}(z)$ when~$q\to 1$.\par 
 Let~${\hat{h}(z,q):=\displaystyle\sum_{n=0}^{\infty} \hat{h}_{n}(q)z^{n}}$ be an entry of the first row of~$\hat{H}(z,q)F_{1}(z,q)$ and let~${\widetilde{h}(z):=\displaystyle \sum_{n=0}^{\infty} \widetilde{h}_{n}z^{n}}$ be the corresponding entry of~$\widetilde{H}(z)$. 
We want to use Lemma~\ref{4lem9} to prove that for all~$n\in \N$,~$\hat{h}_{n}(q)$ converges as~$q$ goes to~$1$ to~$\widetilde{h}_{n}$. We are going to prove now that the assumptions of Lemma~\ref{4lem9} are satisfied.\par
\begin{itemize}
\item The matrices~$B(z,q)$ and~$N(z,q)$ converge entrywise and coefficientwise to~$\widetilde{B}(z)$ and~$\widetilde{N}(z)$ when~$q\to 1$. Therefore, using additionally (\ref{4eq1}), we find that there exists a~$\dq$-equation with coefficient in~$\C[[z]]$ that is satisfied by~$\hat{h}(z,q)$, with~$z$-coefficients that converge to the~$z$-coefficients of a~$\d$-equation with coefficient in~$\C[[z]]$, that is satisfied by~$\widetilde{h}(z)$.
\item As we can see in Remark~\ref{4rem2} (5), the vector space of Lemma~\ref{4lem9} has dimension~$1$.
\item By construction,~$F_{1}(z,q)$ is of the form~${\mathrm{Id}+zG_{1}(z,q)}$, where ${z\mapsto G_{1}(z,q)\in \mathrm{M}_{m}\Big(\C\{z\}\Big)}$. Hence for~$q$ close to~$1$,  the entries  of the first row of~$\hat{H}(z,q)F_{1}(z,q)$ have~$z$-valuation equal to the entries  of the first row of~$\hat{H}(z,q)$, which are~$0$ (see the paragraph just below \textbf{(H4)}). Due to \textbf{(H6)}, the entries  of the first row of~$\widetilde{H}(z)$ have~$z$-valuation equal to~$0$.
\item Let us prove the convergence of~$\hat{h}_{0}(q)$  to~$\widetilde{h}_{0}$. Since~$F_{1}(z,q)$ is of the form~${\mathrm{Id}+zG_{1}(z,q)}$, it is sufficient to prove that the constant term of the entries of the first row of~$\hat{H}(z,q)$ converges to the constant term of the corresponding entry of~$\widetilde{H}(z)$. This is guaranteed by \textbf{(H6)}.
\end{itemize}
We can apply Lemma~\ref{4lem9}, which gives that the first row of~$\hat{H}(z,q)F_{1}(z,q)$ converges entrywise and coefficientwise to the first row of~$\widetilde{H}(z)$ when~$q\to 1$.\\\par 
Let us prove now the convergence of the other rows. Let~${\hat{h}(z,q)}$ be an entry of~$\hat{H}(z,q)F_{1}(z,q)$ and let~${\widetilde{h}(z)}$ be the corresponding entry of~$\widetilde{H}(z)$. Let ~${\hat{h}_{1}(z,q),\dots,\hat{h}_{m}(z,q)}$ be the entries of the first row of~$\hat{H}(z,q)F_{1}(z,q)$ and let~${\widetilde{h}_{1}(z),\dots,\widetilde{h}_{m}(z)}$ be the corresponding entries of~$\widetilde{H}(z)$. From (\ref{4eq1}), we find that there exist~$r\in \N$,~${z\mapsto \Big(d_{i,j}(z,q)\Big)_{i\leq r,j\leq m},\left(\widetilde{d}_{i,j}\right)_{i\leq r,j\leq m}\in \C((z))}$, such that:

\begin{equation}\label{4eq8}
\left\{\begin{array}{lllllllll}
\sum_{i,j}d_{i,j}(z,q)\d_{q}^{i}\Big(\hat{h}_{j}(z,q)\Big)&=&\hat{h}(z,q)\\\\
\sum_{i,j}\widetilde{d}_{i,j}(z)\d^{i}\left(\widetilde{h}_{j}(z)\right)&=&\widetilde{h}(z),
\end{array}\right.
\end{equation}
and such that for all~$i,j$,~$d_{i,j}(z,q)$ converges entrywise to ~$\widetilde{d}_{i,j}(z)$ when~$q\to 1$. The entrywise convergence of~$\hat{h}(z,q)$ to ~$\widetilde{h}(z,q)$ when~$q\to 1$ follows immediately from the case of the first row. \par 
Using (\ref{4eq8}) and the fact that for all~$q$ close to~$1$, the~$z$-valuation of the entry of the first row of~$\hat{H}(z,q)F_{1}(z,q)$ are~$0$, we obtain the existence of~$N'\in \N$, such that for all~$q$ close to~$1$,~${z\mapsto z^{N'}\hat{h}(z,q)\in\C[[z]]}$. We apply the same reasoning on the other entries of~$\hat{H}(z,q)F_{1}(z,q)$ to conclude the existence of~$N\in \N$, such that for~$q$ close to~$1$,~${z\mapsto z^{N}\hat{H}(z,q)F_{1}(z,q)\in\mathrm{M}_{m}\Big(\C[[z]]\Big)}$.
\end{trivlist}
\end{proof}
 
\pagebreak[3]
\subsection{Confluence of the Stokes matrices}\label{4sec83}
In this subsection, we combine Theorems~\ref{4theo1} and \ref{4theo2}, to prove the convergence of a basis of meromorphic solutions of a family of linear~$\dq$-equations to a basis of meromorphic solutions of the corresponding linear~$\d$-equation. We consider the family of equations 
$$\left\{\begin{array}{lllllllllll}
\D_{q}&:=&b_{m}(z,q)\d_{q}^{m}&+&b_{m-1}(z,q)\d_{q}^{m-1}&+&\dots&+&b_{0}(z,q)\\\\
\widetilde{\D}&:=&\widetilde{b}_{m}(z)\d^{m}&+&\widetilde{b}_{m-1}(z)\d^{m-1}&+&\dots&+&\widetilde{b}_{0}(z),
\end{array}\right.$$
and assume that they satisfy the assumptions \textbf{(H2)} to \textbf{(H6)} of~$\S \ref{4sec82}$ and the two following assumptions:\\
\begin{trivlist}
\item \textbf{(H1')} For all~$i\leq m$,~$z\mapsto b_{i}(z,q),\widetilde{b}_{i}(z)\in \C[z]$.
\item \textbf{(H7)}  Every entry~$\hat{h}$ of the matrix~$z^{N}\hat{H}(z,q)F_{1}(z,q)$ given by Theorem~\ref{4theo2} (resp. every entry~$\widetilde{h}$ of the matrix~$z^{N}\widetilde{H}(z)$), satisfies a~family of $\dq$-equations (resp.~$\d$-equation) that verifies the assumptions \textbf{(A2)}  and \textbf{(A3)} detailed~$\S \ref{4sec42}$. 
\end{trivlist}

As in~$\S\ref{4sec82}$, we consider the associated systems:
$$\left\{\begin{array}{lll}
\dq Y(z,q)&=&B(z,q)Y(z,q)\\\\
\d \widetilde{Y}(z)&=&\widetilde{B}(z)\widetilde{Y}(z).
\end{array}\right.$$

The next lemma gives a sufficient condition for the assumption \textbf{(H7)} to be satisfied. See Remark~\ref{4rem2} for the discussion about the cases where the other assumptions are satisfied. 
\pagebreak[3]
\begin{lem}\label{4lem2}
If the~$b_{i}(z,q)$ are independent of~$q$, and if \textbf{(H1')}, \textbf{(H2)} to \textbf{(H6)} hold, then \textbf{(H7)} is satisfied.
\end{lem}

\begin{proof}
The matrix~$z^{N}\widetilde{H}(z)$ satisfies the equation
$$\d \left(z^{N}\widetilde{H}(z)\right)=\widetilde{B}(z)z^{N}\widetilde{H}(z)-z^{N}\widetilde{H}(z)\left(\widetilde{N}(z)-N\times \mathrm{Id}\right),$$
where~${\widetilde{N}(z)=\mathrm{Diag}\left(\widetilde{L}_{i}+\d \widetilde{\l}_{i}(z)\times\mathrm{Id}_{m_{i}}\right)}$ has entries in~$\C[z^{-1}]$.
From (\ref{4eq1}), we obtain 
\begin{equation}\label{4eq9}
\begin{array}{ll}
&\dq\left( z^{N}\hat{H}(z,q)F_{1}(z,q)\right)\Big( \mathrm{Id}+(q-1)N(z,q)\Big)\\\\
=&q^{N}\widetilde{B}(z)z^{N}\hat{H}(z,q)F_{1}(z,q)
-z^{N}\hat{H}(z,q)F_{1}(z,q)\Big(q^{N}N(z,q)-[N]_{q}\times \mathrm{Id}\Big),
\end{array}\end{equation}
where~$N(z,q)$ converges to~$\widetilde{N}(z)$.
Let~$\hat{h}(z,q)$ be an entry of~$z^{N}\hat{H}(z,q)F_{1}(z,q)$  and let~$\widetilde{h}(z)$ be the corresponding entry of~$z^{N}\widetilde{H}(z)$. Using (\ref{4eq9}), we obtain the existence of~${r\in \N^{*}}$,~${z\mapsto d_{1}(z,q),\dots,d_{r}(z,q),\widetilde{d}_{1},\dots,\widetilde{d}_{r}\in \C[z]}$,~$c>0$, such that for all~$i\leq r$, for all~$q>1$ sufficiently close to~$1$,~$\left|d_{i}(z,q)-\widetilde{d}_{i}(z)\right|<(q-1)c\left(\left|\widetilde{d}_{i}(z)\right|+1\right)$,  and such that 
$$
\left\{\begin{array}{lllllllll}
\displaystyle\sum_{i\leq r}d_{i}(z,q)\d_{q}^{i}\Big(\hat{h}(z,q)\Big)&=&0\\\\
\displaystyle\sum_{i\leq r}\widetilde{d}_{i}(z)\d^{i}\left(\widetilde{h}(z)\right)&=&0.
\end{array}\right.$$
In particular,~$\hat{h}$ satisfies the assumptions \textbf{(A1)} and \textbf{(A3)}, with formal limit the formal power series~$\widetilde{h}(z)$. \par 
Moreover, the~$z$-valuations of the~$b_{i}(z,q)$ are independent of~$q$ and are equal to the~$z$-valuations of the~$\widetilde{b}_{i}(z)$. Therefore, the~$z$-valuations of the~$d_{i}(z,q)$ are independent of~$q$ and are equal to the~$z$-valuations of the~$\widetilde{d}_{i}(z)$. Since the slopes of the equation depend only on the~$z$-valuation, we obtain that~$\hat{h}$ satisfies the assumption \textbf{(A2)}, with formal limit the formal power series~$\widetilde{h}(z)$.
\end{proof}

We recall that if~$\widetilde{D}(z)\in \mathrm{M}_{m}\Big(\C(z)\Big)$, we define ~$ \mathbf{S}^{1} \left(\widetilde{D}(z)\right)$ as the union of the~$\R_{\geq 1}x_{i}$, where~$x_{i}$ are the poles of~$\widetilde{D}(z)$.
Let~$\Sigma_{\widetilde{H}}$ be the union of the~$\Sigma_{\widetilde{h}_{i,j}}$, that have been defined in~$\S \ref{4sec62}$, Step 1, where~$\widetilde{h}_{i,j}$ are the entries of~$\widetilde{H}$. Due to \textbf{(H7)}, we may apply Theorem~\ref{4theo1} to the divergent entries of 
~$z^{N}\hat{H}(z,q)F_{1}(z,q)$ and~$z^{N}\widetilde{H}(z)$. Using additionally Remark~\ref{4rem7},~$(2)$, and the reasoning in $\S \ref{4sec62}$, Step~$4$, we may prove a similar result for the convergent entries, and we find the existence of~$k\in \N^{*}$, such that for all~$d\in \R\setminus\Sigma_{\widetilde{H}}$,
$$\lim\limits_{q \to 1} \mathcal{S}_{q}^{[d]}\left(z^{N}\hat{H}F_{1}\right) 
=\widetilde{\mathcal{S}}^{d}\left(z^{N}\widetilde{H}\right),$$
uniformly on the compacts of~${\overline{S}\left(d-\frac{\pi}{2k},d+\frac{\pi}{2k}\right)\setminus   \mathbf{S}^{1}\left(\widetilde{B}(z)\right)}$. 
From Theorem~\ref{4theo1} and Theorem~\ref{4theo2}, there exists~$F_{2}(z,q)\in \mathcal{O}_{m}^{*}$, such that
$$
\begin{array}{r}
\P_{0}^{[d]}(z,q):=z^{-N}\mathcal{S}_{q}^{[d]}\left(z^{N}\hat{H}F_{1}\right) F_{2}(z,q)\mathrm{Diag}_{i}\left(\L_{q,\mathrm{Id}+(q-1)\widetilde{L}_{i}}\displaystyle \displaystyle \prod_{j=1}^{k_{i}} e_{q^{j}}\left(\widetilde{\l}_{i,j}z^{-j}\times\mathrm{Id}_{m_{i}}\right)\right)\\\\
\in \mathrm{GL}_{m}\Big(\mathcal{M}(\C^{*},0)\Big),\end{array}$$ 
 is a fundamental solution of~$\dq Y(z,q)=B(z,q)Y(z,q)$. From~$\S \ref{4sec1}$, we recall that 
$$\widetilde{\P}_{0}^{d}(z):=z^{-N}\widetilde{\mathcal{S}}^{d}\left(z^{N}\widetilde{H}(z)\right)\mathrm{Diag}\left(e^{\widetilde{L}_{i}\log(z)+\widetilde{\l}_{i}(z)\times\mathrm{Id}_{m_{i}}}\right)
\in \mathcal{A}\left(d-\frac{\pi}{2k},d+\frac{\pi}{2k}\right)$$
is a fundamental solution of~$\d \widetilde{Y}(z)=\widetilde{B}(z)\widetilde{Y}(z)$.
\pagebreak[3]
\begin{lem}\label{4lem3}
We have$$\lim\limits_{q \to 1}\P_{0}^{[d]}(z,q)=\widetilde{\P}_{0}^{d}(z),$$
uniformly on the compacts of~$\overline{S}\left(d-\frac{\pi}{2k},d+\frac{\pi}{2k}\right)\setminus   \left(\mathbf{S}^{1}\left(\widetilde{B}(z)\right)\bigcup \{\R_{<0}\}\right)$.
\end{lem} 
\begin{proof}
Due to the preceding discussion and the definition of $\mathcal{O}_{m}^{*}$, we only have to prove the convergence$$\lim\limits_{q \to 1}\mathrm{Diag}_{i}\left(\L_{q,\mathrm{Id}+(q-1)\widetilde{L}_{i}}\displaystyle \displaystyle \prod_{j=1}^{k_{i}} e_{q^{j}}\left(\widetilde{\l}_{i,j}z^{-j}\times\mathrm{Id}_{m_{i}}\right)\right)
=\mathrm{Diag}\left(e^{\widetilde{L}_{i}\log(z)+\widetilde{\l}_{i}(z)\times\mathrm{Id}_{m_{i}}}\right).$$  The fact that$$\lim\limits_{q \to 1}\mathrm{Diag}\left(\L_{q,\mathrm{Id}+(q-1)\widetilde{L}_{i}}\right)=\mathrm{Diag}\left(e^{\widetilde{L}_{i}\log(z)}\right),$$
uniformly on the compacts of a convenient domain has been proved in a more generalize case in Page 1048 of \cite{S00}. See Lemma~\ref{4lem7}, for the convergence of the~$q$-exponential part. 
\end{proof}

Let~$d^{-}<d^{+}$ with~$d^{\pm}\in \R\setminus\Sigma_{\widetilde{H}}$, so that we can define~$\P_{0}^{[d^{\pm}]}(z,q)$. We define the~$q$-Stokes matrix~${ST^{[d^{-}],[d^{+}]}(z,q)\in \mathrm{GL}_{m}\Big(\mathcal{M}_{\mathbb{E}}\Big)}$ (we recall that $\mathcal{M}_{\mathbb{E}}$ is the field of functions invariant under the action of~$\sq$, see the introduction) as follows:
$$\P_{0}^{[d^{+}]}(z,q)=\P_{0}^{[d^{-}]}(z,q)ST^{[d^{-}],[d^{+}]}(z,q).$$
Let~$d-\frac{\pi}{2k}<d^{-}<d<d^{+}<d+\frac{\pi}{2k}$ such that$$\Big(\left[d^{-},d\right[ \displaystyle\bigcup \left]d,d^{+}\right]\Big)\bigcap \Sigma_{\widetilde{H}}=\varnothing.$$ Let us recall that by construction,~$\Sigma_{\widetilde{H}}$ contains~$\widetilde{\Sigma}_{\widetilde{H}}$, the set of singular directions that has been defined in Proposition \ref{4propo5}.  Therefore, following~$\S \ref{4sec1}$, we may define the Stokes matrix in the direction~$d$,~${\widetilde{ST}^{d}\in\mathrm{GL}_{m}(\C)}$, as follows:
$$\widetilde{\P_{0}}^{d^{+}}(z)=\widetilde{\P_{0}}^{d^{-}}(z)\widetilde{ST}^{d}.$$
\pagebreak[3]
\begin{rem}
If~$d$ is not a singular direction (see Proposition~\ref{4propo5}), then although~${\widetilde{ST}^{d}=\mathrm{Id}}$, the entries of~$\widetilde{H}(z)$ might be divergent. In fact, see \cite{VdPS} Page~247, the entries of~$\widetilde{H}(z)$ are convergent if and only if~$\widetilde{ST}^{d}=\mathrm{Id}$ for all~$d\in \R$. On the other hand, the principle of analytic continuation implies that if~$ST^{[d^{-}],[d^{+}]}(z,q_{0})=  \mathrm{Id}$ for some~$d^{-}<d^{+}$ and for some~${q_{0}>1}$, then~${z\mapsto z^{N}\hat{H}(z,q_{0})F_{1}(z,q_{0})\in \mathrm{M}_{m}\Big(\C\{z\}\Big)}$.
\end{rem}
Using Lemma~\ref{4lem3}, we prove:
\pagebreak[3]
\begin{theo}\label{4theo4}
Let~$d-\frac{\pi}{2k}<d^{-}<d<d^{+}<d+\frac{\pi}{2k}$ such that$$\Big(\Big[d^{-},d\Big[ \displaystyle\bigcup \Big]d,d^{+}\Big]\Big)\bigcap \Sigma_{\widetilde{H}}=\varnothing.$$ Then, for~$q$ close to~$1$, we can define~$ST^{[d^{-}],[d^{+}]}(z,q)$ and we have$$\lim\limits_{q \to 1}ST^{[d^{-}],[d^{+}]}(z,q)=\widetilde{ST}^{d},$$
uniformly on the compacts of~$\overline{S}\left(d-\frac{\pi}{2k},d+\frac{\pi}{2k}\right)\setminus   \left(\mathbf{S}^{1}\left(\widetilde{B}(z)\right)\bigcup \{\R_{<0}\}\right)$.
\end{theo}

\pagebreak[3]
\subsection{Confluence to the monodromy}\label{4sec84}
In this subsection, we show how a basis of meromorphic solutions of a family of linear~$\dq$-equations at~$0$ and at~$\infty$ can help us to find the monodromy matrices of the corresponding differential equation. We consider the family of equations 
$$\left\{\begin{array}{lllllllll}
\D_{q}&:=&b_{m}(z,q)\d_{q}^{m}&+&b_{m-1}(z,q)\d_{q}^{m-1}&+&\dots&+&b_{0}(z,q)\\\\
\widetilde{\D}&:=&\widetilde{b}_{m}(z)\d^{m}&+&\widetilde{b}_{m-1}(z)\d^{m-1}&+&\dots&+&\widetilde{b}_{0}(z),
\end{array}\right.$$
that satisfies the assumptions \textbf{(H1')},\textbf{(H2)} to \textbf{(H7)} of~$\S \ref{4sec82}$,~$\S\ref{4sec83}$ and the following assumptions:\\

\begin{trivlist}
\item \textbf{(H8)} The zeros of~$\widetilde{b}_{m}(z)$ have different arguments and there is no zero which has an argument equal to~$\pi$. 
\item \textbf{(H9)} The assumptions \textbf{(H1')},\textbf{(H2)} to \textbf{(H8)} are satisfied with the linear~$\dq$ and~$\d$-equation at infinity, obtained by considering~$z\mapsto z^{-1}$. 
\end{trivlist}

As in~$\S\ref{4sec82}$,~$\S\ref{4sec83}$, we consider the associated systems:
$$\left\{\begin{array}{lll}
\dq Y(z,q)&=&B(z,q)Y(z,q)\\\\
\d \widetilde{Y}(z)&=&\widetilde{B}(z)\widetilde{Y}(z).
\end{array}\right.$$
\par 
Let~$d\in \R\setminus\Sigma_{\widetilde{H}}$. Due to Lemma~\ref{4lem3}, there exists~$k\in \N^{*}$ such that for we have $$\lim\limits_{q \to 1}\P_{0}^{[d]}(z,q)=\widetilde{\P}_{0}^{d}(z),$$ 
uniformly on the compacts of$$\widetilde{\Omega}_{0}:=\overline{S}\left( d-\frac{\pi}{2k},d+\frac{\pi}{2k}\right)\setminus \left(\mathbf{S}^{1}\left(\widetilde{B(z)}\right)\bigcup \R_{<0}\right).$$ 
\par 
We are now interested in the domain of definition of the fundamental solution~$\P_{0}^{[d]}(z,q)$ for~$q$ close to~$1$ fixed.  We recall that if~$D(z)\in \mathrm{GL}_{m}\Big(\C(z)\Big)$, we define~$ \mathbf{S}^{q} (D(z))$ as the union of the~$q^{\N^{*}}x_{i}$, where~$x_{i}$ is a pole of~$D(z)$ or~$D^{-1}(z)$.
Following Page~1035 in \cite{S00}, we obtain that~$\L_{q,\mathrm{Id}+(q-1)\mathrm{Diag}\left(\widetilde{L}_{i}\right)}$ has poles contained in a finite number of~$q$-discrete spiral of the form~$q^{\Z}\b_{i}(q)$, that converge to the spiral~$\R_{<0}$ as~$q$ tends to~$1$. By construction, for~$q$ fixed, the domain of definition of the matrices~$\mathcal{S}_{q}^{[d]}\left(z^{N}\hat{H}F_{1}\right)$,~$F_{2}(z,q)$ and~$\mathrm{Diag}_{i}\left(\displaystyle \displaystyle \prod_{j=1}^{k_{i}} e_{q^{j}}\left(\widetilde{\l}_{i,j}z^{-j}\times\mathrm{Id}_{m_{i}}\right)\right)$, intersected with ~$ \overline{S}\left( d-\frac{\pi}{2k},d+\frac{\pi}{2k}\right)$ is~${ \overline{S}\left( d-\frac{\pi}{2k},d+\frac{\pi}{2k}\right)\setminus \left(\mathbf{S}^{q}\left(\mathrm{Id}+(q-1)\widetilde{B}(z)\right)\right)}$. Notice that,~$\mathbf{S}^{q}\left(\mathrm{Id}+(q-1)\widetilde{B}(z)\right)$ tends to~$\mathbf{S}^{1}\left(\widetilde{B}(z)\right)$ as~$q$ goes to~$1$.
We have proved that for~$q$ fixed close to~$1$, the domain of definition of~$\P_{0}^{[d]}(z,q)$ intersected with ~$ \overline{S}\left( d-\frac{\pi}{2k},d+\frac{\pi}{2k}\right)$ is:$$ \overline{S}\left( d-\frac{\pi}{2k},d+\frac{\pi}{2k}\right)\setminus \left(\mathbf{S}^{q}\left(\mathrm{Id}+(q-1)\widetilde{B}(z)\right)\bigcup q^{\Z}\b_{i}(q)\right).$$ \par
\begin{figure}[ht]
\begin{center}
\includegraphics[width=1\linewidth]{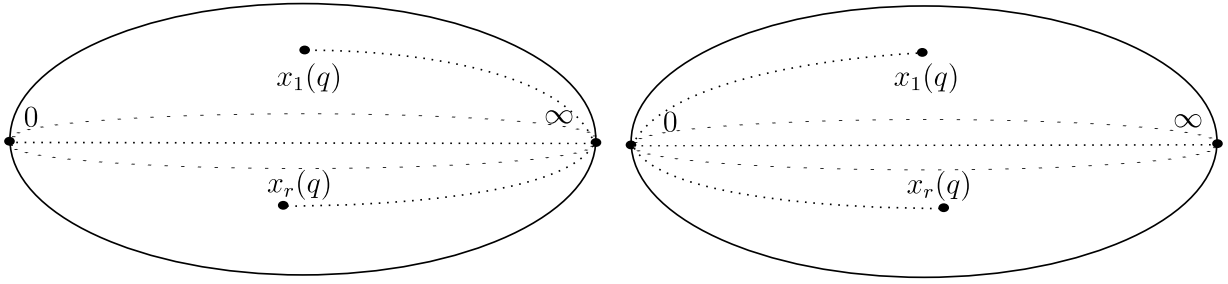}
\caption{Intersection of~$\overline{S}\left(d-\frac{\pi}{2k},d+\frac{\pi}{2k}\right)$ and the domain of definition of~$\P_{0}^{[d]}(z,q)$ (left) and~$\P_{\infty}^{[d]}(z,q)$ (right). \label{4figure2}}
\end{center}
\end{figure}

We consider now the singularity at~$\infty$ putting~$z\mapsto z^{-1}$. After taking a larger set finite modulo~$2\pi\Z$,~$\Sigma_{\widetilde{H}}\subset \R$, we may assume that for all~$d\notin \Sigma_{\widetilde{h}}$, we can also compute a fundamental solution at infinity~$\P_{\infty}^{[d]}(z,q)$ in the same way than~$\P_{0}^{[d]}(z,q)$.
Let~$p=q^{-1}$. Similarly to~$\widetilde{\Omega}_{0}$, let us define~$\widetilde{\Omega}_{\infty}$, such that 
$$
\lim\limits_{q \to 1}\P_{\infty}^{[d]}(z,q)=\widetilde{\P}_{\infty}^{d}(z),$$
uniformly on the compacts of 
~$\widetilde{\Omega}_{\infty}$, 
where~$\widetilde{\P}_{\infty}^{d}(z)$ is the fundamental solution of the linear~$\d$-system at infinity computed with Borel and Laplace transformations. More precisely, there exists~${k'\in \N^{*}}$, such that 
~${\widetilde{\Omega}_{\infty}:=\overline{S}\left(d-\frac{\pi}{2k'},d+\frac{\pi}{2k'}\right)\setminus\Big\{\R_{<0}, t\widetilde{x}_{1},\dots,t\widetilde{x}_{r}\Big|t\in ]0,1]\Big\}}$, where the~$\widetilde{x}_{i}$ satisfies 
~${\widetilde{\Omega}_{0}=\overline{S}\left(d-\frac{\pi}{2k},d+\frac{\pi}{2k}\right)\setminus\left\{\R_{<0},\R_{\geq 1}\widetilde{x}_{1},\dots,\R_{\geq 1}\widetilde{x}_{r}\right\}}$. If we restrict the domain of convergence, we may assume that~$k=k'$.\par 
The Birkhoff matrix~$\left(\P_{\infty}^{[d]}(z,q)\right)^{-1}\P_{0}^{[d]}(z,q)$ is invariant under the action of~$\sq$ and tends to$$\lim\limits_{q \to 1} \left(\P_{\infty}^{[d]}(z,q)\right)^{-1}\P_{0}^{[d]}(z,q)
=\left(\widetilde{\P}_{\infty}^{d}(z)\right)^{-1}\widetilde{\P}_{0}^{d}(z)=:\widetilde{P}^{d},$$
uniformly on the compacts of~$\widetilde{\Omega}_{\infty}\cap \widetilde{\Omega}_{0}$.\par 
Since~$\left(\P_{\infty}^{[d]}(z,q)\right)^{-1}\P_{0}^{[d]}(z,q)$ is invariant under the action of~$\sq$, we obtain that~$\widetilde{P}^{d}$ is locally constant.

\begin{figure}[ht]
\begin{center}
\includegraphics[width=1\linewidth]{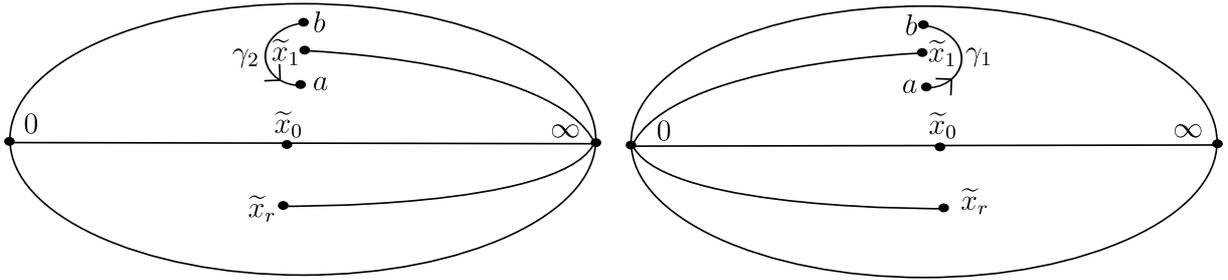}
\caption{Domain of definition of~$\widetilde{\P}_{0}^{d}(z)$ (left) and~$\widetilde{\P}_{\infty}^{d}(z)$ (right).}
\end{center}
\end{figure}

Let~$\widetilde{x}_{0}=-1$. We order the~$\widetilde{x}_{i}$  by increasing arguments in~$\left]d-\frac{\pi}{2k},d+\frac{\pi}{2k}\right[$. The connected component of the domain of definition of~$\widetilde{P}^{d}$ are the~$\widetilde{U}_{j}$, where$$\widetilde{U}_{j}:=\overline{S}\left(d-\frac{\pi}{2k},d+\frac{\pi}{2k} \right)\bigcap \overline{S}\Big(\arg(\widetilde{x}_{j}),\arg(\widetilde{x}_{j+1})\Big) .$$ 
\begin{figure}[ht]
\begin{center}
\includegraphics[width=0.5\linewidth]{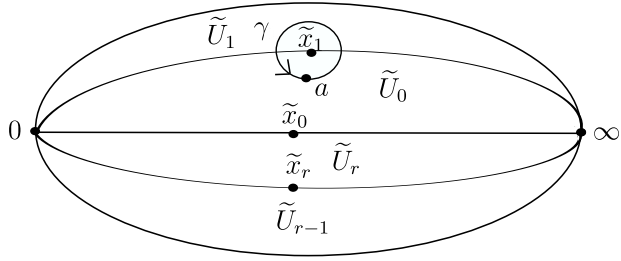}
\caption{Domain of definition of~$\widetilde{P}^{d}$. \label{4figure4}}
\end{center}
\end{figure}
Let~$\widetilde{P}_{j}^{d}\in \mathrm{GL}_{m}(\C)$ be the value of~$\widetilde{P}^{d}$ in~$\widetilde{U}_{j}$. Let us chose~$\widetilde{x}_{j}$ such that~${\widetilde{x}_{j} \in \overline{S}\left(d-\frac{\pi}{2k},d+\frac{\pi}{2k}\right)}$. Let us consider a little positive  path~$\g$ around~$\widetilde{x}_{j}$ starting from~$a\in\widetilde{U}_{j-1}$. We may choose~$\g$ such that we can decompose~$\g$ into~$\g_{1}$ and~$\g_{2}$ such that~$\g_{1}$ comes from~$a$ to~$b\in\widetilde{U}_{j}$ in~$\widetilde{\Omega}_{\infty}$ and~$\g_{2}$ comes from~$b$ to~$a$ in~$\widetilde{\Omega}_{0}$. The analytic continuation along~$\g_{1}$ transforms~$\widetilde{\P}_{0}^{d}(z)$ into~$\widetilde{\P}_{\infty}^{d}(z)\widetilde{P}_{j-1}^{d}$, and the analytic continuation along~$\g_{2}$ transforms~$\widetilde{\P}_{\infty}^{d}(z)$ into~$\widetilde{\P}_{0}^{d}(z)\left(\widetilde{P}_{j}^{d}\right)^{-1}$. We have proved the following theorem, which extends when~$q$ is real, the theorem of the~$\S 4$ in \cite{S00} in the non-Fuchsian case:
\pagebreak[3]
\begin{theo}\label{4theo3}
The monodromy matrix of the~$\d$-equation~$\d \widetilde{Y}(z)=\widetilde{B}(z)\widetilde{Y}(z)$ in the basis~$\widetilde{\P}_{0}^{d}(z)$ around the singularity~$\widetilde{x}_{j}$ is~$\left(\widetilde{P}_{j}^{d}\right)^{-1}\widetilde{P}_{j-1}^{d}$.
\end{theo}

\pagebreak[3]
\appendix
\section{Confluence of a ``continuous''~$q$-Borel-Laplace summation.}

The goal of this appendix is to prove the equivalent of Theorem~\ref{4theo1} for a ``continuous''~$q$-Borel-Laplace summation. We introduce now the ``continuous''~$q$-Laplace transformation. See~$\S\ref{4sec3}$ for the notations.
\pagebreak[3]
\begin{defi}\label{4defi2} Let~$k\in \Q_{>0}$ and let~$d\in \R$. As we can see in \cite{DVZ},~$\S 4.2$, the following maps are defined and we call them the ``continuous''~$q$-Laplace transformation of order~$1$ and~$k$:

$$
\begin{array}{llll}
\mathcal{L}_{q,1}^{d}:& \mathbb{H}_{q,1}^{d}&\longrightarrow&\mathcal{A}(d-\pi,d+\pi)\\
&f&\longmapsto&\frac{q-1}{\log(q)}\displaystyle \int_{0}^{\infty e^{id}}\frac{f(\z)}{ze_{q}\left(\frac{q\z}{z} \right)}d\z,\\\\
\mathcal{L}_{q,k}^{d}:& \mathbb{H}_{q,k}^{d}&\longrightarrow&\displaystyle\bigcup_{\nu=0}^{k-1}\mathcal{A}\left(\frac{2\pi\nu(d-\pi)}{k},\frac{2\pi(\nu+1)(d-\pi)}{k}\right)\\
&g&\longmapsto&\r_{k}\circ \mathcal{L}^{[d]}_{q,1}\circ \r_{1/k}(g).
\end{array}
$$
\end{defi} 

\pagebreak[3]
\begin{rem}
We say that the $q$-Laplace transformation is ``continuous'' because it is defined with a ``continuous'' integral, in opposition to the $q$-Laplace transformation of $\S \ref{4sec3}$, which involves a ``discrete'' Jackson integral. Notice that the term ``continuous''~$q$-Borel-Laplace summation is an abuse of language since the $q$-Borel transformation we use in this summation process is the same as in the ``discrete''~$q$-Borel-Laplace summation.
\end{rem} 
Theorem 4.14 of \cite{DVZ} compares the ``discrete'' and the ``continuous''~$q$-Borel-Laplace summation for the case of formal power series solutions of a linear~$\sq$-equation with coefficients in~$\C(\{z\})$ with only slope~$1$.
The next proposition is the analogue of Proposition~\ref{4propo2} of the present paper. 
\pagebreak[3]
\begin{propo}
Let~$g\in  \mathbb{H}_{q,1}^{d}$. Then 
\begin{itemize}
\item~$ \mathcal{L}_{q,1}^{d}\left(\dq g\right)= \dq\mathcal{L}_{q,1}^{d}\left( g\right)$. \\
\item~$z\mathcal{L}_{q,1}^{d}\left( \dq g\right)= p\mathcal{L}_{q,1}^{d}(\z g)-pz\mathcal{L}_{q,1}^{d}(g)$.
\end{itemize}
\end{propo}

\begin{proof}
To prove the first equality, it is sufficient to prove that the ``continuous''~$q$-Laplace transformation commutes with~$\sq$. To do this, we just have to perform the variable change~$\z\mapsto q\z$ in the integral.\par 
 Let us prove the last equality. We recall that~$\sq\left(e_{q}\left(\frac{q\z }{z}\right)\right)=\frac{e_{q}\left(\frac{q \z}{z}\right)}{1+(q-1)\z/z}$. Let~$p=1/q$. Then, 
$$\begin{array}{llll}
z\mathcal{L}_{q,1}^{d}(\dq g)&=& z\displaystyle \int_{0}^{\infty e^{id}}\dfrac{g(\z)}{ze_{q}\left(\frac{q\z}{z}\right)}\frac{p-1+\frac{(q-1)\z}{qz}}{q-1}d\z \\
&=& \displaystyle \int_{0}^{\infty e^{id}}\dfrac{g(\z)}{e_{q}\left(\frac{q \z}{z}\right)}(-p+p\z/z)\\
&=&p\mathcal{L}_{q,1}^{d}(\z g)-pz\mathcal{L}_{q,1}^{d}(g).
\end{array}$$
\end{proof}

Let~$k\in \N^{*}$. If we consider~$\hat{f}\in \C\left[\left[z^{k}\right]\right]$, solution of a linear~$\dq$-equation with coefficients in~$\C\left[z^{k}\right]$, with~${\hat{\mathcal{B}}_{q,k}\left(\widetilde{f}\right)\in \mathbb{H}_{q,k}^{d}}$, then we have:
$$\dq\left(\mathcal{L}_{q,k}^{d}\circ\hat{\mathcal{B}}_{q,k}\left( \hat{f}\right)\right)=\mathcal{L}_{q,k}^{d}\circ\hat{\mathcal{B}}_{q,k}\left(\dq \hat{f}\right)
\hbox{ and }
\dq \left(z^{k}\mathcal{L}_{q,k}^{d}\circ\hat{\mathcal{B}}_{q,k}\left( \hat{f}\right)\right)=\mathcal{L}_{q,k}^{d}\circ\hat{\mathcal{B}}_{q,k}\left(\dq \left(z^{k}\hat{f}\right)\right) .
$$
Hence,~$\mathcal{L}_{q,k}^{d}\circ\hat{\mathcal{B}}_{q,k}\left(\hat{f}\right)$ is solution of the same linear~$\dq$-equation than~$\hat{f}$.  But in general, if~$\hat{f}\in \C\left[\left[z\right]\right]$ is solution of a linear~$\dq$-equation with coefficients in~$\C\left[z\right]$, we will have to apply successively several~$q$-Borel and ``continuous''~$q$-Laplace transformations in order to compute an analytic solution of the same equation than~$\hat{f}$. See Theorem~\ref{4theo9}.\\\par 
As in~$\S \ref{4sec42}$, let~$z\mapsto\hat{h}(z,q)\in \C[[z]]$  that converges coefficientwise to~$\widetilde{h}(z)\in \C[[z]]$ when~$q\to 1$. We make the following assumptions:
\begin{itemize}
\item There exists
$$z\mapsto b_{0}(z,q),\dots,b_{m}(z,q)\in\C[z],$$ with~$z$-coefficients that converge as~$q$ goes to~$1$, such that for all~$q$ close to~$1$,~$\hat{h}(z,q)$ is solution of:
$$
b_{m}(z,q)\d_{q}^{m}\left(\hat{h}(z,q)\right)+\dots+b_{0}(z,q)\hat{h}(z,q)=0.
$$
Let~$\widetilde{b}_{0}(z),\dots,\widetilde{b}_{m}(z)\in\C[z]$ be the limit as~$q$ tends to~$1$ of the~$b_{0}(z,q),\dots,b_{m}(z,q)$. Notice that the series~$\widetilde{h}(z)$ is solution of:
$$
\widetilde{b}_{m}(z)\d^{m}\left(\widetilde{h}(z)\right)+\dots+\widetilde{b}_{0}(z)\widetilde{h}(z)=0.
$$
\item For~$q$ close to~$1$, the  slopes of the linear $q$-difference equation satisfied by $\hat{h}$ are independent of~$q$, and the set of slopes of the latter that are positive coincides with the set of slopes of the linear differential equation satisfied by $\widetilde{h}$.
\item There exists~$c_{1}>0$, such that for all~$i\leq m$ and~$q$ close to~$1$:$$\left|b_{i}(z,q)-\widetilde{b}_{i}(z)\right|<(q-1)c_{1}\left(\left|\widetilde{b}_{i}(z)\right|+1\right).$$ 
\item The differential equation has at least one slope strictly bigger than~$0$.
\end{itemize}
Let $d_{0}:=\max\left(2,\deg \left(\widetilde{b}_{0}\right),\dots, \deg\left( \widetilde{b}_{m}\right)\right)$.
Let~$k_{1}<\dots<k_{r-1}$ be the slopes of (\ref{4eq4}) different from $0$, let~$k_{r}$ be an integer strictly bigger than~$k_{r-1}$ and~$d_{0}$, and set~$k_{r+1}:=+\infty$. 
Let~$(\k_{1},\dots,\k_{r})$ defined as:
$$\k_{i}^{-1}:=k_{i}^{-1}-k_{i+1}^{-1}.$$  
 As in Proposition~\ref{4propo5}, we define the~$(\widetilde{\k_{1}},\dots,\widetilde{\k_{s}})$ as follows:
we take~$(\k_{1},\dots,\k_{r})$ and for~$i=1,...,i=r$, we replace successively~$\k_{i}$ by~$\a_{i}$ terms~$\a_{i}\k_{i}$, where~$\a_{i}$ is the smallest integer such that~$\a_{i}\k_{i}$ is greater or equal than~$d_{0}$. See Example~\ref{4ex1}. Therefore, by construction, each of the~$\widetilde{\k_{i}}$ are rational number greater than~$d_{0}$.
 Let~$\b\in \N^{*}$ be minimal, such that for all~$i\in \{1,\dots,s\}$,~$\b/\widetilde{\k_{i}}\in \N^{*}$. 
Let us write~$\hat{h}(z,q)=:\displaystyle \sum_{n=0}^{\infty}\hat{h}_{n}(q)z^{n}$ and for~$l\in \{0,\dots,\b-1\}$, let~$\hat{h}^{(l)}(z,q):=\displaystyle \sum_{n=0}^{\infty} \hat{h}_{l+n\b}(q)z^{n\b}$. \par 
\pagebreak[3]
\begin{theo}\label{4theo9}
There exists~$\Sigma_{\widetilde{h}}\subset \R$ finite modulo~$2\pi \Z$, such that if~$d\in  \R\setminus \Sigma_{\widetilde{h}}$ then for all~${l\in\{0,\dots,\b-1\}}$, 
the series~$g_{1,l}:=\hat{\mathcal{B}}_{q,\widetilde{\k_{1}}}\circ\dots\circ\hat{\mathcal{B}}_{q,\widetilde{\k_{s}}}\left(\hat{h}^{(l)}\right)$  converges and belongs to~$\overline{\mathbb{H}}_{\widetilde{\k_{1}}}^{d}$ (see Definition~\ref{4defi1}).\par 
Moreover, for~$j=2$ (resp.~$j=3$,~$\dots$, resp.~$j=r$), ~$g_{j,l}:=\mathcal{L}_{q,\widetilde{\k_{j-1}}}^{d}(g_{j-1,l})$ belongs to~$\overline{\mathbb{H}}_{\widetilde{\k_{j}}}^{d}$.
 Let~${S_{q}^{d}\left(\hat{h}^{(l)}\right):=\mathcal{L}_{q,\widetilde{\k_{s}}}^{d}(g_{r,l})}$. 
The function$$S_{q}^{d}\left(\hat{h}\right):=\displaystyle \sum_{l=0}^{\b-1} z^{l}S_{q}^{d}\left(\hat{h}^{(l)}\right)
\in\mathcal{A}\left(d-\frac{\pi}{k_{r}},d+\frac{\pi}{k_{r}}\right),$$ is solution of (\ref{4eq3}). 
Furthermore, we have$$\lim\limits_{q \to 1}S_{q}^{d}\left(\hat{h}\right)=\widetilde{S}^{d}\left(\widetilde{h}\right),$$
uniformly on the compacts of~$\overline{S}\left(d-\frac{\pi}{2k_{r}},d+\frac{\pi}{2k_{r}}\right)\setminus \displaystyle\bigcup \R_{\geq 1}\a_{i}$, where~$\a_{i}$ are the roots of~${\widetilde{b}_{m}\in \C[z]}$ and~$\widetilde{S}^{d}\left(\widetilde{h}\right)$ is the asymptotic solution of the same linear~$\d$-equation than~$\widetilde{h}$ that has been defined in Proposition~\ref{4propo5}.
\end{theo}

The proof of this theorem is basically the same as the proof of Theorem~\ref{4theo1}. The only difference is that we can not use Lemma~\ref{4lem5}, so we state and prove a similar result for the ``continuous'' summation. \par 
Let~$d\in \R$, let~$k\in \Q_{>0}$ and let~$f$ be a function that belongs to~$\overline{\mathbb{H}}_{k}^{d}$. By definition (see Definition~\ref{4defi1}), there exist~$\e>0$, constants~$J,L>0$, such that for all~$q$ close to~$1$,~$\z\mapsto f(\z,q)$ is analytic on~$\overline{S}(d-\e,d+\e)$, and for all~$\z\in \R_{>0}$:
$$\left|f(e^{id}\z,q)\right|<Je_{q}\left(L\z^{k}\right).$$

\pagebreak[3]
\begin{lem}
In the notations introduced above, let us assume that~$\lim\limits_{q \to 1}f:=\widetilde{f}\in \widetilde{\mathbb{H}}_{k}^{d}$
uniformly on the compacts of~$\overline{S}(d-\e,d+\e)$. 
Then, we have$$\lim\limits_{q \to 1}\mathcal{L}_{q,k}^{d}\Big(f\Big)(z)=\mathcal{L}_{k}^{d}\left(\widetilde{f}\right)(z),$$
uniformly on the compacts of ~${\Big\{z\in \overline{S}\left(d-\frac{\pi}{2k\pi},d+\frac{\pi}{2k\pi}\right)\Big||z|<1/L\Big\}}$.
\end{lem}

\begin{proof}
For the same reasons as in the proof of Lemma~\ref{4lem5}, we may assume that~$d=0$ and~$k=1$.\par 
 Let us fix a compact~$K$ of ~${\left\{z\in \overline{S}\left(-\frac{\pi}{2\pi},+\frac{\pi}{2\pi}\right)\Big||z|<1/L\right\}}$.  
Using the dominated convergence theorem, it is sufficient to prove the existence of a positive integrable function~$h$, such that for all~$q$ close to~$1$,~${\z\in \R_{>0}}$ and~$z\in K$,~$\left|\frac{f(\z,q)}{ze_{q}\left(\frac{q\z}{z}\right)}\right|<h(\z)$.\par  
 Let~$J>0$ be the constant that comes from Definition~\ref{4defi1} and let~$z\in K$.
By definition of the ``continuous''~$q$-Laplace transformation
$$\left|\mathcal{L}_{q,k}^{d}(f)(z)\right|\leq \displaystyle \int_{0}^{\infty} \left|\frac{J}{z}\frac{e_{q}(L\z)}{e_{q}(q\z/z)}\right|d\z
.$$
Let us fix~$q_{0}>1$. Let~$S\in \R$, such that for all~$z\in K$,~$q>1$ and~$\z>S$,~$\z\mapsto \left|\frac{J}{z}\frac{e_{q}(L\z)}{e_{q}(q\z/z)}\right|$ is decreasing. The convergence~${\lim\limits_{q \to 1}\displaystyle\int_{0}^{S}\frac{f(\z,q)}{ze_{q}\left(\frac{q\z}{z} \right)}d\z=\int_{0}^{S}\frac{f(\z)}{z\exp\left(\frac{q\z}{z} \right)}d\z}$ is clear. 
Moreover, we have for all~$q\in]1,q_{0}[$ and~$z\in K$:
$$\int_{S}^{\infty} \left|\frac{J}{z}\frac{e_{q}(L\z)}{e_{q}(q\z/z)}\right|d\z\leq (q-1)\displaystyle \sum_{l=0}^{\infty} \left|\frac{q^{l}SJ}{z}\frac{e_{q}\left(Lq^{l}S\right)}{e_{q}\left(q^{l+1}S/z\right)}\right|.$$
We have seen in the proof of Lemma~\ref{4lem5}, than we can bound this latter quantity uniformly in~$q$ and~$z\in K$. This yields the result.
\end{proof}
\fussy

\nocite{AAR,Ad,And,Ca,DSK,DR,DVH,DV02,DV09,DVRSZ,Ro06,Ro08,Ro11,Trj}
\pagebreak[3]
\bibliographystyle{alpha}
\bibliography{C:/Users/thomas.dreyfus/Dropbox/Maths/biblio}
\end{document}